\documentclass{amsart}
\usepackage[nohug]{diagrams}
\usepackage{amscd}
\usepackage[mathscr]{eucal}
\usepackage{amsfonts}
\usepackage{amsmath}
\usepackage{amsthm}
\usepackage{amssymb}
\usepackage{latexsym}
\usepackage{tabularx,array}
\usepackage[all]{xy}
\newfont{\msam}{msam10}

\newtheorem{theorem}[]{Theorem}
\newtheorem{proposition}[]{Proposition}
\newtheorem{corollary}[]{Corollary}
\newtheorem{lemma}[]{Lemma}
\newtheorem{prop}[theorem]{Proposition}

\theoremstyle{definition}
\newtheorem{definition}[]{Definition}
\newtheorem{remark}[]{Remark}

\let\nc\newcommand

\def\bthm{\begin{theorem}}
\def\ethm{\end{theorem}}
\def\blemma{\begin{lemma}}
\def\elemma{\end{lemma}}
\def\bproof{\begin{proof}}
\def\eproof{\end{proof}}
\def\bprop{\begin{proposition}}
\def\eprop{\end{proposition}}
\def\bcor{\begin{corollary}}
\def\ecor{\end{corollary}}
\nc{\la}{\label}
\def\Z{\mathbb{Z}}
\def\N{\mathbb{N}}

\def\c{\mathbb{C}}
\def\M{\mathbb{M}}

\def\L {\boldsymbol{L}}

\def\Com{\mathtt{Com}}

\def\Alg{\mathtt{Alg}}

\def\DGC{\mathtt{DGC}}
\def\Tw{\mathtt{Tw}}
\def\Mod{\mathtt{Mod}}
\def\cMod{\mathtt{CoMod}}

\def\Sets{\mathtt{Sets}}
\def\DGA{\mathtt{DGA}}

\def\cDGA{\mathtt{CDGA}}
\def\bDGA{\mathtt{BiDGA}}
\def\bcDGA{\mathtt{BiCDGA}}

\def\D{{\mathscr D}}
\def\C{{\mathscr C}}

\def\Ho{{\mathtt{Ho}}}
\nc{\Ob}{{\rm Ob}}
\nc{\Hom}{{\rm{Hom}}}
\nc{\Homcont}{{\mathcal{H}om}}
\nc{\HOM}{\underline{\rm{Hom}}}
\nc{\DER}{\underline{\rm{Der}}}
\nc{\END}{\underline{\rm{End}}}
\nc{\bSym}{\mathbf{\Lambda}}
\nc{\Ext}{{\rm{Ext}}}
\nc{\Rep}{{\rm{Rep}}}
\nc{\DRep}{{\rm{DRep}}}
\nc{\NCRep}{\widetilde{\rm{Rep}}}
\nc{\RAct}{{\rm{RAct}}}
\nc{\bs}{\backslash}
\nc{\ob}{{\tt{Obs}}}
\nc{\CE}{{\mathcal C}}
\nc{\nn}{{{\natural} {\natural}}}
\nc{\n}{{{\natural}}}
\nc{\A}{\mathbb A}
\nc{\B}{{\mathrm{B}}}
\nc{\Ba}{\overline{\mathrm{B}}}
\nc{\bC}{\overline{C}}
\nc{\bOmega}{\boldsymbol{\Omega}}
\nc{\bB}{\boldsymbol{B}}
\nc{\EXT}{\underline{\rm{Ext}}}
\nc{\TOR}{\underline{\rm{Tor}}}
\def\H{\mathrm H}
\def\HC{\mathrm{HC}}

\def\rHC{\overline{\mathrm{HC}}}

\def\rCC{\overline{\mathrm{CC}}}

\nc{\End}{{\rm{End}}}
\nc{\GL}{{\rm{GL}}}
\nc{\gl}{{\mathfrak{gl}}}
\nc{\rgl}{\overline{{\mathfrak{gl}}}}
\nc{\g}{{\mathfrak{g}}}
\nc{\h}{{\mathfrak{h}}}
\nc{\PGL}{{\rm{PGL}}}
\nc{\SL}{{\rm{SL}}}
\nc{\sll}{\mathfrak{sl}}
\nc{\cn}{ \mbox{\rm c\^{o}ne} }
\nc{\PSL}{{\rm{PSL}}}
\nc{\ad}{{\rm{ad}}}
\nc{\Ad}{{\rm{Ad}}}
\nc{\dlim}{\varinjlim}
\nc{\plim}{\varprojlim}
\nc{\colim}{{\tt{colim}}}
\nc{\hocolim}{{\tt{hocolim}}}

\newcommand{\bL}{\boldsymbol{\Lambda}}

\newcommand{\Spec}{{\rm{Spec}}}

\newcommand{\Sym}{{\rm{Sym}}}

\newcommand{\id}{{\rm{Id}}}

\newcommand{\Tr}{{\rm{Tr}}}

\newcommand{\Ker}{{\rm{Ker}}}

\newcommand{\into}{\,\hookrightarrow\,}

\newcommand{\onto}{\,\twoheadrightarrow\,}
\newcommand{\sonto}{\,\stackrel{\sim}{\twoheadrightarrow}\,}
\def\bs{\backslash}

\newcommand{\rar}{\xrightarrow{}}

\newcommand{\Cyl}{\mathtt{Cyl}}

\nc{\FT}{\mathcal{C}}

\numberwithin{equation}{section}
\numberwithin{theorem}{section}
\numberwithin{lemma}{section}
\numberwithin{proposition}{section}
\numberwithin{corollary}{section}
\numberwithin{example}{section}
\numberwithin{remark}{section}

\begin{document}
 \title{Stable Representation Homology and Koszul Duality}
\author{Yuri Berest}
\address{Department of Mathematics,
 Cornell University, Ithaca, NY 14853-4201, USA}
\email{berest@math.cornell.edu}
\author{Ajay Ramadoss}
\address{Departement Mathematik,
ETH Z\"urich, 8092 Z\"urich, Switzerland}
\email{ajay.ramadoss@math.ethz.ch}
%
%

\begin{abstract}
This paper is a sequel to \cite{BKR}, where we study the derived affine scheme $\DRep_n(A)$ 
parametrizing the $n$-dimensional representations of an associative $k$-algebra $A$. In \cite{BKR}, we have constructed canonical trace maps 
$ \Tr_n(A)_{\bullet}\,:\, \HC_{\bullet}(A) \rar \H_{\bullet}[\DRep_n(A)]^{\GL_n} $ extending the usual characters of representations to higher cyclic homology. This raises the natural question whether a well-known theorem of Procesi \cite{P} holds in the derived setting: namely, is the algebra homomorphism $ \bSym\mathrm{Tr}_n(A)_{\bullet}\,:\,\bSym_k[\HC_{\bullet}(A)] \rar \H_{\bullet}[\DRep_n(A)]^{\GL_n}$ defined by $ \Tr_n(A)_{\bullet} $ surjective? In the present paper, we answer this question for augmented algebras.
Given such an algebra, we construct a canonical dense subalgebra $\DRep_{\infty}(A)^{\mathrm{Tr}}$ of the topological DG algebra $ \varprojlim\,\DRep_n(A)^{\GL_n} $. Our main result is that on passing to the inverse limit, the family of maps $\bSym\mathrm{Tr}_n(A)_{\bullet}$ `stabilizes' to an isomorphism  $\bSym_k(\rHC_{\bullet}(A)) \,\cong\, \H_{\bullet}[\DRep_{\infty}(A)^{\mathrm{Tr}}]$. The derived version of Procesi's Theorem does therefore hold in the limit as $ n \to \infty $. However, for a fixed (finite) $n$, there exist homological obstructions to the surjectivity of $ \bSym\mathrm{Tr}_n(A)_{\bullet}$, and we show on simple examples that these obstructions do not vanish in general. We compare our result with the classical theorem of Loday, Quillen and
Tsygan on stable homology of matrix Lie algebras. We show that the Chevalley-Eilenberg complex $\CE_\bullet(\gl_\infty(A), \gl_\infty(k); k)$ equipped with a natural coalgebra structure is Koszul dual to the DG algebra $\DRep_{\infty}(A)^{\mathrm{Tr}}$. We also extend our main results to bigraded DG algebras, in which case we show that
$ \DRep_{\infty}(A)^{\mathrm{Tr}} = \DRep_{\infty}(A)^{\GL_{\infty}}$.
As an application, we compute the Euler characteristics of 
$\DRep_{\infty}(A)^{\GL_{\infty}}$ and $\rHC_{\bullet}(A)$ and derive
some interesting combinatorial identities.
\end{abstract}

\maketitle

\section{Introduction} Let $ k $ be a field of characteristic zero.
If $A$ is a semi-simple Artinian algebra over $k$ (e.g., the group algebra of a finite group), every matrix representation $ \varrho: A \to \M_n(k) $ of $A$ is determined, up to isomorphism, by its character $\,a \mapsto \Tr[\varrho(a)] \,$, and for each $ n \ge 0 $, there are only finitely many isomorphism classes of such representations. The classical character theory generalizes to arbitrary finitely generated algebras in a geometric way. The set of all $n$-dimensional representations of an associative $k$-algebra $A$ can be naturally given the structure of an affine $k$-scheme called the {\it representation scheme} $\, \Rep_n(A) \,$ (we write $ A_n = k[\Rep_n(A)] $ for the corresponding commutative algebra).
The isomorphism classes of $n$-dimensional representations are parametrized by the orbits of the general linear group $ \GL_n(k) $ which acts algebraically on $ \Rep_n(A) $ by conjugation. The classes of semi-simple representations correspond to the closed orbits and are parametrized by the affine quotient scheme $\,\Rep_n(A)/\!/\GL_n(k) := \Spec\,A_n^{\GL_n} $ (see, e.g., \cite{Kr}). Now, the characters of representations define a linear map
\begin{equation}
\la{trr}
\Tr_n(A):\ \HC_0(A) \to A_n^{\GL_n}\ ,
\end{equation}
where $ \HC_0(A) = A/[A,A] $ is the $0$-th cyclic homology (abelianization) of $A$.
A well-known theorem of Procesi \cite{P} asserts that the characters of $A$ actually
generate $ A_n^{\GL_n} $ as an algebra; in other words, the algebra homomorphism
\begin{equation}
\la{trra}
\Sym\, \Tr_n(A):\, \Sym_k [\HC_0(A)] \to A_n^{\GL_n}
\end{equation}
defined by \eqref{trr} is {\it surjective}. This result is a consequence of
the First Fundamental Theorem of invariant theory (see, e.g., \cite{KP}), and it
plays a fundamental role in representation theory of algebras.

In our earlier paper, \cite{BKR}, we constructed a derived version of the
representation scheme $ \Rep_n(A) $ by extending the functor $ \Rep_n $ to
the category of differential graded (DG) algebras and deriving it in the sense of non-abelian
homological algebra \cite{Q1}\footnote{The first construction of this kind was
proposed by Ciocan-Fontanine and Kapranov in \cite{CK}.
The relation of our construction to that of \cite{CK} is explained in \cite[Section~2.3.6]{BKR}.}. The corresponding derived scheme
$ \DRep_n(A) $ is represented (in the homotopy category of DG algebras) by a commutative DG algebra,
which (abusing notation) we also denote $ \DRep_n(A) $.
The homology of $ \DRep_n(A) $ depends only on $\,A\,$ and $\,n\,$,
with $ \H_0[\DRep_n(A)] $ being canonically isomorphic to $ A_n $. Following \cite{BKR}, we call
$ \H_\bullet[\DRep_n(A)] $ the {\it $n$-th representation homology} of $A$ and denote it by
$ \H_\bullet(A, n) $. The action of $ \GL_n $ on $ A_n $ extends naturally to
$ \DRep_n(A) $, and there is an isomorphism of graded algebras
$\,
\H_\bullet[\DRep_n(A)^{\GL_n}] \cong \H_\bullet(A, n)^{\GL_n} $.

Now, one of the key results of \cite{BKR} is the construction of canonical trace maps
\begin{equation}
\la{trr1}
\Tr_n(A)_\bullet:\, \HC_\bullet(A) \to \H_\bullet(A,n)^{\GL_n}\ ,
\end{equation}
extending \eqref{trr} to the higher cyclic homology. Assembled together,
these maps define a homomorphism of graded commutative algebras
\begin{equation}
\la{trra1}
\bL \Tr_n(A)_\bullet:\,
\bL_k [\HC_\bullet(A)] \to \H_\bullet(A,n)^{\GL_n}\ ,
\end{equation}
where $ \bL_k $ denotes the graded symmetric algebra over $k$. Then, given the Procesi Theorem,
it is natural to ask ({\it cf.} \cite[(1.6)]{BKR}):
\begin{equation}
\la{quest}
\textit{Is the map \eqref{trra1} surjective}\,?
\end{equation}

In the present paper, we will study this question for augmented algebras ({\it i.e.}, associative DG algebras equipped with a homomorphism $ A \to k $). The advantage of working with augmented algebras is that there are natural $ \GL$-equivariant maps $ \DRep_{n+1}(A) \to \DRep_{n}(A) $ which form an inverse system and allow one to stabilize the family $ \{\DRep_{n}(A)^{\GL_n}\} $; 
in this way, one can simplify the problem by passing to the infinite-dimensional limit $ n \to \infty $. More precisely, our approach consists of three steps. 

First, using the stabilization maps, we take the inverse limit $\, \varprojlim\,\DRep_n(A)^{\GL_n} \cong \DRep_\infty(A)^{\GL_\infty} $ and construct a canonical DG subalgebra $\, \DRep_\infty(A)^\Tr \,$ in $\, \DRep_\infty(A)^{\GL_\infty} $, which is dense in an appropriate inverse limit topology.
We call $\, \DRep_\infty(A)^\Tr $
the {\it trace subalgebra} and refer to its homology $\,\H_\bullet(A, \infty)^\Tr $
as the {\it stable representation homology} of $A$. It turns out that there is a canonical
coalgebra structure on $\, \DRep_\infty(A)^\Tr $ that makes it a commutative cocommutative
DG Hopf algebra; thus $\,\H_\bullet(A, \infty)^\Tr $ is a graded Hopf algebra.

Second, we stabilize the family of trace maps \eqref{trra1} and prove that they induce 
an isomorphism of graded Hopf algebras
\begin{equation}
\la{trra2}
{\bL}[\rHC_\bullet(A)] \stackrel{\sim}{\to} \H_\bullet(A, \infty)^{\Tr}\ ,
\end{equation}
where $ \rHC_\bullet(A) $ is the reduced cyclic homology of $A$.
This result is a consequence of Theorem~\ref{tS2.3}, which we call
`the stabilization theorem' and which is technically the main theorem of the paper.

Third, for a fixed $n$, we construct a complex $ K_\bullet(A,n) $
whose homology obstructs $ \H_{\bullet}(A, n)^{\GL_n} $ from attaining its
`stable limit' $ \H_\bullet(A, \infty)^{\Tr} $; thus, we get homological obstructions to the surjectivity of \eqref{trra1}. It is
worth noting that for an ordinary algebra ({\it i.e.}, a DG algebra concentrated
in homological degree $0$), the complex $ K_\bullet(A,n) $ is
acyclic in negative degrees, so {\it in degree zero} there are no obstructions:
whence the surjectivity of \eqref{trra}.
In general, however, the answer to \eqref{quest} turns out to be negative;
in Section~\ref{ex11}, we give simple examples
showing that \eqref{trra1} need not be surjective even for $ n = 1 $.

Apart from question \eqref{quest}, we will give two other applications
of the stabilization theorem. First, we clarify the relation between
representation homology and Lie algebra homology. As was already observed in
\cite{BKR}, the isomorphism \eqref{trra2} is analogous to a well-known
Loday-Quillen-Tsygan (LQT) isomorphism describing the stable homology of matrix Lie algebras
in terms of cyclic homology (see \cite{LQ, T}). In its relative form,
the Loday-Quillen-Tsygan Theorem asserts
\begin{equation}
\la{trra3}
\H_\bullet(\gl_\infty(A), \gl_{\infty}(k); k) \cong \bL [\rHC_{\bullet-1}(A)]\ ,
\end{equation}
where $ \gl_\infty(A) := \varinjlim\,\gl_r(A) $ is the Lie algebra of
all finite matrices over $A$ and $ \gl_\infty(k) $ is its subalgebra
consisting of matrices with entries in $k$.
Now, recall that the standard Chevalley-Eilenberg complex $\, \CE_\bullet(\g, \h; k) \,$
computing the Lie homology $\, \H_\bullet(\g, \h; k) \,$
has a natural DG coalgebra structure. Using the results of \cite{LQ} and \cite{BKR},
we will construct, for any $ r $ and $n$, a canonical degree $ -1 $ map
\begin{equation}
\la{itmn}
\tau_{r,n}(A):\, \CE_\bullet(\gl_r(A), \gl_{r}(k); k) \to \DRep_n(A)^{\GL_n}
\end{equation}
which is a twisting cochain ({\it i.e.}, a solution of a Maurer-Cartan equation)
with respect to the DG coalgebra structure on $ \CE_\bullet(\gl_r(A), \gl_{r}(k); k) $. 
Stabilizing \eqref{itmn} on both sides (as $\,r,n \to \infty$), we then get a map
$$
\tau_{\infty, \infty}(A): \CE_\bullet(\gl_\infty(A), \gl_\infty(k); k) \to
\DRep_\infty(A)^\Tr
$$
relating the Chevalley-Eilenberg complex of $ \gl_\infty(A) $
to the trace subalgebra of $A$. Our stabilization theorem together with LQT implies
that $ \tau_{\infty, \infty}(A) $ is actually an {\it acyclic} twisting cochain.
This means that the DG coalgebra
$ \CE_\bullet(\gl_\infty(A), \gl_\infty(k); k) $ is {\it Koszul dual} to the
DG algebra $ \DRep_\infty(A)^\Tr $; in particular, the latter is determined by
the former up to quasi-isomorphism. As a standard consequence of Koszul duality, we have an isomorphism
$$
\H_\bullet(A, \infty)^\Tr \cong \Ext^{-\bullet}_{\CE}(k, k)\ ,
$$
where $ \CE := \CE_\bullet(\gl_\infty(A); \gl_\infty(k), k) $ and the Ext-algebra
on the right is the Yoneda algebra defined in an appropriate category of DG comodules
over $ \CE $. An interesting question, which we leave entirely open, is whether the twisting cochain $ \tau_{r,n}(A) $ may be acyclic in the `unstable' range ({\it i.e.},
for $r$ and $n$ other than $ \infty $).

The second application of our stabilization theorem is concerned with {\it bigraded} DG algebras ({\it i.e.}, augmented DG algebras equipped with an extra polynomial grading). In this case, the trace subalgebra actually coincides with the
algebra of $ \GL_\infty$-invariants, so the stabilization theorem implies
\begin{equation}
\la{trra5}
{\bL}[\rHC_\bullet(A)] \cong \H_\bullet(A, \infty)^{\GL_\infty}\ .
\end{equation}
Equating the (graded) Euler characteristics of both sides of \eqref{trra5}, one can obtain some interesting combinatorial identities. We will compute a number of explicit examples in
Section~\ref{comb}. Our computations are inspired by the Lie homological approach to the famous Macdonald conjectures \cite{Ha, FGT}. The idea of using classical Molien-Weyl matrix integrals in these computations is borrowed from \cite{EG}.

Finally, we would like to comment on our proof of the stabilization theorem. Although
it has the same basic ingredients as the proof of the LQT Theorem (e.g., we use invariant theory and the Milnor-Moore structure theorem for commutative cocommutative Hopf algebras), there are two major differences. First, unlike in the case of Lie homology,
we do not have a `small' canonical complex for computing representation homology; thus, we are
forced to work with arbitrary DG resolutions, which makes our arguments more general and flexible
but less explicit. Second, to stabilize representation homology one has to take inverse limits (rather than direct limits in the Lie homology case). This gives some arguments a topological flavor and requires the use of a more sophisticated
version of invariant theory of inductive limits of groups acting on inverse limits of modules
and algebras\footnote{Luckily, the basics of such an invariant theory have been worked out in
\cite{T-TT}, and the results of this paper can be applied to our situation
({\it cf.} Section~\ref{S1.1}).}.

The paper is organized as follows. In Section~\ref{S1}, we introduce notation,
recall some basic facts about DG algebras and review the material from \cite{BKR} needed
for the present paper. In Section~\ref{S1.3}, we prove an extended version of Procesi's Theorem for DG algebras (Theorem~\ref{pS1.3.1}). Although this result is probably known, we could not
find an appropriate reference in the literature. Section~\ref{S2} contains the main
results of this paper, including the construction of the trace subalgebra
(Sections~\ref{S2.1}-\ref{S2.2}) and the proof of the stabilization theorem (Section~\ref{S2.3}).
In Section~\ref{KDOC}, after a brief overview of the theory of twisting cochains (Section~\ref{actw}), we establish the Koszul duality between stable representation homology and Lie homology (Theorem~\ref{KD2}) and
address our main question \eqref{quest}. In Section~\ref{Cryscoh},
we compute the `stable limit' of the De Rham cohomology of $ \DRep_n(A)^{\GL_n} $.
This computation was motivated by a question of D.~Kaledin ({\it cf.}
Remark~\ref{RKal}). In Section~\ref{S3}, we extend our results to the bigraded DG
algebras. In this case, one can prove a more refined version of the stabilization
theorem (Theorem~\ref{tS3.4}), which gives more information on the stability
of the homology groups. This section also contains explicit
computations of Euler characteristics and examples of combinatorial identities
mentioned above. The paper ends with an Appendix where we prove some basic lemmas
on cyclic derivatives in the $\Z_2$-graded setting.

\subsection*{Acknowledgements}{\footnotesize
We are very grateful to G.~Felder for many interesting suggestions
and for allowing us to use his {\tt Macaulay2} software for computing homology of commutative DG algebras. We would also like to thank B.~Feigin, D.~Kaledin, B.~Shoikhet and B.~Tsygan for inspiring discussions and comments. Yu. B. is
grateful to Forschungsinstitut f\"ur Mathematik (ETH, Z\"urich) for its hospitality
and financial support during his stay in Fall 2012. The work of Yu. B. was partially supported by NSF grant DMS 09-01570 and the work of A.~R. was supprted by the Swiss National Science Foundation (Ambizione Beitrag Nr. PZ00P2-127427/1).}

\section{Preliminaries}
\la{S1}
\subsection{Notation and conventions}
\la{notation}
Throughout this paper, $ k $ denotes a base field of characteristic zero. An unadorned tensor product
$\, \otimes \,$ stands for the tensor product $\, \otimes_k \,$ over $k$. An algebra means an
associative $k$-algebra with $1$; the category of such algebras is denoted $ \Alg_k $. Unless stated otherwise, all differential graded (DG) objects are equipped with differentials of degree $-1$, and the Koszul sign rule is systematically used. The categories of complexes, DG algebras and commutative DG algebras over $k$ are denoted $ \Com_k $, $\,\DGA_k $ and $ \cDGA_k $, respectively.
As usual, $ \Alg_k $ is identified with the full subcategory of
$ \DGA_k $ consisting of DG algebras with a single nonzero component in degree $0$. We call a DG
algebra free if its underlying graded algebra is free, i.e. isomorphic to
the tensor algebra $ T_k V $ of a graded $k$-vector space $V$. Following  \cite{L}, we denote
the graded symmetric algebra of $V$ by $ \bL(V) $: thus, $\, \bL(V) := \Sym_k(V_{\rm ev}) \otimes \Lambda_k(V_{\rm odd}) $, where $ V_{\rm ev}$ and $ V_{\rm odd}$ are the even and odd components of $V$, respectively.

\subsection{Augmented DG algebras}
\la{augdgas}
In this paper, we will work with augmented DG algebras. Recall
that $ A \in \DGA_k $ is {\it augmented} if it is given together with a DG algebra map $ \varepsilon: A \to k $. A morphism of augmented algebras $\, (A, \varepsilon) \to (A', \varepsilon') \,$ is a morphism $ f: A \to A' $ in $ \DGA_k $ satisfying $ \varepsilon' \circ f = \varepsilon $. The category of augmented DG algebras will be denoted $ \DGA_{k/k} $. It is easy to see that $ \DGA_{k/k} $ is equivalent to the category $ \DGA $ of nonunital DG algebras, with
$ A \in \DGA_{k/k} $ corresponding to its
augmentation ideal $\, \bar{A} := \Ker(\varepsilon) \,$. Similarly, we define the (equivalent) categories $\cDGA_{k/k}$ and $ \cDGA $ of commutative DG algebras: augmented and nonunital, respectively.

\subsubsection{Model structures}
The categories $\DGA_k$ and $\cDGA_k$ carry natural model structures in the sense of Quillen \cite{Q1, Q2}. The weak equivalences in these model categories are the quasi-isomorhisms and the fibrations are the degreewise surjective maps. The cofibrations are characterized in abstract terms: as morphisms satisfying the left lifting property with respect to the acyclic fibrations (see \cite{H}). The model structures on $\DGA_{k}$ and $\cDGA_{k}$ naturally induce model structures on the corresponding categories of augmented algebras ({\it cf.} \cite[3.10]{DS}). In particular, a morphism $ f\,:\,A \rar  B$ in $\DGA_{k/k}$ is a weak equivalence (resp., fibration; resp., cofibration) iff $ f\,:\,A \rar B$ is a weak equivalence (resp., fibration; resp., cofibration) in $\DGA_k $. All objects in $ \DGA_{k} $ and $ \DGA_{k/k} $ are fibrant. The cofibrant objects in $ \DGA_{k/k} $ can be descirbed more explicitly than in $ \DGA_k $: by a theorem of Lef\`evre (see \cite[Theorem~4.3]{Ke}), every cofibrant
$ A \in \DGA_{k/k} $ is isomorphic to a retract of $ \boldsymbol{\Omega}(C) $, where $ \boldsymbol{\Omega}(C) $ is the cobar construction of an augmented conilpotent co-associative DG coalgebra $ C $.

\subsubsection{Homotopies}
\la{A.2}
Recall that, in an abstract model category $ \C $, a homotopy 
(more precisely, right homotopy) between two morphisms $\,f,g \in \Hom_{\C}(A, B)\,$ 
is defined to be a map $\,h:\, A \to B^I $ making the following diagram commutative:
\begin{equation*}
\begin{diagram}[small, tight]
  &           &   A\\
  & \ldTo^{f} & \dDotsto^{h} & \rdTo^{g} & \\
  B & \lOnto^{p_1} & B^I & \rOnto^{p_2} & B
\end{diagram}
\end{equation*}
where $ B^I $ is a path object of $B$. If $A$ is cofibrant and 
$ B $ is fibrant in $ \C $, homotopy defines an equivalence relation 
$ \sim $ on $ \Hom_{\C}(A, B) $, and we say that $f$ and $g$ are {\it homotopic} 
if $ f \sim g $ (see \cite[Sect.~4]{DS}).

For DG algebras, there is an explicit construction of homotopies 
(the so-called $M$-homotopies), which we now describe in the augmented case 
({\it cf.} \cite[Appendix~{B.4}]{BKR}).
Recall that for $ A \in \DGA_{k/k}$, $ \bar{A} \subset A $ denotes the augmentation ideal. If $x$ is a homogeneous variable of any degree, we define the 2-term complex
$ V_x := [0 \rar k.x \xrightarrow{d} k.dx \rar 0] $ and let
$ T(V_x) \in \DGA_{k/k} $ be its tensor algebra over $k$ equipped with the canonical augmentation. If $ \deg(x) = 1 $, we also define $ \Omega := \bSym(V_x) \in \cDGA_{k/k} $, which is the algebraic de Rham complex of the affine line (though with differential of degree $-1$).
Now, we say that an acyclic cofibration in $\DGA_{k/k}$ is {\it special} if it is of the form $ A \into A \ast_k (\amalg_{\lambda \in I} T(V_{x_{\lambda}}))$ for some indexing set $I$. It is easy to see that every morphism $f$ in $\DGA_{k/k}$ factors as $f=p\,i$, where $p$ is a fibration and $i$ is a special acyclic cofibration ({\it cf.} Lemma~\ref{lA1.2.1} in Section~\ref{A.1.2}).
\begin{lemma} \la{lA2.1}
Let $i: A \into B:= A \ast_k (\amalg_{\lambda \in I} T(V_{x_{\lambda}}))$ be a special acyclic cofibration in $\DGA_{k/k}$. There is a morphism $p:B \rar A$ in $\DGA_{k/k}$ such that $pi=\id_A$ and $ip$ is homotopic to $\id_B$ via a homotopy $h:B \rar B \otimes \Omega$ such that $h(\bar{B}) \subset \bar{B} \otimes \Omega$.
\end{lemma}

\begin{proof}
The morphism $p$ is obtained by setting $p(x_{\lambda})=p(dx_{\lambda})=0$ for all $\lambda \in I$ and $p|_A=\id_A$. Clearly, $p\,i=\id_A$. We construct a (polynomial) homotopy between $ip$ and $\id_B$ as follows. Set $\phi_t|_{A}=i$ and let $$\phi_t(x_{\lambda})=t.x_{\lambda}\,,\,\,\,\phi_t(dx_{\lambda})=t.dx_{\lambda} \,\,\, \forall\,\lambda \in I\mathrm{.}$$
$\phi_t$ extends to a morphism $\phi_t\,:\,B \rar B$ in $\DGA_{k/k}$. Let $s_t$ be the unique $A$-linear $\phi_t$-derivation from $B$ to $B$ such that
$$s_t(x_{\lambda})=0\,,\,\,\,\,\,s_t(dx_{\lambda}) = x_{\lambda}\,\,\,\forall\,\lambda \in I \mathrm{.}$$
Then, $\phi_0=ip$, $\phi_1=\id_B$ and $\frac{d}{dt}(\phi_t)=[d,s_t]$. Further, by construction, $\phi_t$ is a morphism in $\DGA_{k/k}$ and $s_t(\bar{B}) \subset \bar{B}$ for all $t$. Hence, the homotopy
$$h:B \rar B \otimes \Omega\,,\,\,\,\,\, \alpha \mapsto \phi_t(\alpha)+dt.s_t(\alpha) $$
satisfies $h(\bar{B}) \subset \bar{B} \otimes \Omega$.
\end{proof}

The following proposition gives an explicit description of homotopies
in $ \DGA_{k/k} $.

\begin{prop} \la{pA2.1}
Let $f,g\,:\, A \to B $ be two morphisms in $\DGA_{k/k}$, with $A$ cofibrant. If $ f \sim g $, 
then there exists $h\,:\,A \rar B \otimes \Omega$ in $\DGA_k$ such that $h(0)=f,\, h(1)=g$ and $ h(\bar{A}) \subset \bar{B} \otimes \Omega\,$.
\end{prop}

\begin{proof}
Let $B \stackrel{i}{\into} B^I \stackrel{q}{\onto} B \times_k B$ be a factorization of the diagonal map $ \Delta_B:\, 
B \to B \times_k B $ in $\DGA_{k/k} $, such that $i$ is a special acyclic cofibration and $q$ is a fibration. By definition, $ B^I $ is a very good path object for $ B $ ({\it cf.} \cite[Sect.~4.12]{DS}). If $A$ is cofibrant, any very good path object on $B$ gives a right homotopy between (homotopic) morphisms from $A$ to $B$. Hence, if $ f \sim g $, there exists $h: A \rar B^I$ such that $p_1qh=f$ and $p_2qh=g$, where $p_1,p_2:B \times_k B \rar B$ are the natural projections in $\DGA_{k/k}$. By Lemma~\ref{lA2.1}, there exists $p:B^I \rar B$ such that $pi=\id_B$ and
$ip$ is homotopic to $\id_{B^I}$ via a homotopy $h':B^I \rar B^I \otimes \Omega$ such that $h'(\overline{B^I}) \subset \overline{B^I} \otimes \Omega$. Hence, $f=p_1qh$ is homotopic to
$p_1q(ip)h=ph$ via a homotopy $h_1: A \rar B \otimes \Omega$ such that $h_1(\bar{A}) \subset \bar{B} \otimes \Omega$. Similarly, $p_2q(ip)h=ph$ is homotopic to
$g=p_2qh$ via a homotopy $h_2: A \rar B \otimes \Omega$ such that $h_2(\bar{A}) \subset \bar{B} \otimes \Omega$. The required homotopy $h:A \rar B\otimes \Omega$ is given by the composition
$$\begin{diagram}[small] A & \rTo^{h_1} & B \otimes \Omega & \rTo^{h_2 \otimes \id_{\Omega}} & B \otimes \Omega \otimes \Omega & \rTo^{\id_B \otimes \mu} &B \otimes \Omega \end{diagram}\ ,$$
where $\mu $ denotes the multiplication map on $\Omega$. Clearly, $h(\bar{A}) \subset \bar{B} \otimes \Omega$.
\end{proof}

\subsubsection{Remark} The above proposition is analogous to a
well-known description of homotopies in $ \DGA_k $
(see, e.g., \cite[Prop.~3.5]{FHT} or \cite[Prop.~B.2]{BKR}).
However, the fact that $ h $ can be chosen to respect
augmentations (the last condition in Proposition~\ref{pA2.1})
does not seem to follow automatically from these results.

\subsection{Derived representation schemes}
\la{S1.2}
In this section, we recall the basic construction of derived representation schemes from \cite{BKR}.

\subsubsection{The representation functor} \la{S1.2.1}
For an integer $ n \ge 1 $, denote by $ \M_n(k) $ the algebra of $ n \times n $ matrices with entries in $ k $ and define the following functor
\begin{equation}\la{root}
\sqrt[n]{\,\mbox{--}\,}\,:\,\DGA_k \rar \DGA_k\ ,\quad
A \mapsto  [A \ast_k \M_n(k)]^{\M_n(k)}\ ,
\end{equation}
where $\,A \,\ast_k\, \M_n(k) \,$ is the free product (coproduct) in $ \DGA_k $ and $\,[\,\ldots\,]^{\M_n(k)} $ stands for the (graded) centralizer of $ \M_n(k) $ as the subalgebra in $\,A \,\ast_k\, \M_n(k) \,$. Next, recall that the forgetful functor $\, \cDGA_k \to \DGA_k \,$ has a natural left adjoint that assigns to a DG algebra $ A $ its maximal commutative quotient:
\begin{equation}\la{ab}
(\,\mbox{--}\,)_{\nn} :\ \DGA_k \rar \cDGA_k \ ,\quad A \mapsto A_\nn := A/\langle [A,A]\rangle\ .
\end{equation}
Combining  \eqref{root} and \eqref{ab}, we define
\begin{equation}
\la{rootab}
(\,\mbox{--}\,)_n :\ \DGA_k \rar \cDGA_k\ ,\quad A \mapsto A_n := (\!\sqrt[n]{A})_{\nn}\ .
\end{equation}
The following theorem is a special case of \cite[Theorem~2.1]{BKR}.
\begin{theorem}
\la{nS2.1t1}
For any $\,A \in \DGA_k\,$, the commutative DG algebra $\,A_n\,$ (co)represents the functor
\begin{equation}\la{rep}
\Rep_n(A):\ \cDGA_k \to \Sets\ ,\quad B \mapsto \Hom_{\DGA_k}(A,\, \M_n(B))\ .
\end{equation}
\end{theorem}

\noindent
Theorem~\ref{nS2.1t1} implies that there is a bijection
\begin{equation}
\la{eS1.2.1}
\Hom_{\DGA_k}(A,\M_n(B)) \,=\, \Hom_{\cDGA_k}(A_n, B) \ ,
\end{equation}
functorial in $ A \in \DGA_k $ and $B \in \cDGA_k $. The algebra $\, A_n \,$ should be thought of as the coordinate ring $\, k[\Rep_n(A)] \,$ of an affine DG scheme parametrizing the $n$-dimensional $k$-linear representations of $A$. Letting $\,B = A_n\,$ in \eqref{eS1.2.1}, we get a canonical DG algebra map
$$
\pi_n: A \to \M_n(A_n) \ ,
$$
which is called the {\it universal $n$-dimensional representation} of $A$.

Now, observe that for $ A \in \DGA_{k/k} $, the DG algebra $ A_n $ is naturally augmented, with augmentation map $\, \varepsilon_n : A_n \to k $ coming from \eqref{rootab} applied to the augmentation map of $A$. This defines a functor $\,\DGA_{k/k} \to \cDGA_{k/k} $, which we again denote $ (\mbox{--}\,)_n $. Let $\overline{\M}_n(\mbox{--})\,:\,\cDGA_{k/k} \rar \DGA_{k/k}$ denote the functor $B \mapsto k \oplus \M_n(\bar{B})$. The following theorem is immediate from~\cite[Theorem~2.2]{BKR}.

\begin{theorem}
\la{nS2.1t2}
(a) The functors $(\mbox{--})_n\,:\, \DGA_{k/k} \rightleftarrows \cDGA_{k/k} \,:\, \overline{\M}_n(\mbox{--}) $ form a Quillen pair.

(b) The functor $\, ( \mbox{--}\,)_n\,:\,\DGA_{k/k} \rar \cDGA_{k/k}$ has a total left derived functor defined by
$$
\L(\,\mbox{--}\,)_n:\,\Ho(\DGA_{k/k}) \to \Ho(\cDGA_{k/k}) \ ,
\quad A  \mapsto  (QA)_n\ ,
$$
where $QA \stackrel{\sim}{\twoheadrightarrow} A$ is any cofibrant resolution of $A$ in $\DGA_{k/k}$.

(c) For any $A$ in $\DGA_{k/k}$ and $B$ in $\cDGA_{k/k}$, there is a canonical isomorphism
$$ \Hom_{\Ho(\cDGA_{k/k})}(A_n \,,\, B)\,\cong\, \Hom_{\Ho(\DGA_{k/k})}(A\,,\,\overline{\M}_n(B))\,\text{.}$$
\end{theorem}
\vspace{1ex}

\begin{definition} Given $ A \in \Alg_{k/k} $ we define
$\, \DRep_n(A) := \L(QA)_n \,$, where $ QA \sonto A $ is a cofibrant resolution of $ A $ in $\DGA_{k/k} $. The homology of $ \DRep_n(A) $
is an augmented (graded) commutative algebra, which is independent of
the choice of resolution (by Theorem~\ref{nS2.1t2}). We set
\begin{equation}
\la{rephom} \mathrm{H}_{\bullet}(A,n)\,:=\,\mathrm{H}_{\bullet}[\DRep_n(A)] \end{equation}
and call \eqref{rephom} the {\it $n$-dimensional representation homology} of $A$.
\end{definition}

\vspace{1ex}

By \cite[Theorem~2.5]{BKR}, if $ A \in \Alg_k $ then there is an isomorphism of algebras $\, \H_0(A,n) \cong A_n $, hence in this case $ \DRep_n(A) $ may indeed be
viewed as (the representative of) a `higher' derived functor of \eqref{rep}.

\subsubsection{$\GL_n$-invariants}
\la{glinv}
The group $ \GL_n(k) $ acts naturally on  $\,A_n\,$ by DG algebra
automorphisms. Precisely, each $ g \in \GL_n(k) $ defines a unique automorphism of $\,A_n\,$ corresponding under \eqref{eS1.2.1} to the composite map
$$
A \xrightarrow{\pi_n} \M_n(A_n) \xrightarrow{\text{Ad}(g)} \M_n(A_n)\ ,
$$
where $ \pi_n $ is the universal $n$-dimensional representation of $A$. This action is natural in $A$ and thus defines a functor
\begin{equation}
\la{vgl}
(\,\mbox{--}\,)_n^{\GL} :\ \DGA_{k/k} \to \cDGA_{k/k}\ ,\quad
A \mapsto A_n^{\GL_n}\ ,
\end{equation}
which is a subfunctor of the representation functor $ (\,\mbox{--}\,)_n $.
On the other hand, there is a natural action of $ \GL_n(k) $ on the $n$-th
representation homology of $ A $ so we can form the invariant subaglebra
$ \H_\bullet(A, n)^{\GL_n} $. The next theorem, which is a consequence of
\cite[Theorem~2.6]{BKR}, shows that these two constructions agree.
\begin{theorem}
\la{nS2.1t5}
$(a)$ The functor $(\mbox{--})_n^{\GL}$ has a total left derived functor
$$
\L[(\,\mbox{--}\,)_n^{\GL}]:\,\Ho(\DGA_{k/k}) \to \Ho(\cDGA_{k/k})\,,\,\,\,\, A \mapsto (QA)_n^{\GL} \ .
$$

$(b)$ For any $ A \in \DGA_{k/k} $, there is a natural isomorphism of graded algebras
$$
\H_\bullet[\L(A)_n^{\GL}] \cong \H_\bullet(A, n)^{\GL_n}\ .
$$
\end{theorem}
\noindent
If $\,A \in \Alg_{k/k} \,$, abusing notation we will often write
$ \DRep_n(A)^{\GL} $ instead of $ \L(A)_n^{\GL} $.

\subsubsection{Suspensions}
Unlike $ \DGA_k $ and $ \cDGA_k $, the model categories of augmented DG algebras are pointed: {\it i.e.}, the initial and terminal objects in $ \DGA_{k/k} $ and $ \cDGA_{k/k} $ coincide (both are equal to $k$). Hence, by \cite[\S\,I.2]{Q1}, the corresponding homotopy categories $ \Ho(\DGA_{k/k}) $ and $ \Ho(\cDGA_{k/k}) $ are equipped with {\it suspension functors} which we denote by $ \Sigma \,$. As in the category of topological spaces, $ \Sigma A $ is defined in general as the cofibre of a cofibration $\, QA \amalg QA \into \Cyl(QA) \,$, where $ \Cyl(QA) $ is a cylinder object over 
a cofibrant replacement of $ A $.  In the case of DG algebras, suspension can be computed 
as the homotopy pushout ({\it cf.} \cite[Sect.~11.3]{DS})
$$
\Sigma A \,\cong\, \L\colim\,\{k \xleftarrow{\varepsilon} A \xrightarrow{\varepsilon} k\}\ ,
$$
where $A$ is viewed as an object in $ \Ho(\DGA_{k/k}) $ or $ \Ho(\cDGA_{k/k}) $. The next result shows that the derived representation functor of Theorem~\ref{nS2.1t2} is compatible with 
$ \Sigma $ in a natural way.
\begin{theorem}
\la{susp}
For each $ n \ge 1 $, there is an isomorphism of functors
$\, \L(\,\mbox{--}\,)_n \circ \Sigma \,\cong \,\Sigma \circ \L  (\,\mbox{--}\,)_n\,$.
\end{theorem}
\begin{proof}
This is an immediate consequence of Theorem~\ref{nS2.1t2} and~\cite[\S I.4, Prop.~2]{Q1}.
\end{proof}

\subsection{Explicit presentation}
\la{explpr}
Given a free DG resolution $ R $ of $ A \in \Alg_k $, the DG algebra
$ R_n \in \cDGA_k $ can be described explicitly.
Specifically, let $\{x^\alpha\}_{\alpha \in I}$ be a set of
homogeneous generators of $R$, and let $ d: R_\bullet \to R_{\bullet-1} $ be its differential. Consider a free graded algebra $ \tilde{R} $ on generators $\,\{x^{\alpha}_{ij}\,:\, 1\leq i,j\leq n\, ,\, \alpha \in I\}\,$, and define the algebra map
$$
\pi:\, R \to \M_n(\tilde{R})\ , \quad  x^\alpha \mapsto X^\alpha \ ,
$$
where $ X^\alpha := \|x^{\alpha}_{ij}\| $, $\, \alpha \in I $. Then, letting $\, \tilde{d}(x^{\alpha}_{ij}) := \| \pi(d x^{\alpha}) \|_{ij}\,$, define a differential
$ \tilde{d} $ on generators of $ \tilde{R} $ and extend it to the whole of $ \tilde{R}$ by the Leibniz rule. This makes $ \tilde{R} $ a DG algebra. The abelianization of
$\tilde{R} $ is a free (graded) commutative algebra generated by (the images of) $\,x^\alpha_{ij}\,$ and the differential on $ \tilde{R}_\nn $ is induced by $ \tilde{d}$. The next result  
is proved in \cite[Theorem~2.8]{BKR}.
\bthm
\la{comp}
There is an isomorphism of DG algebras $\, \sqrt[n]{R} \cong \tilde{R} \,$. Consequently, $\,R_n \cong \tilde{R}_\nn\,$.
\ethm
Using Theorem~\ref{comp}, one can construct a finite presentation for $ R_n $, and hence an explicit
model for $ \DRep_n(A) $, whenever a finite free resolution $ R \to A $ is available.

\subsection{Derived characters}
\la{trmaps}
Following \cite{BKR}, we now construct the trace maps relating cyclic homology to representation homology. For a DG algebra $ R $ (either unital or non-unital), we write $\, R_\n := R/[R,R] \,$, where $ [R,R] $ is the subcomplex of $R$ spanned by the (graded) commutators of elements in $R$. Using this notation, we define the functor
\begin{equation}
\la{cycf}
\FT:\,\DGA_{k/k} \to \Com_{\,k}\ ,\quad R \mapsto \bar{R}_{\n} \ ,
\end{equation}
which is called the (reduced) {\it cyclic functor}.
The next theorem, which is a well-known result due to
Feigin and Tsygan \cite{FT}, justifies this terminology.
\begin{theorem}
\la{ftt}
$ (a) $ The functor \eqref{cycf} has a total left derived
functor
$$
\L\FT:\,\Ho(\DGA_{k/k}) \to \D(k)\ ,\quad A \mapsto \FT(R)\ ,
$$
where $ R = QA $ is a(ny) cofibrant resolution of $A$ in $ \DGA_{k/k} $.

$(b)$ For $ A \in \Alg_{k/k} $, there is a natural isomorphism
of graded vector spaces
$$
\H_\bullet[\L\FT(A)] \cong \rHC_\bullet(A)\ ,
$$
where $ \rHC_\bullet(A) $ denotes the (reduced) cyclic homology of
$A$.
\end{theorem}
\noindent

For a simple conceptual proof of Theorem~\ref{ftt} and its generalizations we refer to \cite{BKR}, Section~3.

\vspace{2ex}

Now, fix $ n \ge 1 $ and, for $ R \in \DGA_k $, consider the composite map
$$
R \xrightarrow{\pi_n} \M_n(R_n) \xrightarrow{\Tr} R_n
$$
where $ \pi_n $ is the universal representation of $R$ and $\,\Tr \,$ is the usual matrix trace. This map factors through $ R_\n $ and its image lies in $ R_n^{\GL} $. Hence, we get a morphism of complexes
\begin{equation}
\la{e2s1}
\Tr_n(R)_{\bullet}:\ R_{\natural} \to R_n^{\GL}\ ,
\end{equation}
which extends by multiplicativity to a map of graded commutative algebras
\begin{equation}
\la{trm}
\bSym \Tr_n(R)_{\bullet}:\ \bSym(R_{\natural}) \to R_n^{\GL}\ .
\end{equation}

For an augmented DG algebra $ R \in \DGA_{k/k} $, the natural
inclusion $\,\bar{R} \into R \,$ induces
$\, \FT(R) \to R_\n $. Composed with \eqref{e2s1} this defines a morphism of functors $\,\FT \to (\mbox{--}\,)_n^\GL \,$, which descends to a morphism of the derived functors 
from $ \Ho(\DGA_{k/k}) $ to $ \D(k) \,$ ({\it cf.} \cite[Remark~A.5.1]{BKR}):
\begin{equation}
\la{trm2}
\L\Tr_n:\, \L\FT \to \L(\,\mbox{--}\,)_n^\GL\, .
\end{equation}
Now, if $ A \in \Alg_{k/k} $ is an ordinary algebra, applying \eqref{trm2} to a cofibrant resolution  $\, R = QA \,$  of $ A $ in $ \DGA_{k/k}$, taking homology and using the identification of Theorem~\ref{ftt}$(b)$, we get natural maps
\begin{equation}
\la{trm3}
\Tr_n(A)_{\bullet} :\,\rHC_{\bullet}(A) \to
\H_{\bullet}(A,n)^{\GL_n}\ ,\quad \forall\,n \ge 0 \ .
\end{equation}
In degree zero, $ \Tr_n(A)_0 $ is induced by the obvious linear map
$ A \to k[\Rep_n(A)]^{\GL_n} $ defined by taking characters of representations. Thus, the higher components of \eqref{trm3} may be thought of as derived (or higher) characters\footnote{We refer to \cite[Sect.~4.3]{BKR} for an explicit formula evaluating \eqref{trm3} on cyclic chains.} of $n$-dimensional representations of $A$. For each $ n \ge 1 $, these characters assemble to a single homomorphism of graded commutative algebras which we denote
$$
\bSym \Tr_n(A)_{\bullet}:\,
\bSym[\rHC_{\bullet}(A)] \to
\mathrm{H}_{\bullet}(A,n)^{\GL}\ .
$$
In the present paper, we will study the behavior of $\, \bSym \Tr_n(A)_{\bullet} \,$  as $\, n \to \infty $.

\section{Procesi theorem for DG algebras}
\la{S1.3}
We will need the following result which extends a well-known theorem of Procesi (see \cite{P}).

\begin{theorem}
\la{pS1.3.1}
For any $ R \in \DGA_k $,  the trace morphism \eqref{trm} is degreewise surjective.
\end{theorem}

The rest of this section is devoted to proving Theorem~\ref{pS1.3.1} through a step-by-step reduction to a form of the fundamental theorem of invariant theory in the $\Z_2$-graded setting.

\subsection{Fundamental theorem of invariant theory}
\la{SI.1}
 We begin by introducing some notation
({\it cf.} \cite{KP},  Section 4.3). Let $\sigma \in S_{p+q}$ have the decomposition
$\sigma=(i_1, \ldots ,i_k)(j_1,\ldots ,j_r) \ldots (l_1, \ldots ,l_s)$ into a product of disjoint cycles (including cycles of length
$1$). For matrices $M_1,\ldots ,M_{p+q}$ we define
$$\mathrm{Tr}_{\sigma}(M_1, \ldots ,M_{p+q}) \,:=\, \mathrm{Tr}(M_{i_1} \ldots M_{i_k})\mathrm{Tr}(M_{j_1} \ldots M_{j_r})\ldots \mathrm{Tr}(M_{l_1}\ldots M_{l_s}) \,\,\mathrm{.}$$
Fix $n \in \N$. Let $X_i:=(x^{(i)}_{jk})$ be even $(n \times n)$-matrix variables for $1 \leq i \leq p$ and $Y_i:=(y^{(i)}_{jk})$ be odd $(n \times n)$-matrix variables for $ 1 \leq i \leq q$. Let $k[x^{(i)}_{jk},\,y^{(l)}_{rs}]$ denote the free $\Z_2$-graded commutative
algebra generated by the even variables $\,\{x^{(i)}_{jk}\,, 1 \leq i \leq p\,,1 \leq j,k \leq n \}\,$ and the odd variables
$\,\{y^{(l)}_{rs}\,, 1 \leq l \leq q\,,1 \leq r,s \leq n\}\,$. Suppose that
$h(X_1,\ldots ,X_p,Y_1,\ldots ,Y_q) \,\in\, k[x^{(i)}_{jk},y^{(l)}_{rs}]$ is a polynomial function on the matrix variables
$X_1,\ldots ,X_p,Y_1, \ldots ,Y_q$ that is linear in each factor.

The next lemma is an extension of the fundamental theorem of invariant theory to the $\Z_2$-graded setting ({\it cf.} \cite[Theorem~4.3]{KP}).
\begin{lemma} \la{lS1.3.1}
$\,h(X_1,\ldots ,X_p,Y_1,\ldots ,Y_q)=h(gX_1g^{-1},\ldots ,gX_pg^{-1},gY_1g^{-1},\ldots ,gY_pg^{-1})\,$ for all $\,g \in \GL_n(k)\,$ if and only if
$h$ can expressed in the form
$$\sum_{\sigma \in S_{p+q}} c_{\sigma} \mathrm{Tr}_{\sigma}(X_1,\ldots ,X_p,Y_1,\ldots ,Y_q)$$
where $c_{\sigma} \in k$ for all $\sigma \in S_{p+q}\,$.
\end{lemma}

\begin{proof}
Let $u_1,\ldots ,u_q$ be (formal) odd variables. Let $Z_1,\ldots ,Z_q$ be even matrix variables. Set $Y_i:=u_iZ_i$ for $1 \leq i \leq q$.
Then, $h(X_1,\ldots ,X_p,Y_1,\ldots ,Y_q)= u_1 \ldots u_q \,G(X_1,\ldots ,X_p,Z_1,\ldots ,Z_q)$. Clearly, $ h $ is invariant under
simultaneous conjugation by elements of $\GL_n$ iff $G$ is invariant under simultaneous
conjugation by elements of $\GL_n$. By Theorem 4.3 of~\cite{KP}, which is a form of the fundamental theorem
of invariant theory, $G$ is invariant under simultaneous
conjugation by elements of $\GL_n$ if and only if
$$ G(X_1,\ldots ,X_p,Z_1, \ldots ,Z_q) = \sum_{\sigma \in S_{p+q}} c'_{\sigma} \mathrm{Tr}_{\sigma}(X_1,\ldots ,X_p,Z_1,\ldots ,Z_q)$$ for
some constants $c_\sigma \in k$.
Hence, $h$ is invariant under simultaneous
conjugation by elements of $\GL_n$ iff
 $$h(X_1,\ldots ,X_p,Y_1,\ldots ,Y_q) = \sum_{\sigma \in S_{p+q}} \pm c'_{\sigma} \mathrm{Tr}_{\sigma}(X_1,\ldots ,X_p,Y_1,\ldots ,Y_q) $$
for $Y_i=u_iZ_i$. Putting $u_i=y^{(i)}_{j_ik_i}$ and $Z_i=e_{j_ik_i}$, we see that
$$ h(X_1,\ldots ,X_p,Y_1,\ldots ,Y_q) = \sum_{\sigma \in S_{p+q}} \pm c'_{\sigma} \mathrm{Tr}_{\sigma}(X_1,\ldots ,X_p,Y_1,\ldots ,Y_q) $$
 for $Y_i=e_{j_ik_i}y^{(i)}_{j_ik_i}$. By multilinearity in each factor,
$$ h(X_1,\ldots ,X_p,Y_1,\ldots ,Y_q) = \sum_{\sigma \in S_{p+q}} \pm c'_{\sigma} \mathrm{Tr}_{\sigma}(X_1,\ldots ,X_p,Y_1,\ldots ,Y_q) $$
for $Y_i=(y^{(i)}_{jk})$ for $1 \leq i \leq q$. This proves the required lemma.
\end{proof}

\subsection{Proof of Theorem~\ref{pS1.3.1}} \la{nS3.2}
We prove Theorem~\ref{pS1.3.1} in three steps.

\subsubsection{Reduction to free graded algebras} \la{1.3.1} Since $\mathrm{Tr}_n(A)_{\bullet}$ is a morphism in $\cDGA_k$, it suffices to forget the differentials
and prove our proposition in the category of graded commutative algebras. Recall that
a graded algebra $R$ is said to be free if as graded algebras, $R \cong TV$ for some graded vector space $V$.
Suppose that Theorem~\ref{pS1.3.1} is true for any free graded algebra $R$. Then, write $A$ as the quotient
of a free graded algebra $R$. One then has the following commutative diagram
$$\begin{diagram}
  \bSym(R_{\natural}) &\rOnto^{\mathrm{Tr}_n(R)} & R_n^{\GL}\\
    \dTo    &   &    \dOnto\\
  \bSym(A_{\natural}) &\rTo^{\mathrm{Tr}_n(A)} &  A_n^{\GL}
  \end{diagram}
$$
Indeed, the top arrow is degreewise surjective by assumption. 
On the other hand, since $A_n$ is the quotient of $R_n$ by the ideal generated by the elements of the form $ \|\pi(\alpha)\|_{ij} $ for $ \alpha \in \Ker(R \to A)$ (see Section~\ref{explpr}), $R_n \rar A_n$ is degreewise surjective. Since $\GL_n(k) $ is reductive, $R_n^{\GL} \rar A_n^{\GL}$ is degreewise surjective.
By commutativity of the above diagram, 
$\mathrm{Tr}_n(A)_{\bullet}$ is degreewise surjective.
Hence, it suffices to verify Theorem~\ref{pS1.3.1} for a free graded algebra.

\subsubsection{Reduction to finitely generated algebras} \la{1.3.2}

Let $R$ be a free graded algebra. To verify Theorem~\ref{pS1.3.1} for $R$, it suffices to verify Theorem~\ref{pS1.3.1}
for $R$ in the category of $\Z_2$-graded algebras. Suppose that $R \cong T_k V$ as graded algebras for some graded vector space $V$.
 Let $W_1:= \oplus_{n \in \Z} V_{2n+1}$ and let $ W_0:= \oplus_{n \in \Z} V_{2n}$. Then, $R \cong T_k W$ as $\Z_2$-graded
algebras. Hence, Theorem~\ref{pS1.3.1} for $T_k W$ implies Theorem~\ref{pS1.3.1} for $R$.

Further, since $T_k W$ can be written as a direct limit of its finitely generated subalgebras and since the functor $(\mbox{--})_n$ commutes with direct limits, Theorem~\ref{pS1.3.1} for each finitely generated subalgebra of $T_k W$ implies Theorem~\ref{pS1.3.1} for $T_k W$ itself
(and hence, by our arguments so far, Theorem~\ref{pS1.3.1} in general).

\subsubsection{Reduction to Lemma~\ref{lS1.3.1}}

We need to verify Theorem~\ref{pS1.3.1} for $R=T_k W$ where $W=W_0 \oplus W_1$ is finite dimensional over $k$.
We will verify Theorem~\ref{pS1.3.1} for $R=k\langle x,y \rangle$ where $x$ is
in degree $0$ and $y$ is in degree $1$. The argument for general $R$ is completely analogous. In this case, as
 graded commutative algebras, $R_n$ is the free commutative algebra generated by the variables $x_{ij}$ for
 $1 \leq i,j \leq n$ in degree $0$ and $y_{ij}$ for $1 \leq i,j\leq n$ in degree $1$.
Suppose that
$f \in R_n^{\GL}$. Then, by transforming $f$ via the substitutions
$x_{ij} \mapsto t.x_{ij}$, $y_{ij} \mapsto s.y_{ij}$, one sees that each summand of $f$ that is homogeneous
in the variables $\{x_{ij}\}$ and homogeneous in the variables $y_{ij}$ is in $R_n^{\GL}$.
It therefore suffices to verify
that for any $f \in R_n^{\GL}$ that is homogeneous of degree $p$ in the variables $x_{ij}$ and
homogeneous of degree $q$ in the (odd) variables $y_{ij}$ (for arbitrary nonegative integers $p,q$),  $f$ is in the image of $\bSym \mathrm{Tr}_n(R)_{\bullet} $.

Assume that $f \in R_n^{\GL}$ and that $f$  is homogeneous of degree $p$ in the variables $x_{ij}$ and
homogeneous of degree $q$ in the (odd) variables $y_{ij}$. View $f$ as a polynomial function in the even matrix
variable $X := (x_{ij})$
and the odd matrix variable $Y:=(y_{ij})$. Note that $f$ is of degree $p$ in $X$ and of degree $q$ in $Y$. Further, our assumption
 is that $f(gXg^{-1},gYg^{-1})=f(X,Y)$ for all $g \in \GL_n$.  Put $X:=\sum_{i=1}^p t_i X_i$
 and $Y:=\sum_{i=1}^q s_i Y_i$ where the $X_i$ are even matrix variables and the $Y_j$ are odd matrix variables.
Let $h(X_1,\ldots ,X_p,Y_1,\ldots ,Y_q)$ denote the summand of $f$ that is linear in the $t_i$'s and the $s_j$'s. By
construction,
$$ h(gX_1g^{-1},\ldots ,gX_pg^{-1},gY_1g^{-1},\ldots ,gY_qg^{-1})\, = \, h(X_1,\ldots ,X_p,Y_1,\ldots ,Y_q) $$
or any $g \in \GL_n$. Further, if $X_i=\frac{1}{p}\,X $ for all $1 \leq i \leq p$ and $Y_j=\frac{1}{q}\,Y $
for all $1 \leq j \leq q$, then $h(X_1,\ldots ,X_p,Y_1,\ldots ,Y_q)=\frac{p!q!}{p^pq^q}\,f(X,Y)$. Hence, Theorem~\ref{pS1.3.1} follows from Lemma~\ref{lS1.3.1}.

\section{Stabilization Theorem}
\la{S2}

\subsection{(Pro-)Invariant theory of $\GL_{\infty}\,$}
\la{S1.1}
We recall some basic facts about invariant theory of the inductive general linear group $\GL_{\infty}(k) $ acting on projective systems of DG $k$-algebras. A good reference for this material is \cite{T-TT}, where more details and proofs can be found. While \cite{T-TT} work with ordinary algebras, all the results we need can be easily adapted to the differential graded case.

\subsubsection{A topology on an inverse limit of $k$-algebras} \la{S1.1.1} Let $\mathbb I \subseteq \N $ be an infinite subset of the set of natural numbers, and let
$\{A_i; \,\mu_{ji}: A_j \to A_i\}_{i,j \in \mathbb{I}}$ be an inverse system of DG $k$-algebras indexed by $\mathbb I$. The inverse limit $ A_{\infty}\,:=\,\varprojlim_{i \in \mathbb I} A_i $ comes equipped with a natural subspace topology, which is induced from the product topology on $ \prod_{i \in \mathbb I} A_i $ arising from the discrete topology on each $A_i$. In this topology, the subspaces $ U_i:=\mu_i^{-1}(a_i)\,$, $\, i \in \mathbb I\,$, form a countable basis of open neighborhoods of each element $ (a_i)_{i \in \mathbb I} $ in $A_{\infty}$ (see \cite[Lemma~2.1]{T-TT}). The maps $\mu_i: A_{\infty} \rar A_i$ correspond to the natural projections $\prod_{i \in I} A_i \onto A_i $ and are clearly continuous for all $ i \in \mathbb I $.
We say that a family $\,\{f^{\alpha}\}_{\alpha \in \Lambda} $ of (homogeneous) elements
{\it topologically generates} $ A_{\infty} $ if the DG subalgebra generated by $\{f^{\alpha}\}$ is dense in $A_{\infty}$ with respect to the above topology.

\subsubsection{$\GL_{\infty}$-invariants}
\la{S1.1.2} For each pair of indices
$ i,j \in \mathbb{I} $ with $i \leq j$, denote by
$\,\lambda_{ij}:  \GL_i(k) \into \GL_j(k) \,$  the
natural inclusion, and let $ \GL_{\infty}(k) := \varinjlim_{i \in \mathbb I} \GL_i(k) $. Suppose that each DG algebra $A_i$ is equipped with a $\GL_i(k)$-action in such a way that
\begin{equation}
\la{condin}
g.\mu_{ji}(a_j)\,=\, \mu_{ji}(\lambda_{ij}(g).a_j),\,\,\,\,\forall \, g \in \GL_i(k)\,, \ i < j \,\mathrm{.}
\end{equation}
Then $\GL_{\infty}(k) $ acts naturally on $A_{\infty}$. Indeed, for $g \in \GL_i(k) \subset \GL_{\infty}(k) $ and $ (a_i) \in A_{\infty} $, we let
\begin{eqnarray*}
 (g.a)_j &=& \lambda_{ij}(g).a_j \,\,\,\forall \, i\leq j\\
         &=& \mu_{ji}(g.a_i) \,\,\,\forall \, j <i\ ,
\end{eqnarray*}
and the assignment $ g.(a_i) := ((g.a)_j) $ is easily seen to give a well-defined action map $\, \GL_{\infty} \times  A_\infty \to A_\infty \,$. We write
$ A^{\GL_{\!\infty}}_{\infty} $ for the corresponding invariant subalgebra of $ A_\infty $.

Now, if \eqref{condin} holds, the morphisms $ \mu_{ji}:\, A_j \to A_i $ restrict to the invariant subalgebras so that we may form the inverse system $\{A_i^{\GL_i}\}_{i\in \mathbb{I}}\,$. In this case, Theorem~3.4 of \cite{T-TT} asserts that
\begin{equation}
\la{eS1.2}
 A_{\infty}^{\GL_{\infty}} \,\cong\, \varprojlim_{i \in \mathbb I} A_i^{\GL_i} \mathrm{.}
\end{equation}
The isomorphism \eqref{eS1.2} allows one to equip $ A^{\GL_\infty}_{{\infty}} $ with a natural topology: namely, we put first the discrete topology on each $ A_i^{\GL_i} $ and equip
$\,\prod_{i \in \mathbb{I}} A_i^{\GL_i} \,$ with the product topology; then, identifying
$ A^{\GL_\infty}_{{\infty}} $ with a subspace in
$\,\prod_{i \in \mathbb{I}} A_i^{\GL_i} \,$ via \eqref{eS1.2}, we put on $ A^{\GL_\infty}_{{\infty}} $ the induced topology. The
corresponding topological DG algebra will be denoted $ A^{\GL}_{{\infty}} $.

\subsubsection{Example}
\la{S1.1.3}
Consider the polynomial algebra $ A_n := k[x_{ij};\, 1 \leq i,j \leq n] $ on which $\GL_n(k) $ acts via the coadjoint representation. Put $ X := (x_{ij})\, \in \, \M_n( A_n) $ and define the algebra map
$$k[x] \rar\M_n(A_n), \,\,\,\,\,\,\, x \mapsto X \,\mathrm{.}$$ Combining this with the matrix trace, we get a linear map $ \mathrm{Tr}_n: k[x] \rar A_n $ whose image is clearly in $ A_n^{\GL_n} $. Now, it is easy to check that
$$\mu_{n,n-1}\,:\, A_n \rar A_{n-1}, \,\,\,\,\, x_{ij} \mapsto (1-\delta_{in})(1-\delta_{jn})x_{ij} \,
$$
defines an algebra map for each $ n \ge 1 $ (here $ \delta $ denotes the Kronecker delta). With these maps, the algebras $ \{A_n\}_{n \in \N} $ form an inverse system with $\GL_n$-action satisfying the compatibility conditions \eqref{condin}. Hence, \eqref{eS1.2} applies in this situation. Furthermore, one easily verifies that
$$
\mu_{n,n-1} \circ \Tr_n\,=\,\Tr_{n-1}\ .
$$
Hence, the traces $\mathrm{Tr}_n $ assemble together to give a linear map $\mathrm{Tr}_{\infty}\,:\, k[x] \rar A_{\infty}^{\GL}$. By the Procesi Theorem, the subalgebra of $A_n^{\GL_n}$ generated by the image of $\mathrm{Tr}_n $ coincides with $A_n^{\GL_n}$ for each $ n $. It follows from this and the discussion in Section~\ref{S1.1} that the subalgebra of $A_{\infty}^{\GL}$ generated by the image of $\mathrm{Tr}_{\infty}$ is {\it dense} in $ A_{\infty}^{\GL}$.
As explained in \cite[ Example 1.2]{T-TT}, this is the best one can
achieve in this situation.
Finally, as suggested by its notation, the algebra $ A_n $ arises from applying the functor $(\mbox{--}\,)_n$ to $ A = k[x] $ (see \cite[Sect.~2.4]{BKR}).

\subsection{Stabilization} \la{S2.1}
Fix $A \in\DGA_{k/k} $, and let $ \bar{A} $ be its augmentation ideal
viewed as a non-unital algebra in $ \DGA $.
For each $ n \ge 1 $, bordering a matrix in $ \M_n(k) $ by $0$'s on the right and on the bottom gives an embedding
$\,\M_n(k) \into \M_{n+1}(k) \,$ of non-unital algebras. As a result, for every commutative DG algebra $ B \in \cDGA_k $, we get a map of sets
\begin{equation}
\la{isoun0}
\Hom_{\DGA}(\bar{A},\, \M_n(B)) \to \Hom_{\DGA}(\bar{A},\,\M_{n+1}(B))
\end{equation}
defining a natural transformation of functors from $ \cDGA_k $ to $ \Sets $. Since $B$'s are unital in \eqref{isoun0} and $ A $ is augmented, the restriction maps
\begin{equation}
\la{isoun1}
\Hom_{\DGA_k}(A,\,\M_n(B)) \stackrel{\sim}{\to} \Hom_{\DGA}(\bar{A},\,\M_n(B)) \ ,\quad
\varphi \mapsto \varphi|_{\bar{A}}
\end{equation}
are isomorphisms for all $ n \in \N $. Combining \eqref{isoun0} and \eqref{isoun1}, we thus have a direct system of natural transformations of functors $\, \Rep_n(A) \to \Rep_{n+1}(A)\,$,
which, by Theorem~\ref{nS2.1t1}, yields an inverse system of algebra maps $\,\{\mu_{n+1, n}: A_{n+1} \to A_n \}_{n\in\N} \,$ in $ \cDGA_k $. It is easy to check that the maps $ \mu_{n+1, n} $ respect augmentations (and hence, are actually morphisms in $ \cDGA_{k/k} $)
for all $ n $. Taking limit, we define the (augmented)
commutative DG algebra
$$
A_{{\infty}} := \varprojlim_{n\,\in\,\N} A_n \ .
$$

Next, we recall (see Section~\ref{glinv}) that the group $ \GL_n(k) $ naturally acts on $ A_n $ for each $ n \in \N $. It is easy to check that these actions together with the maps $\,\mu_{n+1, n}: A_{n+1} \to A_n\,$ satisfy the compatibility conditions
\eqref{condin}. Thus, we can apply the results of Section~\ref{S1.1.2} (in particular, \eqref{eS1.2}) to define the (topological) DG algebra
$$
A_{{\infty}}^{\GL} := A_{\infty}^{\GL_{\infty}} \cong \varprojlim\,A_n^{\GL_n} \ .
$$
This, in turn, defines the functor
\begin{equation}
\la{eS2.1.4}
(\mbox{--})_{{\infty}}^{\GL}\,:\ \DGA_{k/k} \rar \cDGA_{k/k} ,\,\,\,\,\, A \mapsto A_{{\infty}}^{\GL} \mathrm{.}
\end{equation}

\subsubsection{Trace subalgebra}
\la{S2.1.1}
Recall the cyclic functor $\,\FT:\,\DGA_{k/k} \to \Com_{\,k} \,$
introduced in Section~\ref{trmaps}. As explained in that section,
for any augmented DG algebra $ A $, the natural trace maps $\Tr_n(A)_{\bullet}\,:\, A_{\natural} \to A_n^{\GL} $
define a morphism of functors $ \Tr_n:\,\FT \to (\,\mbox{--}\,)^\GL_n $ from $ \DGA_{k/k} $
to $ \Com_{\,k} $. Now, it is easy to check that the following diagram
commutes for all $ n \in \N \,$:
\[
\begin{diagram}[small, tight]
  &       &     \FT(A)     &        & \\
  &\ldTo^{\Tr_{n+1}(A)}  &           &  \rdTo^{\Tr_n(A)} &  \\
A_{n+1}^{\GL} &       & \rTo^{\mu_{n+1, n}}  &        &  A_n^{\GL}
\end{diagram}
\]
Hence, by the universal property of inverse limits, there is a morphism of complexes $\, \Tr_{\infty}(A)_\bullet :\, \FT(A) \to A^{\GL}_{{\infty}}\,$ that factors $ \Tr_{n}(A)_\bullet $ for each $ n \in \N $. We extend this morphism to a homomorphism of DG algebras
\begin{equation}
\la{trinf}
\bSym \Tr_{\infty}(A)_\bullet :\ \bL[\FT(A)] \to A^{\GL}_{{\infty}}\ .
\end{equation}
The following lemma follows immediately from Theorem~\ref{pS1.3.1}.
\blemma
\la{lS2.1.1}
The map \eqref{trinf} is topologically surjective, i.e. its image is dense in $ A^{\GL}_{{\infty}} $.
\elemma

We denote the image of \eqref{trinf} by $ A_\infty^{\Tr}  $ and call
it the {\it trace subalgebra} of $ A_{\infty} $. As explained in
Example~\ref{S1.1.3}, the inclusion $ A_\infty^{\Tr} \subset A^{\GL}_{{\infty}} $ is proper in general. The fact that the algebra maps \eqref{trinf} are natural in $ A $ implies that
$ A \mapsto A_{\infty}^{\mathrm{Tr}} $  defines a functor
$\, (\,\mbox{--}\,)^{\Tr}_{\infty}\,:\ \DGA_{k/k} \to \cDGA_k\,$,
and \eqref{trinf} then gives a morphism of functors
\begin{equation*}
\la{morfun}
\bSym \Tr_{\infty}(\,\mbox{--}\,)_\bullet :\ \bL[\FT(\,\mbox{--}\,)] \to (\,\mbox{--}\,)^{\Tr}_{\infty}\ .
\end{equation*}

\subsubsection{{\sl DG} Hopf algebra structure}
\la{S2.1.2} Recall that the matrix rings $ \{\M_n(k)\}_{n \in \N} $
with natural inclusions $ \M_n(k) \into \M_{n+1}(k) $
form an inductive system of non-unital algebras. We set $\, \M_{\infty}(k) := \varinjlim\,\M_n(k)\,$. By definition, $ \M_{\infty}(k) $ is a (non-unital) algebra of all finite matrices\footnote{$\, \M_{\infty}(k) $ can be identified with the subalgebra of $ \End(k^\N)$ consisting of endomorphisms with finite-dimensional image and finite-codimensional kernel.}. More generally, for $ B \in \cDGA_k $, we define
$\, \M_{\infty}(B) := B \otimes \M_\infty(k) \cong \varinjlim\,\M_n(B) \,$ and observe that, for any $ A \in \DGA_{k/k} $, there is an adjunction
\begin{equation}
\la{eS2.1.2}
\Hom_{\DGA}(\bar{A},\, \M_{\infty}(B)) \,=\, \mathcal{H}om_{ \cDGA_k}(A_{{\infty}},\, B) \ ,
\end{equation}
where `$\,\mathcal{H}om \,$' denotes the set of {\it continuous} maps from $A_{\infty} $ equipped with the inverse limit topology (see
Section~\ref{S1.1.1}) to $ B \in \cDGA_k $ equipped with the discrete topology.

Next, recall that there is a natural way to extend the direct sum
of matrices of fixed size $ \M_n(B) \times \M_m(B) \to \M_{n+m}(B) $ to infinite matrices in $\M_{\infty}(B)$. This operation (sometimes called `Whitney sum')
\begin{equation*}
\oplus:\ \M_{\infty}(B) \times \M_{\infty}(B) \to \M_{\infty}(B)
\end{equation*}
is defined by the following rule ({\it cf.} \cite[Sect.~10.2.12]{L}):
\begin{equation}\la{wsum}
\left(\begin{array}{cccc}
\star & \star &  \star & \cdots \\
\star & \star &  \star & \cdots \\
\star & \star &  \star & \cdots \\
\vdots     & \vdots     & \vdots      & \ddots
\end{array} \right) \quad \bigoplus \quad
\left(\begin{array}{cccc}
\ast & \ast &  \ast & \cdots \\
\ast & \ast &  \ast &  \cdots \\
\ast & \ast &  \ast & \cdots \\
\vdots   & \vdots    & \vdots     & \ddots
\end{array} \right)
\ = \
\left(\begin{array}{ccccccc}
\star & 0    &  \star & 0    &  \star &    0 & \cdots\\
0     & \ast &  0     & \ast &      0 & \ast &  \cdots \\
\star & 0    &  \star & 0    &  \star &    0 & \cdots\\
0     & \ast &  0     & \ast &      0 & \ast &  \cdots \\
\star & 0    &  \star & 0    &  \star &    0 & \cdots\\
0     & \ast &  0     & \ast &      0 & \ast &  \cdots \\
\vdots    & \vdots     & \vdots      & \vdots   & \vdots   & \vdots    & \ddots
\end{array} \right) \ .
\end{equation}
It is easy to see that $\,\oplus\,$ is a morphism of non-unital algebras. Hence, for any $ A \in \DGA_{k/k} $, it defines a
morphism of functors $\,\cDGA_k \to \mathtt{Sets}\,$:
$$
\Hom_{\DGA}(\bar{A},\, \M_{\infty}(\,\mbox{--}\,)) \times \Hom_{\DGA}(\bar{A}, \, \M_{\infty}(\,\mbox{--}\,)) \,\to\,  \Hom_{\DGA}(\bar{A},\,  \M_{\infty}(\,\mbox{--}\,))\ .
$$
Combined with \eqref{eS2.1.2}, this last morphism becomes
\begin{equation}
\la{natrr}
 \mathcal{H}om_{ \cDGA_k}(A_{\infty} \otimes A_{\infty}, \,\mbox{--}\,) \,\to \, \mathcal{H}om_{ \cDGA_k}(A_{{\infty}}, \,\mbox{--}\,)\ ,
\end{equation}
where $ A_{{\infty}} \otimes A_{{\infty}} $ is the coproduct in $ \cDGA_k $ equipped with the inductive topology. Now, by Yoneda's Lemma, \eqref{natrr} defines a (continuous) homomorphism of DG algebras
\begin{equation}
\la{copr}
\Delta\,:\,A_{{\infty}} \rar A_{{\infty}} \otimes A_{{\infty}}\ .
\end{equation}
 \begin{lemma}
  \la{lS2.1.2.2}
The map \eqref{copr} restricts to the trace subalgebra $\,A_{\infty}^{\mathrm{Tr}},$ making it a commutative cocommutative {\sl DG} Hopf algebra. Furthermore, $\bSym \mathrm{Tr}_{\infty}(A)_{\bullet} $ becomes a homomorphism of {\sl DG} Hopf algebras.
\end{lemma}
 \begin{proof}
Recall that $A_{\infty}^{\mathrm{Tr}}$ is generated by the image of $\mathrm{Tr}_{\infty}(A)_{\bullet}$ in $A_{\infty} $, and it is easy to check that
\begin{equation}
 \la{lS2.1.2.1}
\Delta\left[\mathrm{Tr}_{\infty}(A)_{\bullet}(x)\right]\,=\, 1 \otimes \mathrm{Tr}_{\infty}(A)_{\bullet}(x)+\mathrm{Tr}_{\infty}(A)_{\bullet}(x) \otimes 1 \ ,\quad \forall\, x \in \FT(A)\ .
\end{equation}
Hence \eqref{copr} restricts to a morphism $\,\Delta: A_{\infty}^{\mathrm{Tr}} \rar A_{\infty}^{\mathrm{Tr}} \otimes A_{\infty}^{\mathrm{Tr}}\,$ in $\cDGA_k$.
Moreover, by \eqref{lS2.1.2.1}, there is a commutative diagram in $\cDGA_k\,$:
\begin{equation}\la{diagco}
\begin{diagram}[small]
\bSym[\mathcal{C}(A)]  & \rOnto^{\bSym \mathrm{Tr}_{\infty}(A)} & A_{\infty}^{\mathrm{Tr}}\\
 \dTo^{\Delta}& & \dTo_{\Delta}\\
 \bSym[\mathcal{C}(A)] \otimes \bSym[\mathcal{C}(A)] & \rOnto^{\ \bSym \mathrm{Tr}_{\infty}(A) \otimes \bSym \mathrm{Tr}_{\infty}(A)\ } & A_{\infty}^{\mathrm{Tr}} \otimes A_{\infty}^{\mathrm{Tr}}
\end{diagram}
\end{equation}
where $\Delta$ on the left is the unique coproduct on $\bSym[\mathcal{C}(A)] $ for which the primitive elements are precisely those of $\mathcal{C}(A)$. Lemma~\ref{lS2.1.2.2} now follows from the observation that $\,\Delta\,$ makes $\bSym[\mathcal{C}(A)] $ a commutative cocommutative DG Hopf algebra.
 \end{proof}
\subsubsection{Remark} The algebra $ \M_{\infty}(B) $ has the property that $\,\M_n(k) \otimes \M_{\infty}(B) \cong \M_{\infty}(B)\,$ for any $ n \ge 1 $ (see, e.g., \cite[Chap.~3, Sect.~4.2]{FF}). Choosing any such isomorphism for $ n = 2 $, we may define the direct sum
in $ \M_\infty(B) $ by the rule
$$
\oplus:\, \M_\infty(B) \times \M_{\infty}(B) \to \M_2(k) \otimes \M_{\infty}(B) \cong \M_{\infty}(B) \ , \quad
(a_1, a_2)\,\mapsto \,
\left(\!\begin{array}{cc}
a_1   & 0    \\
0     & a_2
\end{array} \!\right)\ ,
$$
which is different from \eqref{wsum}.
The corresponding map \eqref{copr} will also satisfy \eqref{lS2.1.2.1}
(and hence, the conditions of Lemma~\ref{lS2.1.2.2}), and \eqref{copr} will certainly depend on a particular choice of $ \oplus \,$. However, the proof of Lemma~\ref{lS2.1.2.2} shows that there is at most one coproduct $ \Delta $ on $ A^{\Tr}_\infty $ making \eqref{diagco} commutative. Hence, the restriction of \eqref{copr} to $ A^{\Tr}_\infty $ is independent of any choices. By Lemma~\ref{lS2.1.2.2},
$\, A^{\Tr}_\infty $ is thus equipped with a canonical Hopf algebra structure.

\subsection{Deriving the trace algebra}
\la{S2.2}
Recall the functor $\,
(\,\mbox{--}\,)^{\Tr}_{\infty}\,:\ \DGA_{k/k} \to \cDGA_k\,$ defined in Section~\ref{S2.1.1}.
\begin{theorem} \la{tS2.2}
$(a)$ The functor $(\mbox{--})_{\infty}^{\mathrm{Tr}}$ has a total left derived functor
$$\L  (\mbox{--})_{\infty}^{\mathrm{Tr}}\,:\, \Ho(\DGA_{k/k}) \rar \Ho(\cDGA_k), \,\,\,\,\,\,\, A \mapsto
(QA)_{\infty}^{\mathrm{Tr}} \mathrm{.} $$

$(b)$ The functor $(\mbox{--})_{{\infty}}^{\GL}$ has a total left derived functor
$$\L  (\mbox{--})_{{\infty}}^{\GL}\,:\, \Ho(\DGA_{k/k}) \rar \Ho(\cDGA_k), \,\,\,\,\,\,\, A \mapsto
(QA)_{\infty}^{\GL} \mathrm{.} $$

$(c)$ Part $(b)$ also holds with $(\mbox{--})_{{\infty}}^{\GL}$ replaced by $(\mbox{--}\,)_{{\infty}}$.
\end{theorem}

\begin{proof}
 By Brown's Lemma (see \cite[Lemma~9.9]{DS}), it suffices to prove that $(\mbox{--})_{\infty}^{\mathrm{Tr}}$ takes any acyclic cofibration between
cofibrant objects in $\DGA_{k/k}$ to a weak equivalence. Let $i:A \rar B$ be an acyclic cofibration in $\DGA_{k/k}$
with $A$ cofibrant. Then, there exists $p:B \rar A$ in $\DGA_{k/k}$ such that $pi=\id_A$ and $ip$ is homotopic to
$\id_B$. By Proposition~\ref{pA2.1}, there exists a morphism $h:B \rar B \otimes \Omega$ in $\DGA_{k/k}$ such that
$h(0)=\id_B$ and $h(1)=ip$ and $h(\bar{B}) \subset \bar{B} \otimes \Omega$ (where $\Omega$ denotes the de Rham algebra of the affine line $\mathbb A^1_k$ and $\bar{B}$ denotes the augmentation ideal of $B$).

By \cite[Lemma~2.5]{BKR}, one has a morphism $h_n^{\GL}\,:\,B_n^{\GL} \rar B_n^{\GL} \otimes \Omega$ for each $n \in \mathbb N$
such that $h_n^{\GL}(0)=\id_{B_n^{\GL}}$ and $h_n^{\GL}(1)= (ip)_n^{\GL}$. One also has a morphism
$\mathcal{C}(h)\,:\, \mathcal{C}(B) \rar \mathcal{C}(B) \otimes \Omega$ with $\mathcal{C}(h)(0)=\id$ and $\mathcal{C}(h)(1)
= \mathcal{C}(ip)$. The proof that $\bSym \mathrm{Tr}_n(\mbox{--})_{\bullet} $ is a natural transformation can be modified
without difficulty to show that the following diagram commutes for all $n$.
\begin{equation}
 \la{eS2.2}
\begin{diagram}[small]
   \bSym[\mathcal{C}(B)] &\rTo^{\bSym \mathrm{Tr}_n(B)} & B_n^{\GL} \\
     \dTo^{\bSym \mathcal{C}(h) }  & & \dTo_{h_n^{\GL}}\\
     \bSym[\mathcal{C}(B)] \otimes \Omega &\rTo^{\ \bSym \mathrm{Tr}_n(B) \otimes \,\id\ } & B_n^{\GL} \otimes \Omega
  \end{diagram}
\end{equation}
For a fixed $n$, denote the diagram~\eqref{eS2.2} by $\mathcal D_n$. Then, the map $\mu_{n+1,n}: B_{n+1}^{\GL}  \rar B_n^{\GL}$ induces a morphism $\mathcal D_{n+1} \rar
\mathcal D_n$ of commutative squares. Hence, the maps $h_n^{\GL}$ form a morphism of inverse systems which
yields a morphism $ h_{\infty}^{\GL}\,:\,B_{\infty}^{\GL} \rar \varprojlim_n (B_n^{\GL} \otimes \Omega)$. Using the diagram \eqref{eS2.2}, it is easy to check that $h_{\infty}^{\GL}$ takes elements of $B_{\infty}^{\Tr}$ to elements of $B_{\infty}^{\Tr} \otimes \Omega$. Thus we get a morphism $ h_{\infty}^{\Tr}\,:\,B_{\infty}^{\Tr} \rar B_{\infty}^{Tr} \otimes \Omega $ in $\cDGA_k$.
Clearly, $h_{\infty}^{\mathrm{Tr}}(0)= \id$ and $h_{\infty}^{\mathrm{Tr}}(1)=(ip)_{\infty}^{\mathrm{Tr}}$.
Hence, $(ip)_{\infty}^{\mathrm{Tr}}$ induces the identity on homology. Since $(pi)_{\infty}^{\mathrm{Tr}}$ does so as well,
$(i)_{\infty}^{\mathrm{Tr}}$ is a weak equivalence. This proves part $(a)$ of the desired theorem.

Part $(b)$ requires less work: let $i:A \rar B$ be an acyclic cofibration in $\DGA_{k/k}$ between cofibrant objects.
Since the forgetful functor $\DGA_{k/k} \rar \DGA_k$ preserves cofibrations and weak equivalences, the family of maps $ \{i_n^{\GL}\,:\,A_n^{\GL} \rar B_n^{\GL}\}_{n \in \N}$ defines a morphism of inverse systems, each term of
which is a quasi-isomorphism by the proof of
\cite[Theorem 2.6]{BKR}. Hence, $ \{\,\mathrm{H}_{\bullet}(i_n^{\GL})\,:\,
\mathrm{H}_{\bullet}(A_n^{\GL}) \rar \mathrm{H}_{\bullet}(B_n)^{\GL}\,\}_{n \in \N} $ is an isomorphism of inverse systems of (graded)
$k$-vector spaces. By Theorem~\ref{pS1.3.1} and the fact that $\mathrm{Tr}_n(A)_{\bullet}=\mu_{n+1,n} \circ \mathrm{Tr}_{n+1}(A)_{\bullet}$ for all $n \in \N$,
 the morphism $\mu_{n+1,n}:A_{n+1}^{\GL} \rar A_n^{\GL}$ is surjective
(and similarly for $B$). Hence, the inverse system $(A_n^{\GL})$ (resp., $(B_n^{\GL})$) is a tower of complexes satisfying
the Mittag-Leffler condition. By \cite[Theorem 3.5.8]{W}, we have a morphism of short exact sequences
$$
\begin{diagram}
0 & \rTo & \varprojlim^{\!1} \mathrm{H}_{p+1}(A_n^{\GL}) & \rTo & \mathrm{H}_p(A_{{\infty}}^{\GL}) & \rTo & \varprojlim \mathrm{H}_{p}(A_n^{\GL}) & \rTo & 0 \\
 & &        \dTo^{\varprojlim^{\!1} \mathrm{H}_{p+1}(i_n^{\GL})} & & \dTo^{\mathrm{H}_p(i_{{\infty}})} & & \dTo^{\varprojlim \mathrm{H}_{p}(i_n^{\GL})}  &  & \\
 0 & \rTo &     \varprojlim^{\!1} \mathrm{H}_{p+1}(B_n^{\GL}) & \rTo & \mathrm{H}_p(B_{{\infty}}^{\GL}) & \rTo & \varprojlim \mathrm{H}_{p}(B_n^{\GL})  & \rTo & 0 \\
\end{diagram}
$$
for each $ p \in \Z $. Since the rightmost and leftmost vertical maps are isomorphisms of $k$-vector spaces, so is the one in the middle. Thus $i_{{\infty}}$ is a weak equivalence. Brown's Lemma then completes the proof of $(b)$.

The proof of part $(c)$ is a trivial modification of that of $(b)$. Here, the fact that $\mu_{n+1,n}\,:\,A_{n+1} \rar A_n$ is surjective (for cofibrant $A$) follows from \cite[Theorem~2.8]{BKR}.
\end{proof}

\begin{definition}
For $A \in \Alg_{k/k}$, we define
$\,
\DRep_{\infty}(A)^{\mathrm{Tr}} := \L(A)_{\infty}^{\mathrm{Tr}}\,$
in $ \,\Ho(\cDGA_k)\,$ and will refer to
the homology $\H_{\bullet}(A, \infty)^{\mathrm{Tr}}\,:=\,
\text{H}_{\bullet}[\DRep_{\infty}(A)^{\mathrm{Tr}}] $ as the {\it stable representation homology} of $A$.
\end{definition}
\subsection{Main theorem}
\la{S2.3}
The following theorem is the main result of this paper.
\begin{theorem}
\la{tS2.3}
 $\bSym\mathrm{Tr}_{\infty}(\mbox{--})_{\bullet} $ induces an isomorphism of  functors
 $\, \bSym(\L  \mathcal{C}) \,\stackrel{\sim}{\to}\, \L (\,\mbox{--}\,)_{\infty}^{\mathrm{Tr}} \,$.
\end{theorem}
\begin{proof}
To begin with, note that $\,\bSym \mathrm{Tr}_{\infty}(\mbox{--})_{\bullet}\, $ is indeed a morphism of functors. This follows from the fact that $ \mathrm{Tr}_{\infty}(\mbox{--})_{\bullet} $ induces a morphism  of the derived functors $ \L\FT \to \L(\mbox{--})_{\infty}^{\mathrm{Tr}} $ from $ \Ho(\DGA_{k/k}) $ to $ \D(k) $ ({\it cf.} \cite[Remark~A.5.1]{BKR}). Since any $ A \in \DGA_{k/k} $ has a free resolution (see \cite[Remark 2.2.5]{H}), Theorem~\ref{tS2.3} follows from Proposition~\ref{pS4.3} below.
\end{proof}
\begin{prop}
\la{pS4.3} For any free $R \in \DGA_{k/k}$, the map $\, \bSym \mathrm{Tr}_{\infty}(R)_{\bullet}: \bSym[\mathcal{C}(R)] \stackrel{\sim}{\to} R_{\infty}^{\mathrm{Tr}} $ is an isomorphism.
\end{prop}

The rest of this subsection is devoted to the proof of Proposition~\ref{pS4.3}.

\subsubsection{Preliminary lemma} Recall that, by Lemma~\ref{lS2.1.2.2},  $ R_{\infty}^{\Tr} $ has a natural Hopf algebra structure with coproduct satisfying \eqref{lS2.1.2.1}. We write $ \mathtt{Prim}\,(R_{\infty}^{\mathrm{Tr}}) $ for the subspace of primitive elements of $ R_{\infty}^{\Tr} $.
\begin{lemma} \la{lS4.3.1} $\mathrm{Tr}_{\infty}(R)_{\bullet}\,:\,\mathcal{C}(R) \rar R_{\infty}^{\mathrm{Tr}}$ is a degreewise surjection onto $\mathtt{Prim}\,(R_{\infty}^{\mathrm{Tr}})$.

\end{lemma}

\begin{proof}
 By \eqref{lS2.1.2.1}, the image of $\mathrm{Tr}_{\infty}(R)_{\bullet} $ is contained in $\mathtt{Prim}\,(R_{\infty}^{\mathrm{Tr}})$. We need to show that every primitive element of $R_{\infty}^{\mathrm{Tr}}$ is in the image of $\mathrm{Tr}_{\infty}(R)_{\bullet} $. For $ q \in \Z $, let $R_{\infty}^{\mathrm{Tr}, q}$ denote the image of the map
$$\bSym^{\!q} \mathrm{Tr}_{\infty}(R)_{\bullet} :\bSym^{\!q}[\mathcal{C}(R)] \rar R_{\infty}^{\mathrm{Tr}}\,.$$
Then $R_{\infty}^{\mathrm{Tr},q}$ is contained in the eigenspace corresponding to the eigenvalue $\,2^q\,$ of the second Adams operation $\psi^2:=\mu \circ \Delta$ of the graded Hopf algebra $R_{\infty}^{\mathrm{Tr}}$. By definition, any $ f \in R_{\infty}^{\mathrm{Tr}}$ can be written as a linear combination $f=\sum_{q=0}^N f_q $ with $f_ q $ in $R_{\infty}^{\mathrm{Tr}, q}$. If $f$ is primitive, $\psi^2(f)=2f=\sum_{q=0}^N 2^q f_q $. Hence, $f_q =0$ for $q \neq 1$, which implies that $f$ lies in the image of $\mathrm{Tr}_{\infty}(R)_{\bullet}$. Hence, $$ \mathrm{Tr}_{\infty}(R)_{\bullet}\,:\,\mathcal{C}(R)  \rar \mathtt{Prim}\,(R_{\infty}^{\mathrm{Tr}})$$ is surjective.
\end{proof}

\subsubsection{Proof of Proposition~\ref{pS4.3}}

Since $\bSym \mathrm{Tr}_{\infty}(R)_{\bullet} $ is a morphism in $\cDGA_k$, it suffices to forget the differentials and verify that $\bSym \mathrm{Tr}_{\infty}(R)_{\bullet} $ is an isomorphism of graded (commutative) algebras. We may therefore forget the differential on $R$. Moreover, since $R$ is the direct limit of its finitely generated graded subalgebras, it suffices to prove Proposition~\ref{pS4.3} for $R$ finitely generated. Further, it is enough to check that $\bSym \mathrm{Tr}_{\infty}(R)_{\bullet} $ is an isomorphism of $\Z_2$-graded commutative algebras. We may therefore, assume without loss of generality that the generators of $R$ are in degrees $0$ and $1$ (see Section~\ref{1.3.2}). By Lemma~\ref{lS4.3.1} and the structure theorem for commutative, cocommutative DG Hopf algebras, the following lemma implies Proposition~\ref{pS4.3}.

\begin{lemma} \la{lS4.3.2}
If $R$ is a finitely generated free graded algebra with generators in degrees $0$ and $1$, then the map $\mathrm{Tr}_{\infty}(R)_{\bullet}\,:\,\mathcal{C}(R)  \rar R_{{\infty}}$
is injective.
\end{lemma}

\begin{proof}
To avoid complicated notation we assume that $ R $ is generated by two elements, say $x$ and $y$, with $|x| =0 $ and $|y|=1 $. For a general  $R$, the argument will be completely similar. If $ R=k\langle x,y\rangle $,
then $\mathcal{C}(R)$ is spanned (as a $k$-vector space) by the set $S$ of cyclic words $w$ in the symbols $x$ and $y$ such that $\mathrm{N}(w) \,\neq \,0$ (see Section~\ref{SI.0}). Hence, any element of $\mathcal{C}(R)$ is of the form
$\sum_{w \in S} c(w)w$ where $c$ is some $k$-valued function on $S$ with finite support.

Let $\mathcal A$ denote the graded commutative algebra $ k[u_1,\ldots ,u_n,\ldots]\,$, where the $u_i$ are variables in homological degree $1$. Note that one has an evaluation homomorphism $R_{{\infty}} \rar \mathcal A$ for every choice
$(x_{ij}) \in \mathfrak{gl}_{\infty}(k)$ and $(y_{ij}) \in \mathfrak{gl}_{\infty}(\mathcal A_1)$. Now, if
$f \in \mathcal{C}(R)$ such that $\mathrm{Tr}_{\infty}(R)_{\bullet}(f)=0 \in R_{{\infty}}$, then
 $\,\mathrm{Tr}\,[f(X,Y)]=0\,$ for any $X \in \mathfrak{gl}_{\infty}(k)$ and
$Y \in \mathfrak{gl}_{\infty}(\mathcal A_1)$. But, if $f =\sum_{w \in S} c(w)w$, then
$$\sum_{w \in S} c(w)\,\mathrm{Tr}\,[w(X,Y)] =0$$ for all $X \in \mathfrak{gl}_{\infty}(k)$ and
$ Y \in \mathfrak{gl}_{\infty}(\mathcal A_1) $. Corollary~\ref{cSI.1} (see Appendix) then implies that $\,f=0\,$, proving the desired lemma.
\end{proof}

\subsubsection{Corollary}
\la{S2.3a}
Let $ A \in \Alg_{k/k} $ be an ordinary (augmented) $k$-algebra.
Then, by Theorem~\ref{ftt}$(b)$, there is an isomorphism of graded
vector spaces $\,\H_{\bullet}[\L\FT(A)] \cong \rHC_{\bullet}(A) \,$. In combination with this fact Theorem~\ref{tS2.3} implies
\begin{corollary}
\la{cS2.3.2}
The map \eqref{trinf} induces an isomorphism of graded commutative Hopf algebras
\begin{equation}
\la{anlqt}
\bSym \mathrm{Tr}_{\infty}(A)_{\bullet} :\,\bSym[\rHC_{\bullet}(A)] \,\stackrel{\sim}{\to} \,\mathrm{H}_{\bullet}(A,\infty)^{\mathrm{Tr}} \mathrm{.}
\end{equation}
\end{corollary}

\subsubsection{Remarks}
\la{S2.3.1}
1.\, Corollary~\ref{cS2.3.2} implies that the cyclic homology of an augmented algebra is determined by its stable representation homology. This answers a question of V.~Ginzburg how the cyclic homology of an associative algebra can be recovered from its representation varieties. A different answer to this question is proposed in \cite{GiS}.

2.\, Theorem~\ref{tS2.3} and its Corollary~\ref{cS2.3.2} hold, {\it mutatis mutandis}, if we replace $\, k \,$ by any finite-dimensional semi-simple $ k$-algebra $S$. If $ A = T_S(V) $ is a free (tensor) algebra over $S$, then the ($S$-relative) cyclic homology
$ \rHC_{n}(A) $ of $A$ vanishes for all $ n \ge 1 $ and so does its
representation homology (see \cite[Theorem 21]{BFR}). In this special case, our Theorem~\ref{tS2.3} boils down to \cite{Gi}, Corollary~4.3.
However, this last corollary (as well as Proposition~4.2 from which it follows) are stated in \cite{Gi} incorrectly: the trace map $ \mathtt{tr}_\infty $ appearing in the statements of these results is injective but not an isomorphism in general ({\it cf.} Example~\ref{S1.1.3} above).

3.\, Finally, the assumption that $A$ is augmented is essential for the result of Corollary~\ref{cS2.3.2}.
For example, if $A = A_1(k) $ is the first Weyl algebra over $k$, then
$ \H_\bullet(A, n) $ is trivial for all $ n \ge 0 $ (see \cite[Example~2.1]{BKR}), hence so is $ \H_{\bullet}(A,\infty)^{\mathrm{Tr}} $. On the other hand, the
cyclic homology of $ A_1(k) $ is nonzero (see, e.g., \cite[Exercise~E.3.1.4]{L}).

\section{Koszul Duality and the Obstruction Complex}
\la{KDOC}
In this section, we will show that the stable representation homology of an augmented algebra
$A$ is Koszul dual to the Lie algebra homology of $ \gl_\infty(A) $. By Koszul duality we will mean a duality between algebras and coalgebras, which manifests itself in an equivalence of approriately defined derived categories of representations. For a detailed exposition of the Koszul duality theory in this general framework we refer to Chapter~2 of the recent monograph
\cite{LV} (see also \cite{Po}). We will only briefly recall basic definitions.

\subsection{Acyclic twisting cochains}
\la{actw}
We will work with augmented counital DG coalgebras $ C $ which are {\it conilpotent}
in the sense that
\begin{equation}
\la{cocom}
\bar{C} = \bigcup_{n \ge 2}\, \Ker\,[\,C \xrightarrow{\Delta^{(n)}} C^{\otimes n}
\onto \bar{C}^{\otimes n}\,]
\end{equation}
where $ \Delta^{(n)} $ denotes the $n$-th iteration of the comultiplication map
$\, \Delta_C: C \to C \otimes C \,$ and
$\, \bar{C} \,$ is the cokernel of the augmentation map $ \varepsilon_C: k \to C \,$.
We denote the category of such coalgebras by $ \DGC_{k/k} $.
Given an algebra $ R \in \DGA_{k/k} $ and a coalgebra $ C \in \DGC_{k/k} $,
we define a {\it twisting cochain} $\, \tau: C \to R $ to be a linear map of
degree $-1$ satisfying
$$
d_R \,\tau + \tau\,d_C + m_R \,(\tau \otimes \tau)\,\Delta_C = 0\ ,\quad
\varepsilon_R \,\tau\, = \, \tau\,\varepsilon_C = 0\ ,
$$
where $ d_R $ and $ d_C $ are the differentials on $ R $ and $ C $ and $ m_R $ is the
multiplication map on $ R $.
We write $ \Tw(C,R) $ for the set of all twisting cochains from $C$ to $R$. It is easy to show
that, for a fixed algebra $ R $, the functor
$$
\Tw(\mbox{--}, R):\,\DGC_{k/k} \to \Sets\ ,\quad C \mapsto \Tw(C,R)\ ,
$$
is representable; the corresponding coalgebra $ \bB(R) \in \DGC_{k/k} $ is called the {\it bar construction} of $ R \,$: it is defined as the tensor coalgebra $ T_k(R[1]) $ with differential lifting $ d_R $ and $ m_R $. Dually, for a fixed coalgebra $ C $, the functor
$$
\Tw(C, \mbox{--}):\,\DGA_{k/k} \to \Sets\ ,\quad R \mapsto \Tw(C,R)\ ,
$$
is corepresentable; the corresponding algebra $ \bOmega(C) \in \DGA_{k/k} $ is called the
{\it cobar construction} of $ C \,$: it is defined as the tensor algebra $ T_k(C[-1]) $
with differential lifting $ d_C $ and $ \Delta_C $.
Thus, we have canonical isomorphisms
\begin{equation}
\la{fism}
\Hom_{\DGA_{k/k}}(\bOmega(C), R)\, = \,\Tw(C,R)\, =\, \Hom_{\DGC_{k/k}}(C, \bB(R))
\end{equation}
showing that $ \bOmega: \DGC_{k/k} \rightleftarrows \DGA_{k/k}: \bB $ are adjoint functors.
Recall that the category $ \DGA_{k/k} $ carries a natural model structure, where the weak equivalences are the quasi-isomorphisms. There is a dual model structure on $ \DGC_{k/k} $,
with weak equivalences being the morphisms $f$ such that $ \bOmega(f) $ is a quasi-isomorphism.
It is easy to check that $ \bOmega $ and $ \bB $ are Quillen functors relative to these model
structures, and in fact, they induce mutually inverse equivalences between the homotopy categories
$\,\Ho(\DGC_{k/k}) $ and $ \Ho(\DGA_{k/k})\,$ ({\it cf.} \cite[Theorem~4.3]{Ke}).

Now, let $ \Mod(R) $ denote the category of right DG modules over $ R $, and dually
let $ \cMod(C) $ denote the category of right DG comodules over $ C $ which
are conilpotent in a sense similar to \eqref{cocom}.
Given a twisting cochain $\, \tau \in \Tw(C,R) \,$  one can define a pair of functors between these categories
\begin{equation*}
\mbox{--} \otimes_{\tau} C :\, \Mod(R) \rightleftarrows \cMod(C)\,:
\mbox{--} \otimes_{\tau} R
\end{equation*}
called the {\it twisted tensor products}. Specifically, if $ M \in \Mod(R) $,
then $ M \otimes_{\tau} C $ is defined to be the DG $C$-comodule whose
underlying graded comodule is $ M \otimes_k C $ and whose differential is given by
$$
d = d_M \otimes \id + \id \otimes d_C +
(m \otimes \id)\,(\id \otimes \tau \otimes \id)\,(\id \otimes \Delta)\ .
$$
Similarly, for a DG comodule $ N \in \cMod(C) $, one defines a
DG $R$-module $ N \otimes_\tau R $.

Next, recall that the {\it derived category}
$ \D(R) $ of DG modules is obtained by localizing $ \Mod(R) $ at the class of all quasi-isomorphisms. To introduce the dual notion for DG comodules one has to replace
the quasi-isomorphisms by a more restricted class of morphisms in $ \cMod(C) $.
We call a morphism $f$ in $ \cMod(C) $ a weak equivalence if
$\,f \otimes_{\tau_C} \bOmega(C) \,$ is quasi-isomorphism in $ \Mod\,\bOmega(C)\,$, where $ \tau_C: C \to \bOmega(C) $ is the universal twisting cochain corresponding to the identity map under \eqref{fism}. The {\it coderived category} $ \D^c(C) $ of DG comodules is then defined by localizing $ \cMod(C) $ at the class of weak equivalences. It is easy to check that the twisted tensor products induce a pair of adjoint functors
\begin{equation}
\la{quieq1}
\mbox{--} \otimes_{\tau} C \, :\, \D(R) \rightleftarrows \D^c(C)\,: \,
\mbox{--} \otimes_{\tau} R\ .
\end{equation}
The following theorem characterizes the class of twisting cochains for which \eqref{quieq1}
are equivalences.
\bthm[see \cite{LV}, Theorem~2.3.2]
\la{Kosc}
For $\,\tau \in \Tw(C,R)\,$, the following are equivalent:
\begin{enumerate}
\item[(i)] the functors \eqref{quieq1} are mutually inverse equivalences of categories;
\item[(ii)] the complex $\,C \otimes_\tau R \,$ is acyclic;
\item[(iii)] the complex $\,R \otimes_\tau C \,$ is acyclic;
\item[(iv)] the natural morphism $ R \otimes_\tau C \otimes_\tau R \stackrel{\sim}{\to} R $ is a quasi-isomorphism;
\item[(v)] the morphism $ \bOmega(C) \stackrel{\sim}{\to} R $ corresponding to $\tau $ under \eqref{fism} is a quasi-isomorphism in $ \DGA_{k/k}$;
\item[(vi)] the morphism $ C \stackrel{\sim}{\to} \bB(R) $ corresponding to $\tau $ under
\eqref{fism} is a weak equivalence in $ \DGC_{k/k} $.
\end{enumerate}
If the conditions {\rm (i)} - {\rm (vi)} hold, the DG algebra $R$ is determined by $C$ up to isomorphism in
$ \Ho(\DGA_{k/k}) $ and the DG coalgebra $C$ is determined by $ R $ up to isomorphism
in $ \Ho(\DGC_{k/k}) $.
\ethm
A twisting cochain $ \tau \in \Tw(C,R) $ satisfying the conditions of
Theorem~\ref{Kosc} is called {\it acyclic}. In this case, the DG coalgebra $C$
is called {\it Koszul dual} to the DG algebra $ R $ and $ R $
is called {\it Koszul dual} to $ C $.

\subsection{Duality between Lie homology and representation homology}
Recall that if $ \g $ is a Lie algebra and $ \h \subseteq \g $ is a Lie subalgebra
of $ \g $, the relative Lie algebra homology $ \H_\bullet(\g, \h; k) $ is computed by
the standard Chevalley-Eilenberg complex $ \CE_\bullet(\g/\h) $. This complex
has a natural structure of a cocommutative DG coalgebra, with comultiplication
induced by the diagonal map  $ \g/\h \to \g/\h \oplus \g/\h $ ({\it cf.} \cite[10.1.3]{L}).

Now, let $ A \in \Alg_{k/k} $. For $ r \ge 1 $, consider the
matrix Lie algebra $\, \gl_r(A) = \M_r(A) \,$ and its canonical Lie subalgebra
$ \gl_r(k) \subseteq \gl_r(A) $. To simplify the notation
write $\, \rgl_r(A) := \gl_r(A)/\gl_r(k) \,$. Then, by \cite[Lemma~6.1]{LQ}
(see also \cite[10.2.3]{L}), there is a morphism of complexes
\begin{equation}
\la{liem1}
\vartheta_\bullet:\,  \CE_\bullet(\rgl_r(A)) \to \rCC_\bullet(A)[1]\ ,
\end{equation}
where $ \rCC_\bullet(A) $ is Connes' cyclic complex computing the reduced
cyclic homology of $A$. Specifically, $ \vartheta_\bullet $ is defined by
$$
\vartheta(\xi_0 \wedge \xi_1 \wedge \ldots \wedge \xi_k) = \sum_{\sigma \in \mathbb{S}_k} \,
\mbox{\rm sgn}(\sigma)\,\Tr (\xi_0 \otimes \xi_{\sigma(1)} \otimes
\ldots \otimes \xi_{\sigma(k)})\ ,
$$
where $ \Tr:\, \M_r(A)^{\otimes (k+1)} \to A^{\otimes (k+1)} $ is the
generalized matrix trace:
$$
\Tr (\alpha \otimes \beta \otimes \ldots \otimes \eta) :=
\sum\, \alpha_{i_0 i_1} \otimes \beta_{i_1 i_2} \otimes \ldots \otimes \eta_{i_{k+1} i_0}\ .
$$
Since $ \CE_\bullet(\rgl_r(A)) $ is a cocommutative DG coalgebra, the morphism
of complexes \eqref{liem1} extends to a (unique) map of DG coalgebras
\begin{equation}
\la{liem2}
\bSym^{\! c}(\vartheta)_\bullet:\,  \CE_\bullet(\rgl_r(A)) \to \bSym^{\! c} (\rCC_\bullet(A)[1])\ ,
\end{equation}
where $ \bSym^{\! c} (V) $ denotes the cofree cocommutative DG coalgebra cogenerated
by a complex $ V $. The map $ \bSym^{\! c}(\vartheta_\bullet) $ factors through the natural inclusion
$\,\CE_\bullet(\rgl_r(A)) \into \CE_\bullet(\rgl_\infty(A))\,$, and the Loday-Quillen-Tsygan
Theorem \cite{LQ,T} implies that the induced map
$$
\bSym^{\! c}\,(\vartheta_\infty)_\bullet:\,  \CE_\bullet(\rgl_\infty(A)) \to
\bSym^{\! c}(\rCC_\bullet(A)[1])
$$
is a quasi-isomorphism. Thus, we have an isomorphism of graded coalgebras
\begin{equation*}
\la{LQTs}
\H_\bullet(\gl_\infty(A), \gl_\infty(k); k) \stackrel{\sim}{\to} \bSym^{\! c}(\rHC_\bullet(A)[1])\ .
\end{equation*}

On the other hand, in \cite[Section~4.3.3]{BKR}, for a fixed cofibrant
resolution $\,\pi: R \sonto A \,$, we constructed a morphism of complexes
\begin{equation}
\la{ccdr}
T_\bullet:\, \rCC_\bullet(A) \to R_n^{\GL_n}\ ,
\end{equation}
that induces the trace maps \eqref{trm3}. Explicitly, \eqref{ccdr} is given by the formula
$$
T(a_1 \otimes \ldots \otimes a_{k+1}) = \sum_{j \in \Z_{k+1}}\, (-1)^{jk}\,
\Tr_n\,[f_{k+1}(a_{1+j}, a_{2+j}, \ldots, a_{k+1+j})]\ ,
$$
where $\, f_{k+1}:\, A^{\otimes (k+1)} \to R \,$ is the $(k+1)$-th component of
the twisting cochain $ \bB(A) \to R $
corresponding to the DG coalgebra map $ f: \bB(A) \to \bB(R) $ such that
$\, \bB(\pi)\circ f =  \id \,$, and $\,\Tr_n:\, R \to   R_n^{\GL_n} \,$ is defined in \eqref{e2s1}.
Again, since $ R_n^{\GL_n} $ is a commutative DG algebra, \eqref{ccdr} extends
to a (unique) map of DG algebras
$$
\bSym(T)_\bullet:\, \bSym[\rCC_\bullet(A)] \to R_n^{\GL_n}
$$
which induces (in the limit $ n \to \infty $) the isomorphism \eqref{anlqt}.

Combining the maps \eqref{liem1} and \eqref{ccdr}, we now define
\begin{equation}
\la{tmn}
\tau_{r,n}(A):\,
\CE_\bullet(\rgl_r(A)) \xrightarrow{\vartheta} \rCC_\bullet(A)[1]
\xrightarrow{s^{-1}}  \rCC_\bullet(A) \xrightarrow{T} R_n^{\GL_n}\ ,
\end{equation}
where $ s^{-1} $ is the inverse of the canonical degree $1$  map
$\,s: \rCC_\bullet(A) \to \rCC_\bullet(A)[1] \,$ inducing
isomorphisms $ \rCC_i(A) \stackrel{\sim}{\to} \rCC(A)[1]_{i+1} $  for
all $i$.
\blemma\la{twmn}
The map \eqref{tmn} is a twisting cochain.
\elemma
\bproof
For any complex $ V \in \Com(k) $, it is straightforward to check that
$$
\tau_V:\, \bSym^{\! c}(V[1]) \onto V[1] \xrightarrow{s^{-1}} V \into \bSym(V)
$$
is a twisting cochain. Now, note that for any $ r $ and $ n $, $\, \tau_{r,n} =
\tau_{r,n}(A) $ can be factored as
$$
\tau_{r,n} = \bSym(T_\bullet) \, \circ\, \tau_V\, \circ \, \bSym^{\! c}(\vartheta_\bullet) \ ,
$$
where $ V = \rCC_\bullet(A) $. Since $  \bSym(T_\bullet) $ is a DG algebra map
and $ \bSym^{\! c}(\vartheta_\bullet) $ is a DG coalgebra map, $ \tau_V $ being a
twisting cochain implies that $ \tau_{r,n} $ is a twisting cochain.
\eproof

It is natural to ask whether the twisting cochain defined by \eqref{tmn}
is acyclic (i.e., satisfies the equivalent conditions of Theorem~\ref{Kosc}).
Although we do not know the answer for $r$ and $n$ finite
({\it cf.} Question~\ref{qkosz} below), the results of the previous section
show that $ \tau_{r,n}(A) $ does become acyclic when we pass to the limit $ r,n \to \infty $.
To make this observation precise define
$$
\tau_{\infty, \infty}(A) := \varprojlim\limits_{n}\,\varinjlim\limits_{r}\,\tau_{r,n}(A)
$$
By Lemma~\ref{twmn}, $ \tau_{\infty, \infty}(A) $ is a twisting cochain
$ \, \CE_\bullet(\rgl_\infty(A)) \to R_\infty^\Tr\,$,
where $ R_\infty^\Tr $ is the trace subalgebra of $ R_\infty^{\GL_{\infty}} $
introduced in Section~\ref{S2.1}.
\bthm\la{KD2}
The twisting cochain $\,\tau_{\infty, \infty}(A) \,$ is acyclic. Thus, for any algebra
$ A \in \Alg_{k/k} $, the Chevalley-Eilenberg coalgebra
$ \CE_\bullet(\rgl_\infty(A)) $ is Koszul dual to the DG algebra $ \DRep_\infty(A)^\Tr $.
\ethm
\bproof
We may factor $ \tau_{\infty,\infty}(A) $  as
$$
\CE_\bullet(\rgl_\infty(A))
\xrightarrow{\bSym^{\! c}\,(\vartheta_\infty)}
\bSym^{\! c} (\rCC_\bullet(A)[1])
\xrightarrow{\tau} \bSym(\rCC_\bullet(A))
\xrightarrow{\bSym(T_\infty)}
R_\infty^\Tr\ ,
$$
where $ \tau = \tau_V $ is a twisting cochain defined in the proof of Lemma~\ref{twmn}.
By the Loday-Quillen-Tsygan Theorem, $\, \bSym^{\! c}\,(\vartheta_\infty) $
is a quasi-isomorphism (in fact, a weak equivalence) in $ \DGC_{k/k} $. By
Proposition~\ref{pS4.3}, $ \bSym(T_\infty) $ is a quasi-isomorphism in $ \DGA_{k/k} $.
Hence, to prove that $ \tau_{\infty, \infty}(A) $ is acyclic, it suffices to prove that $ \tau $ is
acyclic. Choosing any morphism of complexes $\,\rCC_\bullet(A) \to \rHC_\bullet(A) \,$
that induces the identity on homology, we may also factor $ \tau $ as
$$
\bSym^{\! c} (\rCC_\bullet(A)[1])
\stackrel{\sim}{\to}
\bSym^{\! c} (\rHC_\bullet(A)[1])
\xrightarrow{\bar{\tau}} \bSym(\rHC_\bullet(A))
\stackrel{\sim}{\to}
\bSym(\rCC_\bullet(A))\ ,
$$
where $ \bar{\tau} $ is the standard twisting cochain relating the cofree cocommutative
coalgebra of a (graded) vector space to the free commutative algebra of that vector space.
It is well known that such $ \bar{\tau} $ is acyclic (see, e.g.,\cite[Prop.~3.4.13]{LV}).
It follows that $ \tau $ is acyclic, and hence so is $ \tau_{\infty, \infty}(A) $. The second
statement
of the theorem is a formal consequence of the first, since $ \DRep_\infty(A)^\Tr $
is represented in $ \Ho(\cDGA_{k/k}) $ by the DG algebra $ R_\infty^\Tr $.
\eproof

\subsubsection{Questions}
\la{qkosz}
1.\, It would be interesting to determine conditions (or at least, give examples) when
the twisting cochains $ \tau_{r,n}(A) $ are acyclic for finite $r$ and $n$. This certainly
does not hold in general and probably requires some strong homological assumptions on $A$.

2.\, The Chevalley-Eilenberg coalgebras $ \CE_\bullet(\gl_r(A)) $ have a natural
interpretation in terms of formal deformation theory developed in \cite{H1}
(see {\it loc. cit.}, Section~10.4.4). It would be interesting to clarify
how the construction of \cite{BKR} and the present paper fits in the framework of \cite{H1}.

\subsection{Obstruction complex}
\la{S2.4}
We now address the question \eqref{quest} stated in the Introduction.
Given a DG algebra $ A \in \DGA_{k/k} $ and integer $\,n\ge 1\,$, we define the complex
$$
K_{\bullet}(A,n):=\Ker\{ \,\bSym \mathrm{Tr}_n(A)_{\bullet}\,:\,
\bSym[\mathcal{C}(A)] \rar A_n^{\GL}\,\} \mathrm{.}
$$
It is clear that $ K_{\bullet}(A,n) $ is natural in $A$, i.e.
we have a well-defined functor
\begin{equation}
\la{kfun}
K_\bullet(\,\mbox{--}\,, n) :\, \DGA_{k/k} \to \Com_k\ .
\end{equation}

\begin{theorem} \la{tS2.4}
$(a)$ The functor \eqref{kfun} has a total left derived functor $$
\L K_{\bullet}(-,n) \,:\,\Ho(\DGA_{k/k}) \rar \D(k) \,,\,\,\,\,\,\,\,\,A \mapsto K_{\bullet}(QA,n) \mathrm{.}$$
$(b)$  For any $ A \in \DGA_{k/k}$, there is a long exact sequence
$$\begin{diagram}[small]
\ldots &\rTo& \mathrm{HK}_p(A,n) & \rTo & \bSym[\overline{\mathrm{HC}}_{\bullet}(A)]_p &\rTo^{\bSym \mathrm{Tr}_n(A)} &\mathrm{H}_p(A,n)^{\GL} & \rTo^{\partial}& \mathrm{HK}_{p-1}(A,n) &\rTo& \ldots
\end{diagram}$$
where $ \mathrm{HK}_{\bullet}(A,n)\,:=\, \H_{\bullet}[\L K(A,n)]$.
\end{theorem}

\begin{proof}
$(a)$\,
By Brown's Lemma (see \cite[Lemma 9.9]{DS}), it suffices to check that \eqref{kfun} maps any acyclic cofibration between cofibrant objects in $\DGA_{k/k}$ to a weak equivalence (quasi-isomorphism) in $ \Com_k $. For any morphism $\,f:\,A \rar B\,$ in $\DGA_{k/k}$, one has the commutative diagram in $\Com_k$
$$
\begin{diagram}
0 & \rTo & K_{\bullet}(A,n) &\rTo & \bSym[\mathcal{C}(A)]  & \rTo^{\bSym \mathrm{Tr}_n(A)} & A_n^{\GL} & \rTo & 0 \\
 &  & \dTo^{K(f,n)} & & \dTo^{\bSym[\mathcal{C}(f)]}  & &  \dTo^{f_n^{\GL}} & & \\
0 & \rTo & K_{\bullet}(B,n) &\rTo & \bSym[\mathcal{C}(B)]  & \rTo^{\bSym\mathrm{Tr}_n(B)}&  B_n^{\GL} & \rTo & 0
\end{diagram}
$$
If $f:A \rar B$ is an acyclic cofibration between cofibrant objects, $\bSym[\mathcal{C}(f)] $ and $f_n^{\GL}$ are quasi-isomorphisms by (the proofs of) Theorem~2.6 and Theorem~3.1 of \cite{BKR}. Hence,
so is $K_{\bullet}(f,n)$. This proves part $(a)$.

$(b)$\, By Theorem~\ref{pS1.3.1}, for any
DG algebra $ R $, there is a short exact sequence in $\Com_k\,$:
\begin{equation}
\la{eS2.4.1}
0 \to  K_{\bullet}(R,n) \to \bSym[\mathcal{C}(R)]
\xrightarrow{\,\bSym \mathrm{Tr}_n(R)\,}  R_n^{\GL} \to 0\ .
\end{equation}
Picking $ R = QA $ to be a(ny) cofibrant resolution of $A$ in
$\DGA_{k/k}$ and applying the homology functor to \eqref{eS2.4.1}
we get the required long exact sequence.
\end{proof}

For $\, A \in \Alg_{k/k}\,$, the homology $\,\mathrm{HK}_{\bullet}(A,n) $ may be viewed as an obstruction to $\mathrm{H}_{\bullet}(A,n)^{\GL} $  attaining its `stable limit' ({\it cf.} also Remarks~\ref{3.4.1} below). By Theorem~\ref{tS2.4}$(b)$, the kernel of the connecting homomorphism $\,\partial\,:\,\mathrm{H}_p(A,n)^{\GL} \rar \mathrm{HK}_{p-1}(A,n)\,$ measures the
failure of $\bSym \mathrm{Tr}_n(A)_{\bullet} $ to induce a surjective map at the level of homology in degree $p\,$.
In negative degrees, the complex $\, K_{\bullet}(A,\,n) \,$ is acyclic, which means, in particular, that the map $ \bSym[\rHC(A)]_0 \to \H_0(A,\,n)^{\GL} $ is surjective. Since $ \bSym[\rHC(A)]_0 =
\Sym[\rHC_0(A)] $ and $ \H_0(A,\,n)^{\GL} = A_n^\GL $, we recover the classical result of Procesi.
On the other hand, the following simple examples show that $\mathrm{HK}_{\bullet}(A,n)$ and $\partial$ may actually be nonzero
in positive homological degrees, even for $ n = 1 $. Thus, the answer to the question stated in the beginning of this section is negative.

\subsubsection{Examples}
\la{ex11}
$(a)$\ Let $ A:= k[x,y] $ be the polynomial algebra
(endowed with some augmentation). It has a cofibrant
resolution of the form $ R = k \langle x,y,t \rangle $ with $ |x|=|y|=0 \,$, $\, |t|=1 \,$ and $\,dt=[x,y]\,$. For $n=1$,
$\H_{\bullet}(A,1) $ is the homology of the commutative DG algebra $R_{\nn} = k[x,y,t] $ with zero differential. Hence,
$$\mathrm{H}_{\bullet}(A,1) \,\cong\, A[t] \text{.}$$
On the other hand, $\overline{\mathrm{HC}}_1(A) \cong \Omega^1(A)/dA\,$. As shown in \cite[Example~4.1]{BKR},
the class of the $1$-form $\,y\,dx\,$  maps to the class of the cycle $t$ in $R_{\natural \natural}$. Hence, in this case
$\,\bSym \mathrm{Tr}_1(A)_{\bullet} \,:\, \bSym [\overline{\mathrm{HC}}(A)] \rar
\mathrm{H}_{\bullet}(A,1) $ is degreewise surjective with kernel $\mathrm{HK}_{\bullet}(A,1)$.
 Since $\overline{\mathrm{HC}}_0(A) = \bar{A} $ as well as $\overline{\mathrm{HC}}_1(A)$ are infinite-dimensional vector spaces, we see that $\mathrm{HK}_{p}(A,1)$ is nonzero (in fact,
 infinite-dimensional over $k$) for all $p \geq 0$.
However, the maps $\partial\,:\, \mathrm{H}_p(A,1) \rar \mathrm{HK}_{p-1}(A,1)$ vanish for all $p \geq 1$ in this case.

$(b)$\ Let $A=k[x]/(x^2)$ be the ring of dual numbers. Then, $A$ has a free resolution of the form
$ R=k\langle x, x_1, x_2,\ldots \rangle $, where $\mathrm{deg}(x)=0$ and $\mathrm{deg}(x_i)=i $ for all $ i \in \mathbb N$. The differential on $R$ is given by the formula
\begin{equation}
\la{dualnum}
dx_i \,=\, x x_{i-1} - x_1 x_{i-2} + x_2 x_{i-3}- \ldots +{(-1)}^{i-1} x_{i-1} x \,\text{.}
\end{equation}
Hence, $\mathrm{H}_{\bullet}(A,1) $ is the homology of the commutative DG algebra $ k[x, x_1,\ldots ,x_i,\ldots ] $ with differential given by the same formula \eqref{dualnum}. This homology
is easy to compute in small degrees (see \cite[Section~6]{BFR}).
In particular, we have
\begin{eqnarray*}
\H_1(A,1) &=& 0 \ , \\
\H_3(A,1) & = & A \cdot (x x_3 - 2x_1 x_2)\ ,\\
\H_5(A,1) & = & A \cdot (-2x_1 x_2^2 + x x_2 x_3) +
A \cdot (-x_2 x_3 - 4 x_1 x_4 + 2 x x_5)\ ,\\
\H_7(A,1) & = & A \cdot (-x_2^2 x_3 - 4 x_1 x_2 x_4 + 2 x x_2 x_5) +
A \cdot (-x_3 x_4 - 2 x_1 x_6+ x x_7) \ ,\\
.\ .\ .\ .\ .\ . & &
\end{eqnarray*}
On the other hand,
$\overline{\HC}_p(A)$ is known to vanish in {\it all} odd degrees (see \cite[Section~4.3]{LQ}). Hence the algebra map
$\bSym \mathrm{Tr}_1(A)_{\bullet}$ is not surjective in this case.

$(c)$\
Let $ A := k\langle x,y \rangle/(x[x,y]y) $. We claim that
$\partial\,:\, \mathrm{H}_{1}(A,1) \to \mathrm{HK}_0(A,1)$ is nonzero. To see this note that
$A$ is a (polynomially) graded algebra (with generators $x,y$ of polynomial degree $1$). Hence, $A$ has a cofibrant resolution that is both homologically
and polynomially graded of the form $R= k\langle x,y,t,u_1,..\rangle$ where $x,y$ have homological degree $0$ and polynomial degree $1$, $t$
has homological degree $1$ and polynomial degree $4$ with $dt=x[x,y]y$ and the variables $ u_1, \ldots $ are bihomogeneous with homological degree
$\geq 2$. Note that the image of $t$ is a $1$-cycle in $R_{\natural \natural}$. However, the image of $t$ in
$\mathcal C(R)$ is not a $1$-cycle. Let $ \bar{t} $ denote the class of the image of $t$ in $\mathrm{H}_1(A,1)\,$.
Then
$\,\partial \bar{t} \,=\, \overline{dt} \,\,\in \,\mathrm{HK}_0(A,1) \,$.
Our claim therefore follows if we verify that $dt$ is not a boundary in $K_{\bullet}(A,1)$. For this, note first
that $\bSym \mathrm{Tr}_1(A)_{\bullet} $ preserves the polynomial grading as well. Hence, as complexes of $k$-vector spaces,
$$ K_{\bullet}(A,1) \cong \oplus_{r \ge 0}\, K_{r,\bullet}(A,1) $$ where $K_{r, \bullet}(A,1)$ is the subcomplex of $K_{\bullet}(A,1)$
spanned by the homogeneous elements of polynomial degree $ r $. Now, it is easy check that $K_{r,1}(A,1)=0$ for $ r \leq 4$. Hence $dt$ is not a boundary in $K_{\bullet}(A,1)$.

\section{De Rham cohomology of derived representation schemes}
\la{Cryscoh}
In this section, we compute the `stable limit' of the de Rham cohomology of the derived
representation scheme $ \DRep_n(A)^\GL $. We begin by recalling the definition of the Karoubi-de Rham homology.

\subsection{Karoubi-de Rham homology} For $ R \in \DGA_k $,
let $\Omega^1 R$ denote the kernel of the multiplication map from $R \otimes R$ to $R$. Note that $\Omega^1 R$
has a natural $R$-bimodule structure. For $r \in R$, we denote the element $1 \otimes r - r \otimes 1 \, \in \Omega^1R $
 by $ \partial r $ (as in~\cite[Section 2.6]{L}). This defines a canonical (universal) derivation
$\partial\,:\,R \rar \Omega^1 R\,,\,\,\,\,r \mapsto \partial r$. We will regard $\partial$ as a degree $-1$ derivation from
$R$ to $\Omega^1R[-1]$ and extend it to a degree $-1$ derivation on
the DG algebra $\mathrm{T}_R (\Omega^1R[-1])$ using the Leibniz rule. The derivation $\partial$ (anti)commutes with the derivation $d_R$ induced by
$R$. The DG algebra $(\mathrm{T}_R (\Omega^1R)[-1], d_R+\partial)$ is the algebra of noncommutative differential forms on $R$. In
what follows, we will assume that $\mathrm{T}_R (\Omega^1R[-1])$ is equipped with the differential $d_R+\partial$ unless
explicitly stated otherwise.

\vspace{1ex}

\begin{definition}
The {\it Karoubi-de Rham homology} of $R$ is defined to be the homology of the
complex $\mathrm{T}_R (\Omega^1R[-1])_{\natural}$. We also define the {\it reduced Karoubi de Rham homology} of $R$
to be the homology of the complex $\mathcal C(\mathrm{T}_R (\Omega^1R[-1]))\,\cong\, \mathrm{T}_R (\Omega^1R[-1])_{\natural}/k  \,$.
We denote the Karoubi-de Rham homology (resp., reduced Karoubi-de Rham homology) of $R$ by $\mathrm{HDR}_{\bullet}(R)$
 (resp., $\overline{\mathrm{HDR}}_{\bullet}(R)$).
\end{definition}

\vspace{1ex}

\blemma \la{P2}
For any free non-negatively graded DG algebra $R$, we have $\,
\overline{\mathrm{HDR}}_{\bullet}(R) = 0 \,$.
\elemma

\begin{proof}
Since $R$ is free, the DG algebra $\mathrm{T}_R (\Omega^1R[-1])$ is free as well. On the other hand, $\mathrm{T}_R (\Omega^1R[-1])$ is the
direct sum total complex of a second quadrant bicomplex whose term in bidegree $(-q,p)$ is $(\Omega^q R)_p$ and whose horizontal differential
is $\partial$. One therefore, has a spectral sequence
$$
E^1_{p,-q}\,:=\, \mathrm{H}_{-q}[(\Omega^{\bullet} R)_p] \ \Rightarrow\ \mathrm{H}_{p-q}[\mathrm{T}_R (\Omega^1R[-1])]\,\text{.}$$
The argument in~\cite[Section 2.5]{CEG} goes through to show that $ E^1_{p,-q} = 0$ unless $p=q=0$, in which case it is isomorphic to $k$.
Thus, $k \into \mathrm{T}_R (\Omega^1R[-1])$ is an acyclic cofibration between cofibrant objects in $\DGA_k$. By the proof of~\cite[Theorem~3.1]{BKR},
$\mathcal C(k) \stackrel{\sim}{\to} \mathcal C(\mathrm{T}_R (\Omega^1R[-1]))$ is a quasi-isomorphism. This proves the desired proposition.
\end{proof}

\subsection{The Karoubi-de Rham complex of a DG representation scheme}
For $B \,\in\, \cDGA_k$, let $ \Omega^1_{\rm com}(B) $ denote the
module of K\"ahler differentials, and let $ \partial_B \,:\,B \rar
\Omega^1_{\rm com}(B) $ be the canonical derivation. View $\partial_B$ as a degree $-1$ derivation $\, B \to \Omega^1_{\rm com}(B)[-1] \,$. 
By the Leibniz rule,  $\partial_B$ extends to a degree $-1$ derivation
on $\bSym_B (\Omega^1_{\rm com} B[-1])$ which (anti)commutes with the differential $d_B$ coming from the intrinsic differential on
$B$. The commutative DG algebra $(\bSym_B (\Omega^1_{\rm com} B[-1]), d_B+\partial_B)$ is called the {\it de Rham algebra}
 of $B$, and  we will denote it henceforth by $\mathrm{DR}(B)$. The 
 de Rham cohomology of $ B $ is then defined to be
$ \mathrm{H}_{\mathrm{DR}}^{n}(B) := \mathrm{H}_{-n}[\mathrm{DR}(B)]$.

\begin{prop} \la{P3}
For any $ R\,\in\,\DGA_k \,$, there is a natural isomorphism  of DG algebras
$$
(\mathrm{T}_R(\Omega^1 R[-1]))_n \,\cong\, \mathrm{DR}(R_n) \ .
$$
\end{prop}

\begin{proof}
First, note that the representation functor induces
a canonical isomorphism ({\it cf.} \cite[Lemma~3.3]{VdB})
\begin{equation}\la{candd}
(\mathrm{T}_R (\Omega^1R[-1]), d_R)_n \,\cong\, (\bSym_{R_n} (\Omega^1_{\rm com}(R_n)[-1]), d_{R_n})\,\text{.}
\end{equation}
To simplify the notation we set $ \tilde{R} := \mathrm{T}_R (\Omega^1R[-1]) $ and identify
$ \tilde{R}_n = \bSym_{R_n}(\Omega^1_{\rm com}(R_n)[-1]) $
using \eqref{candd}. Then, 
by \cite[Proposition~5.2]{BKR}, there is a natural map of graded Lie
algebras $\tau_n: \mathrm{Der}_{\bullet}(\tilde{R}) \to
\mathrm{Der}_{\bullet}(\tilde{R}_n)$ such that
 $$ 
 \tau_n(\theta)(r_{ij}) \,=\, (\theta(r))_{ij}\,,\,\,\,\, 1 \leq i,j \leq n \ ,
 $$
for all $ r \in \tilde{R} $ and $\theta \,\in \mathrm{Der}_\bullet (\tilde{R}) $.
In fact, the above formula uniquely determines $\tau_n (\theta)$. To prove the proposition, it therefore suffices to show that $\tau_n (\partial)\,=\,\partial_{n}\,$, where 
$\,\partial := \partial_{\tilde{R}} \,$ and $\,
\partial_n := \partial_{\tilde{R}_n} \,$. For this, it suffices to check that  
$\tau_n (\partial)$ and $ \partial_{n} $ agree on a set of homogeneous generators of 
$\tilde{R}_n $: for example, on the elements of the form $ r_{ij} $ and
$ \partial_{n}r_{ij}$, where $ r \in R $. 
Using \cite[Lemma~5.4]{BKR}, it is easy to show that indeed $\tau_n (\partial)(r_{ij})\,=\, \partial_{n}(r_{ij})$ for all $ r_{ij} $. On the other hand, $\, \tau_n (\partial)^2=0 $ since 
$\partial^2=0$ and $ \tau_n $ is a Lie algebra homomorphism. 
Hence,
$$
\tau_n (\partial)(\partial_{n}r_{ij}) =
\tau_n (\partial)[\tau_n (\partial)r_{ij}]
= \tau_n (\partial)^2(r_{ij})=0 \ ,
$$
which shows that $ \tau_n (\partial) $ and $ \partial_n $ agree
on $ \partial_{n} r_{ij} $ as well.
This completes the proof of the proposition.
\end{proof}

Let $R \stackrel{\sim}{\rar} A$ be a non-negatively graded cofibrant resolution of an\, {\it augmented}\, algebra $ A \in \Alg_{k/k}$. Since $\Omega^1(k) =0$, the DG algebra $\mathrm{T}_R(\Omega^1 R[-1])$ is augmented as well. Hence, we may apply the trace algebra functor $(\,\mbox{--}\,)_{\infty}^{\mathrm{Tr}} $ to $\mathrm{T}_R(\Omega^1 R[-1])$ to obtain an (augmented) commutative DG algebra. Note,
by Proposition~\ref{P3}, $[\mathrm{T}_R(\Omega^1 R[-1])]_{\infty}^{\mathrm{Tr}}$ is the stable limit of $\mathrm{DR}(R_n)^{\GL}$ as $n \rar \infty$. We denote this limit and its homology by
$$
\mathrm{DR}(R_{\infty})^{\mathrm{Tr}} := [\mathrm{T}_R(\Omega^1 R[-1])]_{\infty}^{\mathrm{Tr}} \quad , \quad \H_{\mathrm{DR},\bullet}[\mathrm{DRep}_{\infty}(A)]^{\Tr} := \H_\bullet[\mathrm{DR}(R_{\infty})^{\mathrm{Tr}}]\ .
$$
Theorem~\ref{tS2.3} (or Proposition~\ref{pS4.3}) of Section~\ref{S2.3} then implies

\begin{theorem} \la{T1}
For any $ A \in \Alg_{k/k} \,$, there is an isomorphism of graded algebras
$$
\H_{\mathrm{DR},\bullet}[\mathrm{DRep}_{\infty}(A)]^{\Tr}\,
\cong\, k \,\text{.}
$$
\end{theorem}

\begin{proof}
By Proposition~\ref{pS4.3},  $\,\mathrm{DR}(R_{\infty})^{\mathrm{Tr}} \cong \bSym[\FT(\mathrm{T}_R(\Omega^1 R[-1]))] $. Taking homology, we get
$$
\mathrm{H}_{\mathrm{DR},\bullet}(\DRep_{\infty}(A))^{\Tr}\,\cong\, \bSym[\overline{\mathrm{HDR}}_{\bullet}(R)] \ .
$$
Since $R$ is concentrated in non-negative degrees, the result follows from Lemma~\ref{P2}.
\end{proof}
\subsubsection{Remark}
\la{RKal}
If $R$ is a DG resolution of $A$ concentrated in non-negative degrees, the de Rham cohomology
$ \H^\bullet_{\mathrm{DR}}(R_n) $ is known to be isomorphic to the {\it crystalline cohomology}
$\,\H_{\mathrm{crys}}^{\bullet}[\Rep_n(A)]$ of the classical representation scheme of $A$
(see \cite{FT}). Hence
$$
\H_{\mathrm{DR}, \bullet}(\DRep_n(A))^{\GL} \, \cong \ \H_{\mathrm{crys}}^{-\,\bullet}[\Rep_n(A)]^{\GL} \ ,\quad \forall\,n \ge 1 \ .
$$
The result of Theorem~\ref{T1} is therefore unsurprising: it can be viewed as yet another manifestation of the fact that $ \Rep_n(A)/\!/\GL_n $  `smoothens out' into an affine space
as $ n \rar \infty$ ({\it cf.} \cite[Remark 4.3]{Gi}). It is interesting to compare this with the conjecture of Kontsevich and
To\"en on de Rham cohomology of the derived moduli {\it stack} $\, {\mathcal M}(A) $ classifying all finite-dimensional DG modules over a (smooth proper) DG algebra $A$:
roughly speaking, the de Rham cohomology $ \H_{\mathrm{DR}}({\mathcal M}(A)) $ is expected to be isomorphic to the graded symmetric algebra of the periodic cyclic homology $ \overline{\mathrm{HP}}_{\bullet}(A) $ ({\it cf.} \cite[Conjecture~8.6]{Ka}).

\section{Stabilization theorem for bigraded DG algebras}
\la{S3}
 In this section, we extend Theorem~\ref{tS2.3} to DG algebras equipped with an additional (polynomial) grading. We begin by recalling the definition and establishing basic properties of such bigraded
 differential algebras. For many interesting examples we refer the reader
to \cite{Pi}.

\subsection{Bigraded algebras} \la{S3.1}

A {\it bigraded DG algebra} is a DG $k$-algebra of the form $ A = \oplus_{r=0}^{\infty} A(r)$ such that

\begin{enumerate}
\item[(i)] For each $\, r \ge 0\,$, $\,A(r)$ is a complex of $k$-vector spaces concentrated in non-negative homological degrees. Thus $A(r)_n$ consists of
the elements of homological degree $n$ and polynomial degree $r$.

\item[(ii)] $A(r)_n \cdot A(s)_m \subseteq A(r+s)_{m+n}$, where `$\,\cdot \,$' denotes the multiplication in $A$.

\item[(iii)] The differential on $A$ satisfies the (graded) Leibniz rule with respect to homological grading.

\item[(iv)] $A(0)=k$. Hence, $A$ is naturally augmented over $k$.

\end{enumerate}

\vspace{1ex}

Note that (iii) and (iv) together imply that there is a forgetful functor $\bDGA_k \rar \DGA_{k/k}^+$ from the category of
bigraded DG algebras to the category of nonnegatively graded DG algebras augmented over $k$. One similarly defines the category of bigraded
commutative DG algebras. In all cases, parities that appear whenever the Koszul sign rule is applied come from
homological degrees only.

\subsection{A model structure on $\bDGA_k$} \la{A.1} In this section, we outline a proof that the category $\bDGA_k$ (resp., $\bcDGA_k$) is a model category. This does not follow automatically from
a general theorem of \cite{H} since the bigraded DG algebras do not
seem to be controlled by any operad. Instead, we will use a direct approach of \cite{M} and \cite{J}. We begin with the following important definition.

\subsubsection{Noncommutative Tate resolutions} \la{A.1.1} Let $S \in \bDGA_k$. A {\it noncommutative Tate extension} $S \rar R$
of $S$ is a morphism in $\bDGA_k$ such that there exists a sequence $ V^{(0)} \subseteq V^{(1)} \subseteq \ldots $ of bigraded $k$-vector spaces satisfying

\begin{enumerate}
\item[(i)] each $ V^{(i)} $ is concentrated in non-negative homological degree and positive polynomial
degree.

\item[(ii)] $ S \ast_k T(V^{(i)}) \,\subset \, R $ as bigraded algebras, with $\bigcup_{n=0}^{\infty}\, S \ast_k T(V^{(n)}) = R $ as bigraded algebras.

\item[(iii)] If $d$ is the differential on $R$, then $d V^{(i)} \subseteq S \ast_k T(V^{(i-1)})$.
\end{enumerate}

The notion of a Tate resolution in $\bcDGA_k$ is defined analogously, with $\ast_k$ replaced by $\otimes_k$ and $T(\,\mbox{--}\,)$ replaced by $\bSym(\,\mbox{--}\,)\,$ (see, e.g., \cite{GS}).

\begin{lemma}
\la{lA1.1.1}
$(i)$ Every morphism $S \rar A$ in $\bDGA_k$ factors as $S \stackrel{i}{\longrightarrow} R \stackrel{p}{\longrightarrow}  A\,$,
 where $i$ is a noncommutative Tate extension and $p$ is a surjective quasi-isomorphism.

$(ii)$ Noncommutative Tate extensions satisfy the left lifting property with respect to surjective quasi-isomorphisms.
\end{lemma}
\begin{proof}
 The proof of $(i)$ (resp., $(ii)$) is similar to the proof of Proposition 3.1 (resp., Proposition 3.2) of~\cite{FHT}.
\end{proof}

Let $x$ be a bihomogeneous variable in positive homological as well as positive polynomial degree. Consider the complex of graded vector spaces
$$
V_x := [0 \rar k.x \xrightarrow{d} k.dx \rar 0]\ ,
$$
with the polynomial degree of $dx$ equal to that of $x$. Then,
\begin{lemma} \la{lA1.1.2}
$(i)$ $S \rar S \ast_k T(V_x)$ has the left lifting property with respect to any morphism in $\bDGA_k$ that
is surjective in all positive homological degrees.

$(ii)$ Any morphism of the form $S \rar S \ast_k (\amalg_{\lambda \in I} T(V_{x_{\lambda}}))$ has the left lifting property
with respect to any morphism in $\bDGA_k$ that is surjective in all positive homological degrees. Here, $I$ is any (possibly uncountably infinite)
indexing set.
\end{lemma}
Now, we define the {\it weak equivalences} in $\bDGA_k$ to be the class of all quasi-isomorphisms and the {\it fibrations} to be the class of morphisms that are surjective in positive homological degrees.
The {\it cofibrations} are then the morphisms in $\bDGA_k$ that have the left lifting property with respect to acyclic fibrations. In particular, the noncommutative Tate extensions are cofibrations, and in fact, Lemma~\ref{lA1.1.1}$(ii)$ implies
that any cofibration is a retract of a noncommutative Tate extension  ({\it cf.} \cite{M}).
By definition, the morphisms in Lemma~\ref{lA1.1.2}$(ii)$ are acyclic cofibrations in $\bDGA_k$. We refer to such morphisms as {\it special acyclic cofibrations}.

\subsubsection{} \la{A.1.2} We now claim
\begin{prop}
\la{pA1.2}
 With weak equivalences, fibrations and cofibrations defined as above, $\bDGA_k$ is a model category.
\end{prop}

To prove Proposition~\ref{pA1.2} we need the following lemma (which is analogous to \cite[Lemma 2]{J}).

\begin{lemma}\la{lA1.2.1}
 Every morphism $f:S \rar A$ in $\bDGA_k$ factors as $S \stackrel{i}{\longrightarrow} R \stackrel{p}{\rar} A\,$,
where $i$ is a special acyclic cofibration and $p$ is a fibration.
\end{lemma}

\begin{proof}
Consider the set $I$ of all bihomogeneous $a \in A$ in positive homological degree. For each $a \in I$, consider the morphism $p_a\,:\, T(V_{x_a}) \rar A$ in $\bDGA_k$ given by $x_a \mapsto a,\,\,dx_a \mapsto da$. Put $i$ to be the inclusion $S \rar S\ast_k (\amalg_{a \in I} T(V_{x_a}))$ and $p$ to be the morphism $f \ast_k \amalg_{a \in I} p_a\,:\,  S\ast_k (\amalg_{a \in I} T(V_{x_a})) \rar A$. By construction, $f=pi$, $i$ is a special acyclic cofibration and $p$ is a fibration.
\end{proof}

The first three axioms (MC1), (MC2) and (MC3) of a model category (see \cite[3.3]{DS})
are easily verified for $\bDGA_k$. Lemmas~\ref{lA1.1.2}$(ii)$ and~\ref{lA1.2.1} together imply (MC5). (MC4)(i) is satisfied by definition. To prove Proposition~\ref{pA1.2}, we therefore only need to verify (MC4)(ii). For this, note that Lemma~\ref{lA1.2.1} implies that any acyclic cofibration is a retract of a standard acyclic cofibration: indeed, if $i$ is an acyclic cofibration, $i=p\tilde{i}$ where $\tilde{i}$ is a standard acyclic cofibration and $p$ is a (necessarily acyclic) fibration. Since $i$ has the left lifting property with respect to $p$, $i$ is a retract of $\tilde{i}$. Since Lemma~\ref{lA1.1.2}$(ii)$ states that any standard acyclic cofibration has the left lifting property with respect to all fibrations, any acyclic cofibration has the left lifting property with respect to all fibrations. This completes the proof of Proposition~\ref{pA1.2}.

\subsubsection{} \la{A.1.3}

We remark that $\bcDGA_k$ has a model structure where fibrations are surjections in positive homological degree, weak equivalences are quasi-isomorphisms and cofibrations are retracts of the classical (commutative) Tate resolutions. The proof of this fact is similar to that of Proposition~\ref{pA1.2} except for the following modifications: free products are replaced by tensor products and bigraded DG algebras of the form $T(V_x)$ are replaced by corresponding bigraded commutative DG algebras of the form $\bSym(V_x)$.

\subsection{Bigraded representation functors} \la{3.2} The functors
 $\,(\,\mbox{--}\,)_n\,$, $\,(\,\mbox{--}\,)_n^{\GL}\,\,$ and $\,\bSym[\mathcal{C}(\,\mbox{--}\,)]\,$ defined originally as
 functors $ \,\DGA^+_{k/k} \to \cDGA^+_k $  enrich to functors from $\bDGA_k$ to $\bcDGA_k$. Furthermore, the functors
$(\mbox{--})_{{\infty}},\, (\mbox{--})_{\infty}^{\mathrm{Tr}}\,:\, \bDGA_k \rar \bcDGA_k$ can be defined exactly in the same way as the
corresponding functors from $\DGA_{k/k}$ to $\cDGA_k$. However, we caution the reader that the functors
$\,(\mbox{--})_{{\infty}},\, (\mbox{--})_{\infty}^{\mathrm{Tr}}\,:\, \bDGA_k \rar \bcDGA_k\,$ are {\it not} enrichments of the corresponding functors from $\DGA_{k/k}$ to $\cDGA_k$. The difference here comes from the fact that we take the inverse limit in $\bcDGA_k$ rather than in $\cDGA_k\,$: the forgetful functor
$\bcDGA_k \rar \cDGA_k$ does not preserve infinite limits in general. Yet another illustration of the difference
between the inverse limits in $\bcDGA_k$ and $\cDGA_k$ is the following important proposition.

\begin{prop} \la{pS3.2}
 Let $k \rar R$ be a noncommutative Tate extension in $\bDGA_k$ with finitely many generators in each
polynomial degree. Then there is an isomorphism $\,R_{\infty}^{\mathrm{Tr}} \cong R_{{\infty}}^{\GL} \,$ in $\,\bcDGA_k\,$.
\end{prop}

\subsection{Proof of Proposition~\ref{pS3.2}} \la{3.3} \subsubsection{}\la{3.3.1}
It suffices to prove that $\bSym \mathrm{Tr}_{\infty}\,:\,\bSym[\mathcal{C}(R)] \rar R_{{\infty}}^{\GL}$
is a surjection in $\bcDGA_k$. For this, it is enough to show that for any $r \in \mathbb N$,
there exists a natural number $N(r)$ such that the natural map $\mu_{n+1,n}\,:\,R_{n+1}^{\GL}(r) \rar R_{n}^{\GL}(r)$ is an isomorphism of $k$-vector spaces
for all $n \geq N(r)$.
Indeed, if this is the case, then $R_{{\infty}}^{\GL}(r) \cong R_{N(r)}^{\GL}(r)$. Since the diagram
$$\begin{diagram}
    &   &  R_{{\infty}}^{\GL}\\
   & \ruTo^{\bSym \mathrm{Tr}_{{\infty}}(R)} & \dTo\\
  \bSym[\mathcal{C}(R)] &\rOnto_{\bSym \mathrm{Tr}_{N(r)}(R)} & R_{N(r)}^{\GL}
  \end{diagram}
  $$
commutes (with the bottom arrow being a surjection by Theorem~\ref{pS1.3.1}),
$\bSym \mathrm{Tr}_{{\infty}}(R)_{\bullet} $ induces a surjection between $\bSym[\mathcal{C}(R)](r)$ and
$R_{{\infty}}^{\GL}(r)$ for each $r$ in $\mathbb N$. Thus, $R_{\infty}^{\mathrm{Tr}}(r) \cong R_{{\infty}}^{\GL}(r)$
for each $r \in \mathbb N$, which implies that $R_{\infty}^{\mathrm{Tr}} \cong R_{{\infty}}^{\GL}$ in $\bcDGA_k$.

\subsubsection{}\la{3.3.2}
Further observe that the claim that there exists an $N(r) \in \N$ such that $\mu_{n+1,n}\,:\,R_{n+1}^{\GL}(r) \rar R_{n}^{\GL}(r)$ is an isomorphism of
 $k$-vector spaces
for $n \geq N(r)$ will follow if we can show that for there exists an $N(r) \in \N$ such that for all $n \geq N(r)$,
\begin{equation} \la{eS3.1}
\bSym \mathrm{Tr}_{n}(R)_{\bullet}\,:\, \bSym[\mathcal{C}(R)](r) \to R_{n}^{\GL}(r)
 \end{equation}
 is an isomorphism of $k$-vector spaces. Since the map~\eqref{eS3.1} is a surjection by Theorem~\ref{pS1.3.1},
it suffices to show that for any fixed $r$, the map~\eqref{eS3.1} is an injection for $n$ sufficiently large.

\subsubsection{}\la{3.3.3} In order to do this we modify the proof of \cite[Proposition 11.1.2]{CEG}. To avoid unnecessarily complicated notation,
we demonstrate our proof for $R=k\langle x, y\rangle$ with the homological degree of $x$ (resp., $y$) being $0$ (resp, $1$)
and the polynomial degree of $x$ and $y$ being $1$ (and $dy=0$). The proof in the general case is completely similar. In this case,
$\mathcal{C}(R)$ is spanned by the collection $S$ of nonempty cyclic words $w$ in $x$ and such that $\mathrm{N}(w)\,\neq\,0$ (see Section~\ref{SI.0}). The elements of $\bSym[\mathcal{C}(R)] $ are precisely
expressions of the form
$$\sum_J c_J. \prod_{w \in S} (w(x,y))^{J(w)}$$
where $J$ runs over some finite collection of $\N \cup \{0\}$-valued functions on $S$ with finite support satisfying $J(w) \leq 1$ if $w$ is of odd parity (recall that $x$ is even and $y$ is odd).

 Let $\mathcal A:=k[u_1,\ldots ,u_r]$ where the elements $u_i$ are of homological degree $1$. Since any pair $X \in \M_n(k),\,\,Y \in \M_n(\mathcal A_1)$ corresponds to an evaluation homomorphism (of graded commutative algebras) from $R_n$ to $\mathcal A$, Lemma~\ref{lSI.3} implies that if $ \alpha \,\in\, \bSym[\mathcal{C}(R)](r) $ satisfies $\bSym \mathrm{Tr}_{n}(R)_{\bullet} (\alpha)=0 $
for all $n \in \mathbb N$, then $ \alpha = 0 $. Hence, ~\eqref{eS3.1} is an injection for $n$ sufficiently large. This completes the proof of Proposition~\ref{pS3.2}.

\subsection{Derived representation functors}
\la{3.4}

The following proposition collects facts about the existence of total left derived functors of various representation
(and related) functors from $\Ho(\bDGA_k)$.
Note that the cyclic functor $\mathcal{C}$ is treated here as a functor from $\bDGA_k$ to
the (model) category of complexes of graded $k$-vector spaces. The analogue of Theorem~\ref{ftt} holds in this more refined setting.

\begin{prop} \la{pS3.1.2}
$(a)$  If $F$ is one of the functors $\,(\mbox{--})_n\,$, $\,(\mbox{--})_n^{\GL}\,$ and $\bSym[\mathcal{C}(\,\mbox{--}\,)] $, then $F$ has a total left derived
functor $$ \L F\,:\,\Ho(\bDGA_k) \rar \Ho(\bcDGA_k)\,,\,\,\,\,\,\,\, A \mapsto F(QA)\ . $$
Moreover, we have $\, \L \bSym[\mathcal{C}(\,\mbox{--}\,)]
\cong \bSym[\L  \mathcal{C}(\,\mbox{--}\,)] \,$.

$(b)$ If $F$ is one of the functors $\,(\mbox{--})_{{\infty}}\,$, $\,(\mbox{--})_{{\infty}}^{\GL} $ and $\mathrm{Tr}_{{\infty}}$, then $F$ has a total left derived
functor $$ \L F\,:\,\Ho(\bDGA_k) \rar \Ho(\bcDGA_k)\,,\,\,\,\,\,\,\, A \mapsto F(QA)\ . $$
\end{prop}

\begin{proof}
 Note that the forgetful functors  $\bDGA_k \rar \DGA_{k/k}$
and $\bcDGA_k \rar \cDGA_k$ preserve cofibrations as well as weak equivalences. Hence, any functor $F: \DGA_{k/k} \rar \cDGA_k$
that maps acyclic cofibrations between cofibrant objects in $\DGA_k$ to weak equivalences in $\cDGA_k$ and
that enriches to a functor $\bDGA_k \rar \bcDGA_k$
maps acyclic cofibrations between cofibrant objects in $\bDGA_k$ to weak equivalences in $\bcDGA_k$. By Brown's Lemma,
such a functor $F$ has a total left derived functor $$ \L F\,:\,\Ho(\bDGA_k) \rar \Ho(\bcDGA_k)\,,\,\,\,\,\,\,\, A \mapsto F(QA) $$
where $QA$ is any cofibrant resolution of an object $A$ in $\bDGA_k$. By the proofs of 
Theorems 2.6 and 3.1 of~\cite{BKR}, this is the case when
$F$ is one of the functors $(\mbox{--})_n, (\mbox{--})_n^{\GL}\,,\,\,n \in \N$ and $\bSym[\mathcal{C}(\mbox{--})] $. That $\L  \bSym[\mathcal{C}(\mbox{--})]
\cong \bSym[\L  \mathcal{C}(\mbox{--})] $ follows from an easy verification left to the reader. This proves $(a)$.

The proof of part $(b)$ is completely analogous to that of Theorem~\ref{tS2.2}. We leave the details to the interested reader
in order to avoid being repetitive.
\end{proof}
\begin{theorem} \la{tS3.40}
$ \bSym \mathrm{Tr}_{\infty}(\mbox{--})_{\bullet} $ induces an isomorphism of functors from $ \Ho(\bDGA_k) $ to $ \Ho(\bcDGA_k) \,$:
$$
\L  \bSym[\mathcal{C}(\mbox{--})]\, \stackrel{\sim}{\to}\,
\L  (\,\mbox{--}\,)_{\infty}^{\mathrm{Tr}}\ .
$$
\end{theorem}
\begin{proof}
The proof of Theorem~\ref{tS2.3} goes through in the bigraded setting.
\end{proof}

\begin{theorem}
\la{tS3.4}
Suppose that $A \in \bDGA_{k/k}$ has a noncommutative Tate resolution with finitely many generators in each polynomial degree. Then
\begin{enumerate}
\item[$(a)$]
$\L  (A)_{\infty}^{\mathrm{Tr}} \cong \L  (A)_{{\infty}}^{\GL}\,$   in $\,\Ho(\bcDGA_k)\,$.
Consequently, $\, \bSym[\overline{\mathrm{HC}}_{\bullet}(A)] \cong \H_\bullet(A, \infty)^{\GL} $.
\item[$(b)$] For any $r \in \mathbb N$, there exists a number $N(r) \in \mathbb N$ such that for any $n \geq N(r)$,
$$ \bSym \mathrm{Tr}_{n}(A)_{\bullet}\,:\, \bSym[\overline{\mathrm{HC}}_{\bullet}(A)](s) \rar \mathrm{H}_{\bullet}(A,n)^{\GL}(s)$$
is an isomorphism for $s \leq r$. Hence, if $\,n \geq N(r)\,$ then
$$\mathrm{H}_{\bullet}(A,n)^{\GL}(s) \cong \mathrm{H}_{\bullet}(A,\infty)^{\GL}(s)
\ , \ s \leq r \,.
 $$
\end{enumerate}
\end{theorem}

\begin{proof}
Part $(a)$ is immediate from Proposition~\ref{pS3.2}. Moreover, for $R$ a noncommutative Tate resolution of $A$ (with finitely many generators
in each polynomial degree), we showed in Section~\ref{3.3}
that the map~\eqref{eS3.1} is an isomorphism for $n \gg 0 $. At the level of homology this implies $(b)$.
\end{proof}

\subsubsection{Remarks}
\la{3.4.1}
1. Part $(b)$ of Theorem~\ref{tS3.4} is a generalization of \cite[Proposition 11.1]{CEG}, which
proves stabilization for the family of ordinary representation schemes $ \Rep_n(A) $ of a graded algebra $A$.

2. The assumption of Theorem~\ref{tS3.4} is far from vacuous. Indeed, following the proof of Proposition 3.1 of~\cite{FHT}, it can be shown that the quotient of a finitely generated free (polynomially graded) algebra generated by elements of positive polynomial degree by a homogeneous ideal whose generators are of positive polynomial degree has the required kind of noncommutative Tate resolution.

3. Theorem~\ref{tS3.4} is analogous to Theorem~6.9 of \cite{LQ}. However, we are constrained to work with
a specific subcategory of the category of bigraded DGAs for this. Also, unlike \cite[Theorem~6.9]{LQ}, two natural questions remain open in this setting:
we do not yet have a formula expressing $N(r)$ in terms of $r$ and we do not yet have an explicit description of the first obstruction to stability.
One reason for this is the lack of an explicit `small' complex intrinsic to $A$ for computing representation homology of $A$.

4. In the case of augmented algebras, $ R_\infty^\Tr $ is only a proper (dense)
DG subalgebra of $\,R_{{\infty}}^{\GL}\,$: i.e., $ R_\infty^\Tr  \not=  R_{{\infty}}^{\GL} $. However, Proposition~\ref{pS3.2} shows that the functor $\L  (\mbox{--})_{\infty}^{\mathrm{Tr}}\,:\,\Ho(\DGA_{k/k}) \rar \Ho(\cDGA_k)$ may be viewed as a natural analogue of $\,\L (\mbox{--})_{{\infty}}^{\GL}\,:\,\Ho(\bDGA_k) \rar \Ho(\bcDGA_k)$ in that case.

\subsection{Euler characteristics and combinatorial identities}
\la{comb}
In this section, we compute the Euler characteristic of the $\GL$-invariant part of representation homology $ \H_\bullet(A, n)^\GL $ as $ n \to \infty $. Comparing the result of this computation to cyclic homology of $A$, we get (as a consequence of Theorem~\ref{tS3.4}) some interesting combinatorial identities. For illustration, we consider two families of finite-dimensional
nilpotent algebras $\, A = \c[x]/(x^{m+1})\,$, $\, m\ge 1 \,$, and
$\,A= \c[x_1, x_2, \ldots, x_d]/(x_1, x_2, \ldots, x_d)^2\,$, $\, d \ge 1 \,$. As mentioned in the Introduction, our choice of examples is motivated by the Lie homological interpretation of the famous
Macdonald conjectures (see \cite{Ha, FGT}).

Let $ A $ be a bigraded DG algebra satisfying the assumption of Theorem~\ref{tS3.4}.
If $ R $ is a Tate resolution of $A$ in $ \bDGA_{k/k} $, the underlying graded algebra
of $ R $ can be written as $ T_k V $, where $\, V = V_{\rm ev} \oplus V_{\rm odd} \,$
is a $k$-vector superspace, each component of which is polynomially graded:
$$
V_{\rm ev} = \bigoplus_{i \ge 0} V_{\rm ev}(i)\quad ,\quad
V_{\rm odd} = \bigoplus_{i \ge 0} V_{\rm odd}(i)\ .
$$
By our assumption, the dimensions of $  V_{\rm ev}(i) $ and $  V_{\rm odd}(i) $
are finite for all $ i \,$; we set
\begin{equation}
\la{dcoef}
d_i := \dim_k[V_{\rm ev}(i)]\,-\, \dim_k[V_{\rm odd}(i)]\ ,\quad i \ge 0 \ .
\end{equation}
Now, fix $ n \in \N $ and let
$$
\mathbb{V}_{{\rm ev}, n} :=  V_{\rm ev} \otimes \M_n(k)\quad ,\quad
\mathbb{V}_{{\rm odd}, n} :=  V_{\rm odd} \otimes \M_n(k)\ .
$$
Note that $\, \dim_k[\mathbb{V}_{{\rm ev}, n}(i)] -
\dim_k[\mathbb{V}_{{\rm odd}, n}(i)] =  d_i \, n^2 \,$.
By Theorem~\ref{comp}, the DG algebra $ R_n $ representing $ \DRep_n(A) $
is isomorphic (as a graded algebra) to the symmetric algebra of
$ \mathbb{V}_n = \mathbb{V}_{{\rm ev}, n} \oplus \mathbb{V}_{{\rm odd}, n}\,$:
$$
R_n \cong \bSym(\mathbb{V}_{n}) =
\Sym_k(\mathbb{V}_{{\rm ev}, n}) \otimes \Lambda_k(\mathbb{V}_{{\rm odd}, n})\ .
$$
Hence, the (graded) Euler characteristic of representation homology
$$
\chi(\H_\bullet(A,n), q) := \sum_{i, s = 0}^{\infty}\, (-1)^i \, q^s \dim_k[\H_i(A,n)(s)]
$$
is given by the formula ({\it cf.} \cite[2.1.2]{GS})
\begin{equation}
\la{eul}
\chi(\H_\bullet(A,n), q) = \prod_{i = 1}^{\infty}\, (1 - q^i)^{-d_i n^2}\ .
\end{equation}
Now, let us assume that $ k = \c $. Then, we can write the Euler characteristic of the $ \GL_n$-invariant part of representation homology using the classical Molien-Weyl formula (see \cite{W1}):
\begin{equation}
\la{eugl}
\chi(\H_\bullet(A,n)^\GL, \,q) = \int_{U(n)}\
\prod_{i = 1}^{\infty}\,\frac{d \mu(g)}{\det(I - \Ad(g)\,q^{i})^{d_i}}\ .
\end{equation}
The integral in \eqref{eugl} is taken over the unitary group $ U(n) \subset \GL_n(\c) $ on which the Haar measure $ d \mu $ is normalized so that the volume of $\, U(n) $ equals $1$; the determinant is calculated in the adjoint representation of $ \GL_n(\c) $ on $ \M_n(\c) $. Both
\eqref{eul} and \eqref{eugl} are regarded as formal power series in $ \c[[q]] $.

The following observation is due to Etingof and Ginzburg \cite{EG}, who proved it in the special case of Koszul complexes. The proof of \cite{EG} works, {\it mutatis mutandis}, for
an arbitrary free bigraded algebra.
\blemma[{\it cf.} \cite{EG}, Corollary~2.4.5]
\la{EGL}
Each coefficient of the formal power series \eqref{eugl} stabilizes
as $ n \to \infty $. Let
$$
\zeta(A, q) := \lim_{n \to \infty} \, \int_{U(n)}\
\prod_{i = 1}^{\infty} \,\frac{d \mu(g)}{\det(I - \Ad(g)\,q^{i})^{d_i}}
$$
denote the corresponding generating function in $ \c[[q]] $. Then
\begin{equation}
\la{zeta}
\zeta(A, q) = \prod_{s = 1}^{\infty}\,
\biggl(1 - \sum_{i = 1}^{\infty}\, d_i \,q^{si} \biggr)^{-1} \ ,
\end{equation}
where $ d_i $ are defined in \eqref{dcoef}.
\elemma

Next, we look at the (reduced) cyclic homology of $A$. If $A$ is bigraded, then so is $\,
\rHC_\bullet(A) \,$. Assume that the dimensions of the graded components of $ \rHC_{\rm ev}(A) $ and $ \rHC_{\rm odd}(A) $ are $ \{a_i\}_{i \ge 0} $ and $ \{b_i\}_{ i \ge 0} $ respectively. Then
\begin{equation}\la{euhc}
\chi(\bSym[\rHC_\bullet(A)],\, q) = \prod_{i = 1}^{\infty}\, (1 - q^i)^{b_i - a_i}
\end{equation}
As consequence of Theorem~\ref{tS3.4} and Lemma~\ref{EGL}, we thus get
\begin{corollary}\la{combin}
Let $A$ be a bigraded DG algebra over $ \c $ satisfing the assumption of Theorem~\ref{tS3.4}. Then
\begin{equation}
\la{cid}
\chi(\bSym[\rHC_\bullet(A)], q) = \zeta(A, q)\ .
\end{equation}
%
%
%
\end{corollary}

We conclude with simple examples showing that \eqref{cid} actually yields interesting
combinatorial identities even for ordinary finite-dimensional algebras.

\subsubsection{Examples}
$(a)\ $ Let $ A = \c[x]/(x^2) $ be the ring of dual numbers. We equip $A$ with the usual polynomial grading letting $ \deg(x) = 1 $. The cyclic homology of $A$ is given by (see \cite[4.3]{LQ})
\begin{equation*}
\rHC_i(A) =
\left\{
\begin{array}{lll}
0 \ & \mbox{\rm if} &\ i = 2j+1 \\*[1ex]
\c\, x^{\otimes (2j+1)}\ & \mbox{\rm if} &\ i = 2j
\end{array}
\right.
\end{equation*}
Hence,
$$
\chi(\bSym[\rHC_\bullet(A)], q) = \prod_{j = 0}^{\infty} \,(1 - q^{2j+1})^{-1} \ .
$$
On the other hand, $A$ has a free resolution
$ R = \c\langle x, x_1, x_2, x_3, \ldots \rangle $
with differential \eqref{dualnum}. The generators $ x_i $ of $ R $ have homological degree
$ i $ and polynomial degree $ i+1 $, so that $ d_i = (-1)^{i+1} $ for each $ i \ge  1 $. Hence,
$$
\zeta(A,q) = \prod_{s = 1}^{\infty}\,
\biggl(1 - \sum_{i = 1}^{\infty}\, (-1)^{i+1}\,q^{si} \biggr)^{-1} =
 \prod_{s = 1}^{\infty}\,(1 + q^s)\ .
$$
Thus, in this case, \eqref{cid} becomes
\begin{equation}
\la{cid1}
\prod_{j = 0}^{\infty} \,(1 - q^{2j+1})^{-1}\,=\, \prod_{s = 1}^{\infty}\,(1 + q^s)\ ,
\end{equation}
which is a well-known combinatorial identity equating the generating
functions for partitions into odd summands (the left-hand side) and
distinct summands (the right-hand side).

$(b)$ Let $ A = \c[x]/(x^{m+1}) $, where $ m \ge 1 $.
The (graded) cyclic homology of $ A $ is computed in
\cite[Theorem~3.1]{H}. The corresponding Euler characteristic is given by
$$
\chi(\bSym[\rHC_\bullet(A)], q) = \prod_{{n \ge 1 \atop\, (m+1) \,\nmid\, n}}
(1 - q^n)^{-1} \ .
$$
Since $A$ is an algebra with `monomial' relations, to construct its free resolution $ R $ one can use the Gr\"obner basis algorithm developed in \cite{DK}. The corresponding $\zeta$-function has the following combinatorial description. For $ l, m \in \N \,$, define an {\it $m$-train of length $l$} to be an increasing sequence of natural numbers
$$
\tau = \{n_1 = 1 < n_2 < n_3 < \ldots < n_l\}
$$
such that $\,n_{j+1} - n_{j} < m+1 \,$ for all $\,j\,$. Call
$ w_{\tau} := m + n_l $ the {\it weight} of $ \tau $ and
write $ T(l,m) $ for the set of all $m$-trains of length $l$.
Then
\begin{equation}
\la{zeta1}
\zeta(A, q) = \prod_{s = 1}^{\infty}\,\biggl(1 - q^s +
\sum_{l = 1}^{\infty} \ \sum_{\tau \in T(l,m)} (-1)^{l-1} \, q^{s \, w_\tau}
\biggr)^{-1}\ .
\end{equation}
Notice that each $ \tau \in T(l,m) $ corresponds to a sequence
$\,\{\alpha_1,\,\alpha_2,\,\ldots ,\,\alpha_l\} \,$ of positive integers $\, < m+1 \,$,
such that $\, w_{\tau} = \sum_{i=1}^l \alpha_i + m + 1\,$. (Indeed, put
$\, \alpha_i := n_{i+1} - n_i $). Hence,
$$
\sum_{\tau \in T(l,m)} \,q^{w_\tau} = q^{m+1}(q + q^2 + \ldots + q^m)^{l-1}\ .
$$
Substituting this into \eqref{zeta1}, we get
\begin{eqnarray*}
\zeta(A, q) &=&
\prod_{s = 1}^{\infty}\,
\biggl(1 - q^s +  q^{(m+1)s} \sum_{l = 1}^{\infty} \ (-1)^{l-1}(q^s + q^{2s} + \ldots + q^{ms})^{l-1}
\biggr)^{-1} \\
&=&
\prod_{s = 1}^{\infty}\, (1 - q^s + q^{(m+1)s}/(1 + q^s + \ldots + q^{ms}))^{-1} \\
&=&
\prod_{s = 1}^{\infty}\, (1 + q^s + q^{2s}+ \ldots + q^{ms}) \ .
\end{eqnarray*}
Thus, for  $ A = \c[x]/(x^{m+1}) $, equation \eqref{cid} yields the identity
\begin{equation}
\la{cid2}
\prod_{{n \ge 1 \atop\, (m+1) \,\nmid\, n}}
(1 - q^n)^{-1} \ =\
\prod_{s = 1}^{\infty}\, (1+ q^s+ q^{2s}+ \ldots +q^{ms})\ ,
\end{equation}
which obviously specializes to \eqref{cid1} when $ m = 1 $.

$(c)$ Let $\, A = \c[x_1, x_2, \ldots, x_d]/(x_1, x_2, \ldots, x_d)^2\,$, $\, d \ge 1 $.
By \cite[Proposition~2.2.14]{L}, we have
$$
\rHC_\bullet(A) = \overline{T(V[1])}_\n [-1] \  ,
$$
where $\, V = \oplus_{i=1}^d\, k x_i \,$. For a basis in $ \overline{T(V[1])}_\n $, we can
take a set $ S $ of ``good cyclic words'' in letters $\, \{s x_1, \ldots, s x_d\} \,$
(see Appendix~\ref{SI}, Lemma~\ref{lSI.0.1} below). Let $\, c_{r_1, r_2, \ldots, r_d} \,$
denote the number of good cyclic words with $ r_i $ occurences of $ s x_i $ for
$\, i = 1,2,\ldots, d\,$. Then,
$$
\chi(\bSym[\rHC_\bullet(A)], q_1, \ldots, q_d) =
\prod_{r_1, r_2, \ldots, r_d \ge 1}
(1 \,-\, q_1^{r_1}\, q_2^{r_2} \,\ldots\, q_d^{r_d})^{(-1)^{\sum\,r_i}\, c_{r_1, r_2, \ldots, r_d}}
$$
On the other hand, the (multigraded) $\zeta$-function of $ A $ is given by
$$
\zeta(A,q_1, \ldots, q_d) = \prod_{s=1}^\infty(1 + q_1^s + \ldots + q_d^s)\ .
$$
Thus, \eqref{cid} yields the identity
\begin{equation}\la{cidd}
\prod_{r_1, r_2, \ldots, r_d \ge 1}
(1 \,-\, q_1^{r_1} \,q_2^{r_2}\, \ldots \, q_d^{r_d})^{(-1)^{\sum\,r_i}\, c_{r_1, r_2, \ldots, r_d}}
= \prod_{s=1}^\infty(1+ q_1^s + \ldots + q_d^s) \ .
\end{equation}
Note that when $ q_1 = q_2 = \ldots = q_d = q $, we get
\begin{equation*}
\prod_{r = 1}^\infty (1 - q^{r})^{(-1)^r\, c_r} = \prod_{s=1}^\infty(1+ d\,q^s) \ ,
\end{equation*}
where $ c_r := \dim_k\, V[1]^{\otimes r}\, \cap \,T(V[1])_\n\,$. The left-hand side of
this last identity can be expressed in combinatorial terms:
\begin{equation}
\la{cidd1}
\prod_{r=1}^\infty (1 - q^r)^{(-1)^r \Phi_r(d)}
\prod_{j,k \geq 1} (1 - q^{2k(2j-1)})^{-M_{2j-1}(d)} = \prod_{s=1}^\infty(1+ d\,q^s) \ ,
\end{equation}
where $ \Phi_r(d) $ and $ M_r(d) $ are the numbers of cyclic (resp.,
primitive\footnote{Recall that a cyclic word $ u $ is primitive if it is not of the form $ z^n $
for some cyclic word $z$.} cyclic) words of length $r$ in $\{sx_1, sx_2, \ldots, sx_d \} $. The generating functions for these numbers are known ({\it cf.} \cite[Section~3.3]{GS}):
\begin{equation}
\la{pol}
\sum_{r \ge 1} \Phi_r(d)\,t^r = -\sum_{m=1}^{\infty}\, \frac{\phi(m)}{m}\, \log(1 - d t^m)\ ,
\end{equation}
\begin{equation}
\la{mob}
\sum_{r\ge 1} M_r(d)\, t^r = -\sum_{n=1}^{\infty}\, \frac{\mu(n)}{n}\,\log(1-dt^n)\ ,
\end{equation}
where $ \phi(m) $ is the Euler $ \phi$-function and $ \mu(n) $ is the M\"obius function.
Equation~\eqref{pol} is sometimes called the Polyakov formula, while \eqref{mob} is a classical
identity due to E.~Witt (1937).

%
%

\appendix
\section{(Super) Cyclic Derivatives}
\la{SI}

In this appendix, we extend some combinatorial results about cyclic derivatives on noncommutative algebras (see \cite{RSS, K, Vo}) to the graded and $\Z_2$-graded setting. Although the proofs of these results are not very hard and could in principle be included as part of the proofs our main theorems, we decided to present them separately. We feel that graded cyclic derivatives may be useful to others working
on questions related to combinatorics and invariant theory in the $\Z$- and $\Z_2$-graded setting.

\subsubsection{} \la{SI.0}

Consider the graded algebra $\A := k\langle x, y_1,\ldots,y_r \rangle$, where the generators $x,y_1,\ldots,y_r$ are either in homological degree $0$ or in homological degree $1$.
Recall that the term {\it word} simply refers to a (noncommutative) monomial in the variables $x,y_1,\ldots,y_r$. In turn, the variables $x,y_1,\ldots,y_r$ are referred to as {\it alphabets}.
Note that $\A$ is in fact, bigraded (with $x$ and $y_1,\ldots,y_r$ each having polynomial degree $1$).
Generalizing~\cite{RSS} to the (homologically) graded setting, we define the cyclic derivative $\frac{\partial w}{\partial x}$ of a word $w := v_1\ldots v_n$ with respect to $x$ by setting
$$\frac{\partial w}{\partial x} := \sum_{v_k =x} {(-1)}^{|v_1\ldots v_{k-1}||v_k\ldots v_n|} v_{k+1}\ldots v_n v_1\ldots v_{k-1} \,\mathrm{.}$$
The cyclic derivative with respect to $x$ gives us a $k$-linear operator $\frac{\partial}{\partial x}\,:\,\A \rar \A$.

More conceptually, equip $\A \otimes_k \A$ with the structure of a $\A-\A$ bimodule where $a(b \otimes c):=ab \otimes c,\,\,
(b \otimes c).d:=b \otimes cd$ for all $a,b,c,d \,\in\,\A$. Let $\delta_x\,:\,\A \rar \A \otimes_k \A$ be the derivation
(of degree $|x|$ with respect to the above bimodule structure) such that $\delta_x(x)=1 \otimes 1$ and $\delta_x(y_i)=0$ for $1 \leq i \leq r$.
If $\tilde{\mu}\,:\, \A \otimes \A \rar \A$ is the map $a \otimes b \mapsto (-1)^{|a||b|} ba$, then
$$ \frac{\partial}{\partial x} \,=\, \tilde{\mu} \circ \delta_x \text{.}$$
In this appendix, we  prove the following graded analog of Proposition 6.1 of~\cite{RSS}.

\begin{prop} \la{pSI.0.1}
$\Ker(\frac{\partial}{\partial x})\,=\, k\langle y_1,\ldots ,y_r \rangle +[\A,\A] \,$.
\end{prop}

One can define the cyclic gradient $$\triangledown\,:\, \A \rar\A^{r+1} \,,\,\,\,\,\,  a \mapsto \left(\frac{\partial a}{\partial x},\,\frac{\partial a}{\partial y_1},\,\ldots,\,\frac{\partial a}{\partial y_r}\right) \text{ .}$$
Proposition~\ref{pSI.0.1} automatically implies the following corollary, a super version of a result of Voiculescu (see \cite[Theorem 2]{Vo}).

\begin{corollary}
$\Ker(\triangledown)\,=\, k +[\A,\A]$.

\end{corollary}
 The proof of Proposition~\ref{pSI.0.1} is a modification of that of Proposition 6.1 of~\cite{RSS}. In order to write down this proof, we need to undertake some preparations.
Let $\tau\,:\,\A \rar \A$ denote the operator such that for each word $w:=v_1\,\ldots\,v_n$ in
$\A$,
$$\tau(w) \,:=\,(-1)^{|v_n||v_1\ldots v_{n-1}|} v_nv_1\ldots v_{n-1} \mathrm{.}$$
Here $|\ldots|$ stands for homological degree.
Let $\A^{(n)}$ denote the (homologically graded) subspace of $\A$ spanned by the words of polynomial degree $n$. Note that $\tau^n|_{\A^{(n)}} = \id_{\A^{(n)}}$. We define the operator $\mathrm{N}:=1+\tau + \ldots +\tau^{n-1}\,:\,\A^{(n)} \rar \A^{(n)}$ for each $n \in \N$ (with $\mathrm{N}(1):=1$). This gives us an operator $\mathrm{N}\,:\,\A \rar \A$. One also has the operator $\mathrm{T}_x\,:\,\A \rar \A$. This is the unique $k$-linear operator on $\A$ satisfying
$$
\mathrm{T}_x(xw)=w,\,\,\,\,\, \mathrm{T}_x(y_i w)=0,\,\,\,\,\forall\,\,1 \leq i \leq k \ , 
$$
for all words $w$ in $x,y_1,\ldots ,y_r$. By definition,
$$\frac{\partial}{\partial x} \,=\, \mathrm{T}_x \circ \mathrm{N} \,:\,\A \rar \A \,\mathrm{.}$$
\begin{lemma} \la{lSI.0.1}
$\Ker(\mathrm{N})\,=\,[\A,\A]$.
\end{lemma}

\begin{proof}
Clearly, $[\A,\A] \,\subset \,\Ker(\mathrm{N})$. Indeed, if $w_1$ and $w_2$ are words in $\A$, $[w_1,w_2]\,=\, (1-\tau^{l(w_2)})w$ where $w:=w_1w_2$ and $l(w_2)$ denotes the polynomial degree of $w_2$. Since $ \mathrm{N}\, \tau^r =\mathrm{N}$ for any $r$, $\, [w_1,w_2] \,\in\,\Ker(\mathrm{N})$.

To prove the reverse inclusion, we modify the proof of Proposition 6.1 of~\cite{RSS}. Order the alphabets $x,y_1,\ldots,y_r$. This induces a lexicographic order on the set of words in $\A$. Let $l(w)$ denote the polynomial degree (i.e, length) of a word $w$. For a word $w \in \A$, let $w^*$ denote the cyclic permutation of $w$ that is lexicographically minimum. Note that for at least one $0 \leq p < l(w)$, $\tau^p w=\pm w^*$. Now, suppose that the coefficient of $w^*$ in $\mathrm{N}(w)$ is $0$. Then, $\pm w^*= \tau^p w \,=\,-\tau^q(w)$ for some $0 \leq p < q <l(w)$. Hence, $w = -\tau^{q-p}(w)$. In particular, $w = \frac{1}{2}(w -\tau^{q-p}(w)) $. But $w -\tau^{q-p}w \in [\A,\A]$. Hence, if the coefficient of $w^*$ in $\mathrm{N}(w)$ is $0$, then $w \in [\A,\A]$ and $\mathrm{N}(w)=0$\footnote{This is the main
difference between our situation and the ungraded situation. For example, if $\A = k \langle x,y \rangle$ with $|x|=0,\,|y|=1$, then $\mathrm{N}(x^2yx^2y)=0$ and $x^2yx^2y= \frac{1}{2}[x^2y,x^2y]$.}. We refer to words of $\A$ that are in $[\A,\A]$ as ``bad''. Words that are not ``bad'' are called ``good''.

  Therefore it suffices to check that if $\mathrm{P}$ is a linear combination of good words in $\A$ such that $\mathrm{N}(\mathrm{P})=0$, then $\mathrm{P} \in [\A,\A]$. Then,
  $$\mathrm{P}\,=\,\sum_{i=1}^{n_1} c_iw_i + \sum_{i=n_1+1}^{n_2} c_iw_i + \ldots +\sum_{i=n_{s-1}+1}^{n_s} c_iw_i$$
  where $w_i$ and $w_j$ occur in the same summation if $w_i^*= w_j^*$.
  Hence,
  $$0= \mathrm{N}(\mathrm{P})= \sum_{i=1}^{n_1} c_i\mathrm{N}(w_i) + \sum_{i=n_1+1}^{n_2} c_i\mathrm{N}(w_i) + \ldots +\sum_{i=n_{s-1}+1}^{n_s} c_i\mathrm{N}(w_i)$$
  $$ = \sum_{i=1}^{n_1} \pm c_i\mathrm{N}(w_1^*) + \sum_{i=n_1+1}^{n_2} \pm c_i\mathrm{N}(w_{n_1+1}^*) + \ldots +\sum_{i=n_{s-1}+1}^{n_s} \pm c_i\mathrm{N}(w_{n_{s-1}+1}^*)\mathrm{.}$$
  This implies that $\sum_{i=1}^{n_1} \pm c_i = \sum_{i=n_1+1}^{n_2} \pm c_i = \ldots =
  \sum_{i=n_{s-1}+1}^{n_s} \pm c_i = 0$. Hence,
  $$ \mathrm{P} = \sum_{i=1}^{n_1} c_i(w_i \pm w_1^*) + \sum_{i=n_1+1}^{n_2} c_i(w_i \pm w_{n_1+1}^*) + \ldots +\sum_{i=n_{s-1}+1}^{n_s} c_i(w_i \pm w_{n_{s-1}+1}^*)$$
  $$= \sum_{i=1}^{n_1} c_i(w_i \mp w_i^*) + \sum_{i=n_1+1}^{n_2} c_i(w_i \mp w_i^*) + \ldots +\sum_{i=n_{s-1}+1}^{n_s} c_i(w_i \mp w_i^*) \mathrm{.}$$
  The reader may easily keep track of the signs that appear above to check that $w_i \mp w_i^* \in [\A,\A]$ for all $i$. This proves the required lemma.
\end{proof}

Lemma~\ref{lSI.0.1} implies that a basis for $\A_{\natural}:=\frac{\A}{[\A,\A]}$ is given by cyclic words $w$ satisfying $\mathrm{N}(w) \,\neq\,0$. We denote this basis by $S$.

\subsubsection{Proof of Proposition~\ref{pSI.0.1}}
Clearly, $k \langle y_1,\ldots, y_r\rangle$ is contained in $\Ker(\frac{\partial}{\partial x})$. Note that 
$$
\A \,=\,k \langle y_1,\ldots,y_r\rangle \oplus \A_x \ ,
$$ 
where $\A_x$ is the (bigraded) subspace spanned by words with at least one occurrence of the alphabet $x$. As in Section 6 of~\cite{RSS}, we claim that if $\mathrm{P} \,\in\, \A_x$ satisfies $\frac{\partial \mathrm{P}}{\partial x}\,=\,0$, then $\mathrm{P} \,\in\,\Ker(N)$. To see this, first note that $\mathrm{N}(\mathrm{P}) \in \A_x$.
Further, suppose that $\mathrm{N}(\mathrm{P}) \,\neq\,0$. Then, since $\mathrm{N}(\mathrm{P}) \,\in\,\A_x$ and $\mathrm{N}(\mathrm{P})\,=\,\tau^k\mathrm{N}(\mathrm{P})$ for all $k$, $\mathrm{N}(\mathrm{P})$ contains at least one nonzero summand of the form $c.xw$ where $c \in k$ and $w$ is a word in $\A$. This makes $c.w$ a nonzero summand of $\mathrm{T}_x \circ \mathrm{N}(\mathrm{P}) \,=\,\frac{\partial \mathrm{P}}{\partial x}$. This shows that if $\mathrm{P} \,\in\, \A_x$, then $\frac{\partial \mathrm{P}}{\partial x}=0$ iff $\mathrm{N}(\mathrm{P})\,=\,0$. Proposition~\ref{pSI.0.1} now follows immediately from Lemma~\ref{lSI.0.1}.

\subsubsection{} \la{SI.2}

In this subsection and the one that follows, we work in the $\Z$-graded rather than $\Z_2$-graded setting. The results in this and the subsequent subsection are extensions to the $\Z$-graded setting of corresponding results from Section 11 of~\cite{CEG} and Section 4.2 of~\cite{Gi} (which were proven in {\it loc. cit} in the ungraded setting). Fix $r \in \N$. Let $\mathcal A:=k[u_1,\ldots,u_r]$ where $u_i$ is a variable of homological degree $1$ for each $i$.

\begin{lemma} \la{lS3.3.3} For $n$ sufficiently large, the set of words in $X$ and $Y$ of length $\leq r$
is linearly independent for generic $X \in \M_n(k)$ and $Y \in \M_n(\mathcal A_1)$.
\end{lemma}

\begin{proof} Suppose there is a nontrivial linear
dependence relation \begin{equation}  \la{eS3.5} \sum_{|w| \leq r} c_w. w(X,Y) = 0 \end{equation} for all
$n \in \N$ and for all $X \in \M_n(k), \, Y \in \M_n(\mathcal A_1)$. Applying the substitutions $X \mapsto t.X, Y \mapsto t.Y$,
we see that we can assume our relation to be homogeneous without loss of generality. Further, fixing $X$ and applying
the substitution $Y \mapsto s.Y$, we see that we can assume our relation to be homogeneous in $X$ as well as $Y$.
Suppose that our relation is homogeneous of degree $p$ in $X$ and degree $q$ in $Y$.
Putting $Y:=u_1.Z_1+\ldots+u_q.Z_q$ where $Z_1,\ldots,Z_q \in \M_n(k)$ are arbitrary, and taking the coefficient of $\,u_1 \ldots u_q\, $ in
equation~\eqref{eS3.5}, we obtain a linear dependence relation
$$ \sum_{|w'| \leq r} d_{w'}.w'(X,Z_1,\ldots,Z_q) =0 $$ satisfied by all $X,Z_1,\ldots,Z_q \in \M_n(k)$ for all $n$. However,
taking $n$ to be the dimension over $k$ of the quotient of $k\langle x, z_1,\ldots,z_r\rangle$ by the $r+1$-st power of the augmentation ideal,
and $X,Z_1,\ldots,Z_q$ to be the operators corresponding to multiplication on the left by $x,z_1,\ldots,z_q$ respectively, we see that
this is not possible. This proves the required lemma.
\end{proof}

Let $S$ denote the collection of nonempty {\it cyclic} words $w$ in the symbols $x$ and $y$ such that $\mathrm{N}(w) \neq 0$ (see Section~\ref{SI.0}).
Let $X$ be an even matrix variable and let $Y$ be an odd matrix variable. A polynomial function in $X$ and $Y$ of the form $$\sum_J c_J. \prod_{w \in S} \mathrm{Tr}(w(X,Y))^{J(w)}$$ where $J$ runs over some finite collection of functions from $S$ to $\mathbb N \cup \{0\}$
with finite support and $c_J$ are elements of $k$ is called {\it good} if $J(w) \leq 1$ for every $w$ of odd parity. The following lemma is a generalization (to the $\Z$-graded setting) of a similar statement (in the ungraded setting) that appears in Section 11 of~\cite{CEG}.

\begin{lemma}  \la{lSI.3} There are no non-trivial relations of polynomial degree $ \leq r $ in $X$ and $Y$ of the form
\begin{equation} \la{eS3.2} \sum_J c_J. \prod_{w \in S} \mathrm{Tr}(w(X,Y))^{J(w)} =0\,,\,\,\,\,\,\,\,\, \forall\,
n \in \mathbb N,\,\,X \in \M_n(k),\,\,Y \in \M_n(\mathcal A_1) \end{equation}
where the left-hand side is good.
\end{lemma}

\begin{proof} For a cyclic word $w(X,Y)=v_1\ldots.v_p$ of length $p$ in $x$ and $y$, let
$$\frac{\partial w}{\partial x}:= \sum_{v_k=X} {(-1)}^{|v_{k+1}\ldots v_p||v_1\ldots v_k|}v_{k+1}\ldots v_p v_1 \ldots v_{k-1}\,$$ denote the cyclic derivative of $w$ with respect to $x$. Suppose that there is a relation of the form~\eqref{eS3.2}
of polynomial degree $s\,\leq \,r$ where the left hand side is a good polynomial. If the relation~\eqref{eS3.2} has at least one summand involving $X$, we differentiate~\eqref{eS3.2} with respect to $X$
to obtain
\begin{equation} \la{eS3.3} \sum_J \sum_{u \in \mathrm{Supp}(J)\,,\,\,\, x \in u} \pm J(u)c_J . \prod_{w \in S}\mathrm{Tr}(w(X,Y))^{J(w)-\delta_{wu}}.\frac{\partial u}{\partial x}(X,Y)=0 \mathrm{.} \end{equation}

By Lemma~\ref{lS3.3.3} and Proposition~\ref{pSI.0.1}, the matrices of the form $\frac{\partial u}{\partial x}(X,Y)$ appearing on the right hand side of equation~\eqref{eS3.3} form a linearly independent set.
Taking the coefficient of any nonzero $\frac{\partial u}{\partial x}(X,Y)$ appearing on the R.H.S of~\eqref{eS3.3}, one obtains a good relation of the form~\eqref{eS3.2} of degree strictly less than $s$. If the relation~\eqref{eS3.2} does not involve $X$, we differentiate it with respect to $Y$ and argue as above to obtain another good relation of the form~\eqref{eS3.2} of degree strictly less than $s$. By induction,
we obtain a (nontrivial) good relation of the form~\eqref{eS3.2} of degree $1$ or $0$ satisfied for all $n$ sufficiently large.
Since this is impossible, the desired lemma follows.
\end{proof}

\subsubsection{} \la{SI.3}

Let $\mathcal A:=k[u_1,\ldots, u_n,\ldots]$, where the $u_i$'s have homological degree $1$. Lemma~\ref{lSI.3} immediately implies

\begin{lemma} \la{lSI.4}
There are no non-trivial relations in $X$ and $Y$ of the form
\begin{equation} \sum_J c_J . \prod_{w \in S} \mathrm{Tr}(w(X,Y))^{J(w)} =0\,,\,\,\,\,\,\,\,\,
\forall\,\,X \in \mathfrak{gl}_{\infty}(k),\,\,
Y \in \mathfrak{gl}_{\infty}(\mathcal A_1)
\end{equation}
where the left-hand side is good.
\end{lemma}

The following corollary generalizes a fact proven in the ungraded setting (using a polarization argument) in Section 4.2 of~\cite{Gi}.

\begin{corollary} \la{cSI.1}
For any $k$-valued function $c$ on $S$ with finite (non-empty) support, there exist
$X \in \mathfrak{gl}_{\infty}(k),\,\,Y \in \mathfrak{gl}_{\infty}(\mathcal A_1)$
such that
$$ 
\sum_{w \in S}\, c(w)\, \mathrm{Tr}(w(X,Y)) \,\neq \, 0 \mathrm{.}
$$
\end{corollary}


\begin{thebibliography}{G}
%
\bibitem[BKR]{BKR}
Yu.~Berest, G.~Khachatryan and A.~Ramadoss,
\textit{Derived representation schemes and cyclic homology},
Adv. Math. \textbf{245} (2013), 625--689.
%
\bibitem[BFR]{BFR}
Yu.~Berest, G.~Felder and A.~Ramadoss,
\textit{Derived representation schemes and noncommutative geometry}
in {\it Expository Lectures on Representation Theory}, Contemp. Math. 
\textbf{607}, Amer. Math. Soc., Providence, RI, 2014, pp. 113--162.
%
\bibitem[CK]{CK}
I. Ciocan-Fontanine and M. Kapranov,
\emph{Derived Quot schemes}, Ann. Sci. ENS \textbf{34} (2001), 403--440.
%
\bibitem[CEG]{CEG}
W.~Crawley-Boevey, P.~Etingof and V.~Ginzburg,
\textit{Noncommutative geometry and quiver algebras}, Adv. Math.
\textbf{209} (2007), 274--336.
%
\bibitem[DK]{DK}
V.~Dotsenko and A.~Khoroshkin, \textit{Free resolutons via Gr\"{o}bner bases},
{\tt arXiv:0912.4985}.
%
\bibitem[DS]{DS} W. G. Dwyer and J. Spalinski,
{\it Homotopy theories and model categories} in \textit{Handbook of Algebraic Topology},
Elsevier, 1995, pp. 73--126.
%
\bibitem[EG]{EG}
P.~Etingof and V.~Ginzburg,
\textit{Noncommutative complete intersections and matrix integrals},
Pure Appl. Math. Q. \textbf{3}(1) (2007), 107--151
%
\bibitem[FF]{FF}
B.~Feigin and D.~Fuchs, \textit{Cohomologies of Lie groups and Lie algebras},
Encyclopaedia Math. Sci. \textbf{21}, Springer, Berlin, 2000, pp. 125--223.
%
\bibitem[FT]{FT} B.~Feigin and B.~Tsygan, {\it Additive K-theory and
crystalline cohomology}, Funct. Anal. Appl. \textbf{19}(2) (1985), 124-–132.
%
\bibitem[FHT]{FHT}
Y.~Felix, S.~Halperin and J.-C.~Thomas, {\it Differential graded algebras in topology} in
{\it Handbook of Algebraic Topology}, Elsevier, 1995, pp. 829--865.
%
\bibitem[FGT]{FGT}
S. Fishel, I.~Grojnowski and C. Teleman, {\it The strong Macdonald conjectures and
Hodge theory on the loop Grassmannian}, Annals Math. \textbf{168} (2008), 175--220.
%
\bibitem[GM]{GM}
S.~Gelfand and Yu. Manin, \textit{Methods of Homological Algebra}, Springer, Berlin, 2000.
%
\bibitem[G]{Gi} V.~Ginzburg, \textit{Noncommutative symplectic geometry, quiver varieties and operads}, Math. Res. Lett. \textbf{8} (2001), 377--400.
%
\bibitem[GS]{GiS} V.~Ginzburg and T.~Schedler,
\textit{A new construction of cyclic homology}, \texttt{arXiv:1201.6635}
%
\bibitem[GoS]{GS} V.~Gorbounov and V.~Schechtman,
\textit{Homological algebra and divergent series}, SIGMA \textbf{5} (2009), 034, 31 pp.
%
\bibitem[Ha]{Ha}
P. Hanlon, \textit{Cyclic homology and Macdonald conjectures},
Invent. Math. \textbf{86} (1986), 131--159.
%
\bibitem[H]{H}
V. Hinich, \textit{Homological algebra of homotopy algebras},
Comm. Algebra \textbf{25} (1997), 3291--3323.
%
\bibitem[H1]{H1}
V. Hinich, \textit{DG coalgebras as formal stacks},
J. Pure Appl. Alg. \textbf{162} (2001), 209--250.
%
\bibitem[Hir]{Hir}
P. Hirschhorn, \textit{Model Categories and their Localizations},
Mathematical Surveys and Monographs \textbf{99}, Amer. Math. Soc., 2009.
%
\bibitem[J]{J}
R.~Jardine, \textit{A closed model structure for differential graded algebras} in
{\it Cyclic Cohomology and Noncommutative Geometry},
Fields Institute Communications \textbf{17} (1997), 55--58.
%
\bibitem[Ka]{Ka}
D.~Kaledin, \textit{Motivic structures in non-commutative geometry},
International Congress of Mathematicians, Vol. II,
Hindustan Book Agency, New Delhi, 2010, pp. 461--496.
%
\bibitem[Ke]{Ke}
B.~Keller, \textit{A-infinity algebras, modules and functor categories} 
in {\it Trends in Representation Theory of Algebras and Related Topics}, 
Contemp. Math. \textbf{406}, Amer. Math. Soc., Providence, RI, 2006, pp. 67--93.
%
\bibitem[K]{K} M. Kontsevich,
{\it Formal (non)commutative symplectic geometry}, The Gelfand Mathematical Seminars,
1990-1992, Birkh\"{a}user Boston, Boston, MA, 1993, pp. 173--187.
%
\bibitem[Kr]{Kr}
H.~Kraft, \textit{Geometrische Methoden in der Invariantentheorie},
Vieweg \& Sohn, Braunschweig, 1984.
%
\bibitem[KP]{KP} H.~Kraft and C.~Procesi,
\textit{Classical Invariant Theory}, book in preparation,
available at {\tt www.math.unibas.ch/$\sim$kraft/Papers/KP-Primer.pdf}
%
\bibitem[L]{L}
J.-L.~Loday, \textit{Cyclic Homology}, Grundl. Math. Wiss. \textbf{301}, 2nd Edition, Springer-Verlag, Berlin, 1998.
%
\bibitem[LQ]{LQ}
J-L. Loday and D. Quillen, \textit{Cyclic homology and the Lie algebra homology of matrices}, Comment. Math. Helvetici \textbf{59} (1984), 565--591.
%
\bibitem[LV]{LV}
J-L. Loday and B.~Vallette, \textit{Algebraic Operads},
Grundlehren der Mathematischen Wissenschaften \textbf{346},
Springer, Heidelberg, 2012.
%
\bibitem[M]{M}
H.~J.~Munkholm, \textit{DGA algebras as a Quillen model category. Relations to shm maps},
J. Pure Appl. Alg. \textbf{13} (1978), 221--232.
%
\bibitem[Pi]{Pi}
D.~I.~Piontkovski,
\textit{Graded algebras and their differentially graded extensions},
J. Math. Sci. (N.Y.) \textbf{142} (2007),  2267--2301.
%
\bibitem[Po]{Po}
L. Positselski,
\textit{Two Kinds of Derived Categories, Koszul Duality and
Comodule-Contramodule Correspondence}, Mem. Amer. Math. Soc.
\textbf{212} (2011), no. 996,  133 pp.
%
\bibitem[P]{P}
C.~Procesi, \textit{The invariant theory of $ n \times n$ matrices}, Adv. Math. \textbf{19}
(1976), 306--381.
%
\bibitem[Q1]{Q1}
D. Quillen, \textit{Homotopical Algebra},
Lecture Notes in Math. \textbf{43}, Springer-Verlag, Berlin, 1967.
%
\bibitem[Q2]{Q2}
D. Quillen, \textit{Rational homotopy theory}, Ann. Math. \textbf{90} (1969), 205--295.
%
\bibitem[RSS]{RSS}
G.-C.~Rota, B.~Sagan and P.~R.~Stein, \textit{A cyclic derivative in noncommutative algebra}, J. Algebra \textbf{64} (1980), 54--75.
%
\bibitem[TV]{TV}
B. To\"en and M.~Vaqui\'e, \textit{Moduli of objects in DG categories},
Ann. Sci. ENS \textbf{40} (2007), 387--444.
%
\bibitem[T]{T}
B.~Tsygan, \textit{Homology of matrix Lie algebras over rings and the Hochschild homology},
Uspekhi Mat. Nauk \textbf{38} (1983), 217--218.
%
\bibitem[T-TT]{T-TT} T. Thon-That and T-D. Tran,
\textit{Invariant theory of a class of infinite-dimensional groups},
J. Lie Theory \textbf{13} (2003), 401--425.
%
\bibitem
[VdB]{VdB} M. Van den Bergh, \textit{Noncommutative quasi-Hamiltonian spaces}, 
in {\it Poisson Geometry in Mathematics and Physics}, Contemp. Math. 
\textbf{450}, Amer. Math. Soc., Providence, RI, 2008, pp. 273--299.
%
\bibitem[Vo]{Vo}
D.~V.~Voiculescu, \textit{A note on the cyclic gradient}, Indiana Univ. Math. J.
\textbf{49} (3) (2000), 837--841.
%
\bibitem[W]{W}
C.~Weibel, \textit{An Introduction to Homological Algebra},
Cambridge Studies in Advanced Mathematics  \textbf{38}, 
Cambridge University Press, Cambridge, 1994.
%
\bibitem[We]{W1}
H. Weyl,
Zur Darstellungstheorie und Invariantenabz\"ahlung der projektiven,
der Kom\-p\-lex- und der Drehungsgruppe,
Acta Math. \textbf{48} (1926), 255--278.
Reprinted in ``Gesammelte Abhandlungen'', Band III,
Springer-Verlag, Berlin-Heidelberg-New York, 1968, 1--25.

\end{thebibliography}
\end{document}